\documentclass[reqno,12pt]{amsart}
\usepackage{}
\usepackage{amsfonts}
\usepackage{a4wide}

\usepackage{amsmath,amssymb,amsthm,amsfonts}
\usepackage{mathrsfs}
\usepackage{bbm}
\usepackage{hyperref}

\usepackage{color}

\newtheorem{lemma}{Lemma}[section]
\newtheorem{theorem}{Theorem}[section]

\newtheorem{proposition}{Proposition}[section]
\newtheorem{remark}{Remark}[section]

\numberwithin{equation}{section}

\newcommand{\dis}{\displaystyle}

\newcommand{\R}{\mathbb{R}}

\renewcommand{\S}{\mathbb{S}}
\newcommand{\T}{\mathbb{T}}

\newcommand{\FH}{\mathbf{H}}

\newcommand{\FP}{\mathbf{P}}

\newcommand{\FW}{\mathbf{W}}
\newcommand{\FX}{\mathbf{X}}
\newcommand{\FY}{\mathbf{Y}}

\newcommand{\Fb}{\mathbf{b}}
\newcommand{\Fk}{\mathbf{k}}

\newcommand{\CA}{\mathcal{A}}

\newcommand{\CE}{\mathcal{E}}
\newcommand{\CF}{\mathcal{F}}
\newcommand{\CG}{\mathcal{G}}
\newcommand{\CH}{\mathcal{H}}
\newcommand{\CI}{\mathcal{I}}
\newcommand{\CJ}{\mathcal{J}}
\newcommand{\CK}{\mathcal{K}}
\newcommand{\CL}{\mathcal{L}}
\newcommand{\CN}{\mathcal{N}}
\newcommand{\CM}{\mathcal{M}}

\newcommand{\CT}{\mathcal{T}}
\newcommand{\CW}{\mathcal{W}}

\newcommand{\CQ}{\mathcal{Q}}

\newcommand{\CV}{\mathcal{V}}

\newcommand{\na}{\nabla}

\newcommand{\al}{\alpha}
\newcommand{\bet}{\beta}
\newcommand{\ga}{\gamma}
\newcommand{\om}{\omega}
\newcommand{\Om}{\Omega}
\newcommand{\la}{\lambda}
\newcommand{\de}{\delta}
\newcommand{\si}{\sigma}
\newcommand{\pa}{\partial}
\newcommand{\ka}{\kappa}
\newcommand{\eps}{\epsilon}

\newcommand{\ta}{\theta}

\newcommand{\vps}{\varepsilon}

\newcommand{\Ga}{\Gamma}

\newcommand{\lag}{\langle}
\newcommand{\rag}{\rangle}

\newcommand{\eqdef}{\overset{\mbox{\tiny{def}}}{=}}

\begin{document}

\title[The Boltzmann equation for uniform shear flow]{The Boltzmann equation for uniform shear flow}

\author[R.-J. Duan]{Renjun Duan}
%\thanks{}
\address[RJD]{Department of Mathematics, The Chinese University of Hong Kong,
Shatin, Hong Kong, P.R.~China}
\email{rjduan@math.cuhk.edu.hk}

\author[S.-Q. Liu]{Shuangqian Liu}
\address[SQL]{School of Mathematics and Statistics, Central China Normal University, Wuhan 430079, and Department of Mathematics, Jinan University, Guangzhou 510632, P.R.~China}
\email{tsqliu@jnu.edu.cn}

%\keywords{}

\date{\today}
\maketitle

\begin{abstract}
The uniform shear flow for the rarefied gas is governed by the time-dependent spatially homogeneous Boltzmann equation with a linear shear force. The main feature of such flow is that the temperature may increase in time due to the shearing motion that induces viscous heat and the system becomes far from equilibrium. For Maxwell molecules, we establish the unique existence, regularity, shear-rate-dependent structure and non-negativity of self-similar profiles for any small shear rate. The non-negativity is justified through the large time asymptotic stability even in spatially inhomogeneous perturbation framework, and the exponential rates of convergence are also obtained with the size proportional to the second order shear rate. The analysis supports the numerical result that the self-similar profile admits an algebraic high-velocity tail that is the key difficulty to overcome in the proof.  
\end{abstract}

%\thanks{}
%\maketitle
%\begin{abstract}
%\Red{To be added.}
%
%\end{abstract}

\setcounter{tocdepth}{1}
\tableofcontents
\thispagestyle{empty}

%\newpage
\section{Intoduction}
%\subsection{The problem}

\subsection{Brief background}
In the paper we are concerned with the {\it uniform shear flow} (USF for short) described by the Boltzmann equation in the specific case of Maxwell molecules for which particles interacts via the exact inverse power law repulsive potential $U(r)=r^{-4}$ (cf.~\cite{CerBook}). For the USF of the rarefied gas, the flow velocity behaves as $u^{\rm sh}=(\al x_2,0,0)$ in space, namely, the velocity component in $x_1$-direction is linear along the $x_2$-direction for a constant shear rate $\al> 0$. The shearing motion and the induced viscous heating drive the system to depart from equilibrium. Thus, the energy and hence the temperature monotonically increase in time. It then becomes interesting to determine the global existence of such USF as well as its large time behavior. It turns out that for the Maxwell molecules the existence can be transferred to look for self-similar profiles by taking into account the growth of temperature. Moreover, the self-similar profile is determined by non-Maxwellian solutions of a stationary problem on the Boltzmann equation with the shear force and the velocity relaxation term whose balance leads to the conservation of energy. The shear strength affects how far the self-similar profile is from the Maxwellian equilibrium and a perturbation approach in $\al$ is expected to give the existence of solutions for any small shear rate. In general, the self-similar profile is anisotropic in velocity variables due to shearing motion. The main feature of the self-similar profile verified numerically by Monte Carlo simulations (cf.~\cite{GaSa}) is that it has the polynomial large-velocity tail that will induce the key difficulty in studying the topic. 

We remark that the solutions to the Boltzmann equation for the USF are also called {\it homoenergetic solutions} which were introduced by Galkin \cite{G1} and Truesdell \cite{T}, and they should be distinguished by Truesdell and Mucaster \cite{TM} from the planar Couette flow, cf.~\cite[Chapter 5]{GaSa}, \cite[Chapter 4]{Ko} and \cite[Chapter 4]{Sone07}, for instance.

\subsection{Boltzmann equation for USF}
Mathematically, the USF is governed by the following spatially homogeneous Boltzmann equation
\begin{equation}
\label{usf}
\pa_t F -\al v_2\pa_{v_1}F=Q(F,F).
\end{equation}
Here the unknown $F=F(t,v)\geq 0$ stands for  the velocity distribution function of gas particles with velocity  $v=(v_1,v_2,v_3)\in\R^3$ at time $t\geq0$, and the constant $\al>0$ denotes the shear rate as mentioned before. The Boltzmann collision operator $Q(\cdot,\cdot)$ is bilinear taking the non-symmetric form of
\begin{equation}\label{def.Q}
\begin{split}
Q(F_1,F_2)(v)=&\int_{\R^3}\int_{\S^2}B_0(\cos\ta)[F_1(v'_\ast)F_2(v')-F_1(v_\ast)F_2(v)]\,d\omega dv_{\ast},
%\eqdef Q_+(F_1,F_2)-Q_-(F_1,F_2),
\end{split}
\end{equation}
where the velocity pairs $(v_\ast,v)$ and $(v_\ast',v')$ satisfy the relation
\begin{equation}
\label{v.re}
v'_\ast=v_\ast-[(v_\ast-v)\cdot\omega]\om,\quad v'=v+[(v_\ast-v)\cdot\omega]\om,
\end{equation}
denoting the $\omega$-representation  according to conservations of momentum and energy of two particles before and after the collision
\begin{equation*}
v_\ast+v=v_\ast'+v',\quad
|v_\ast|^2+|v|^2=|v'_\ast|^2+|v'|^2.
\end{equation*}
Through the paper, we assume that the collision kernel $B_0(\cos\theta)$ with $\cos\theta=(v-v_\ast)\cdot \om/|v-v_\ast|$ is independent of the relative speed $|v-v_\ast|$ for the Maxwell molecule model and satisfies the Grad's angular cutoff assumption
\begin{equation}
\label{aca}
0\leq B_0(\cos\theta)\leq C |\cos\theta|
\end{equation}
for a generic constant $C>0$.

\subsection{Moment equations and self-similar formulation}
Provided that $F(t,v)$ decays in large velocity fast enough, we multiply \eqref{usf} by the Boltzmann collision invariants and take integration in velocity so as to obtain
\begin{eqnarray}\label{metG}
\left\{\begin{array}{rll}
\begin{split}
&\frac{d}{dt} \int_{\R^3}F\, dv=0,\\
&\frac{d}{dt}\int_{\R^3} v_1 F\, dv + \al \int_{\R^3} v_2 F\, dv=0,\\
&\frac{d}{dt} \int_{\R^3} v_i F\, dv=0,\quad i=2,3,\\
&\frac{d}{dt} \int_{\R^3} \frac{1}{2}|v|^2 F\, dv +\al \int_{\R^3} v_1v_2 F\, dv=0,
\end{split}
\end{array}\right.
\end{eqnarray}
for any $t\geq 0$. In light of this, without loss of generality, we may assume that the solution $F(t,v)$ to \eqref{usf} satisfies
\begin{equation}\label{clF}
\int_{\R^3} F(t,v) \,dv=1,\quad \int_{\R^3} v_i F(t,v)\, dv=0,\ i=1,2,3,\quad \forall\,t\geq 0.
\end{equation}
The last identity of \eqref{metG} implies that the macroscopic energy of $F(t,v)$ can change in time due to the appearance of shear force. Physically the shearing motion should induce the viscous heat into the system so that the energy indeed increases in time. Moreover, it will be justified later that the heat flux $\int v_1v_2 F\,dv$ turns out to be strictly negative in large time for any small $\al>0$.  

From \cite[Chapter 2]{GaSa} as well as \cite[Section 5.1]{JNV-ARMA}, for the Maxwell molecule model, a specific solution $F(t,v)$ can be self-similar of the form
\begin{equation}
\label{def.trans}
F(t,v)=e^{-3\beta t} G(\frac{v}{e^{\beta t}})
\end{equation}
for a suitable constant $\beta$, where the self-similar stationary profile $G=G(v)$ satisfies
\begin{equation}\label{stG}
-\bet \na_v\cdot(v G)-\al v_2\pa_{v_1}G=Q(G,G).
\end{equation}
To find a solution to \eqref{stG}, it is natural to require that $G(v)$ also satisfies the same conservation laws \eqref{clF} as for $F(t,v)$ and in addition $G$ has a fixed positive energy, namely, without loss of generality,
\begin{equation}
\label{con.ge}
\int_{\R^3}|v|^2G(v)\,dv=3.
\end{equation}
Therefore, from the solvability of the stationary equation \eqref{stG}
\begin{equation*}
%\label{ }
\int_{\R^3}[1,v,|v|^2]\{-\bet \na_v\cdot(v G)-\al v_2\pa_{v_1}G\}\,dv=0,
\end{equation*}
the condition energy law \eqref{con.ge} is equivalent to require
\begin{equation}
\label{def.beta}
\beta=-\al \frac{\int_{\R^3} v_1v_2G\,dv}{\int_{\R^3} |v|^2G\,dv}=-\frac{\al}{3}\int_{\R^3} v_1v_2G\,dv.
\end{equation}
Plugging this back to \eqref{stG} gives
\begin{equation}
\label{eqG}
\frac{1}{3}\int_{\R^3} v_1v_2G\,dv\,\na_v\cdot (vG)-v_2\pa_{v_1}G=\frac{1}{\al}Q(G,G).
\end{equation}
The above equation is a crucial formulation for studying the existence of $G(v)$ via the Hilbert's perturbation approach in the small parameter $\al>0$. In particular, $\beta$ is no longer regarded as an unknown constant, but replaced by a nonlocal integral term. Note that $\al>0$ plays the same role as the Knudsen number. We would emphasize that the current work is only focused on the small shear rate regime with unit Knudsen number, but it is still possible to make use of \eqref{eqG} to discuss the situation of the large shear rate for small Knudsen number in the hydrodynamic regime.

\subsection{Main results}
With these preparations above, we are ready to state the main results of the paper. The first one is concerned with the existence of smooth self-similar profiles for the stationary Boltzmann problem  \eqref{eqG} under the assumption on smallness of shear rate $\al>0$. For this purpose we define the global Maxwellian
\begin{equation}
\label{def.mu}
\mu=(2\pi)^{-3/2}\exp(-|v|^2/2),
\end{equation}
and introduce the velocity weight function $w_l=w_l(v):=(1+|v|^2)^l$ with $l\in \R$.

\begin{theorem}\label{st.sol}
There is $l_0>0$ such that for any $l\geq l_0$, there is $\al_0=\al_0(l)>0$ depending on $l$ such that for any $\al\in (0,\al_0)$, the stationary Boltzmann equation \eqref{eqG} admits a unique smooth solution $G=G(v)$ satisfying
\begin{equation}
\label{thm.cons}
\int_{\R^3}[1,v,|v|^2]G(v)\,dv=[1,0,3],
\end{equation}
and
\begin{equation}
\label{thm.est.G}
\|w_l \na_v^k[G-(1-\frac{\al}{2b_0}v_1v_2)\mu]\|_{L^\infty}\leq C_{k,l} \al^2,
\end{equation}
for any integer $k\geq 0$, where $C_{k,l}$ is a constant independent of $\al$, and $b_0$ is a positive constant defined by
\begin{align}\label{b0}
b_0=3\pi\int_{-1}^1B_0(z)z^2(1-z^2)\,dz.
\end{align}
\end{theorem}

\begin{remark}
Here are a few remarks in oder on Theorem \ref{st.sol}. 
\begin{itemize}

\item[(a)] The estimate \eqref{thm.est.G} implies that as $\al\to 0$, the self-similar profile $G(v)$ behaves as 
\begin{equation}
\label{thm.com.exp}
G(v)=\mu-\frac{\al }{2b_0}v_1v_2\mu + O(\al^2),
\end{equation} 
where $\mu$ is uniquely determined by conservation laws \eqref{thm.cons}, 
and correspondingly by \eqref{def.beta}, as $\al\to 0$, the constant $\beta$ behaves as 
\begin{equation}
\label{beta.exp}
\beta=\frac{\al^2}{6b_0}+O(\al^3).
\end{equation}
Thus, Theorem \ref{st.sol} not only gives the existence of smooth solutions $G(v)$, but also provides the $\al$-dependent structure of $G$. Note that beyond the expansion \eqref{thm.com.exp} up to the first order, it is possible to further obtain the coefficient velocity functions of the second and third orders of $\al$ by iteration, see \cite[(2.126) and (2.127), page 88]{GaSa}. Moreover, \eqref{beta.exp} implies that $\beta$ is strictly positive and hence by \eqref{con.ge}, the energy of the self-similar solution \eqref{def.trans}, given by
\begin{equation*}
%\label{ }
\int_{\R^3} |v|^2 F(t,v)\,dv=3e^{2\beta t},
\end{equation*} 
indeed tends to infinity exponentially in time.

\item[(b)] In general, from \eqref{thm.est.G}, $G(v)$ has to be anisotropic in $v$ due to the shearing motion, and any $l$-th order velocity moments of $G(v)$ are finite as long as the shear rate $\al>0$ is small enough. Due to the dependence of $\al_0$ on $l$, in particular, one can choose $\al_0(l)\sim \frac{1}{l}$ for any $l>0$ large enough from the later proof, it is impossible to obtain a positive shear rate $\al_0$ such that  \eqref{thm.est.G} holds true uniformly in any $l>0$, particularly allowing $l\to \infty$. This property exactly features that $G(v)$ may admit the polynomial large-velocity tail as confirmed numerically by Monte Carlo simulations, cf.~\cite{GaSa}.

\item[(c)] The constant $b_0>0$, describing the magnitude of collisions, is obviously finite under the angular cutoff assumption \eqref{aca}, and it has been assumed to be independent of $\al>0$, meaning that the shear rate need to be small enough compared to the strength of collisions. It is interesting to further study the property of self-similar profiles in case when collisions are strong or weak enough corresponding to the hydrodynamic limit or the free molecule limit, respectively. Furthermore, since  $b_0$ can be well-defined even in the case of the angular non-cutoff by \eqref{b0}, it is also interesting to extend the current result to the non-cutoff situation that is certainly more challenging than the current consideration due to the necessary $L^\infty$ estimates on solutions.      
\end{itemize}

\end{remark}

Moreover, we are concerned with the global existence and large time behavior of solutions to the original USF equation \eqref{usf} supplemented with a suitable initial data, namely, in terms of the self-similar reformulation \eqref{def.trans} with the value of $\beta$ obtained from Theorem \ref{st.sol}, it is natural to further study whether or not it is holds true that
\begin{equation}
\label{lta.h}
e^{3\beta t} F(t,e^{\beta t}v)\to G(v)
\end{equation}
in a certain sense as time goes to infinity whenever they are close to each other initially, where $G(v)$ is the self-similar profile obtained in Theorem \ref{st.sol}  and the constant $\beta$ is defined by \eqref{def.beta} in terms of $G(v)$. As a byproduct, a direct consequence of such large time asymptotic stability is the non-negativity of $G(v)$. To treat this issue, instead of directly starting with the spatially homogeneous Boltzmann equation \eqref{usf} for the USF, we turn to the spatially inhomogeneous setting for a more general purpose. In fact, let the rarefied gas flow be contained in an infinite channel $\T_{x_1}\times \R_{x_2}$ and uniform in $x_3$-direction, then the governing Boltzmann equation  takes the form of
\begin{equation}\label{2-D}
\pa_t \widetilde{F} +w_1 \pa_{x_1} \widetilde{F}+w_2 \pa_{x_2} \widetilde{F}=Q(\widetilde{F},\widetilde{F}),
%\ x\in\T,\ y\in\R,\ w\in\R^3,
\end{equation}
for the spatially inhomogeneous velocity distribution function $\widetilde{F}=\widetilde{F}(t,x_1,x_2,w)$ with $t\geq 0$, $x_1\in \T$, $x_2\in \R$ and $w=(w_1,w_2,w_3)\in \R^3$. Looking for a solution of the specific form $\widetilde{F}=F(t,x_1,w_1-\al x_2,w_2,w_3)$, one can see that $F$ satisfies the following Boltzmann equation in a one-dimensional periodic box:
\begin{equation}
\label{beF}
\pa_t F +v_1\pa_{x}F-\al v_2\pa_{v_1}F=Q(F,F),\quad t>0,x\in \T,v\in \R^3,
\end{equation}
supplemented with initial data
\begin{equation}
\label{beFid}
F(0,x,v)=F_0(x,v),\ x\in\T,\ v\in\R^3.
\end{equation}
Here for brevity of presentation we have used $x$ to denote the first  component of space variables.

The second result of the paper is related to the large time asymptotics of solutions to the spatially inhomogeneous problem \eqref{beF} and \eqref{beFid}.

\begin{theorem}\label{ge.th}
Let $G(v)$ be the self-similar profile obtained in Theorem \ref{st.sol}  and the constant $\beta$ be defined in \eqref{def.beta}. There are constants $\varepsilon_0>0$, $\la>0$ and $C>0$ such that if $F_0(x,v)\geq0$ and
\begin{equation}\label{ge.th.ids}
\sum\limits_{0\leq \ga_0\leq2}\left\|w_l\mu^{-\frac{1}{2}}\pa_x^{\ga_0}\left[F_0(x,v)-G(v)\right]\right\|_{L^\infty}\leq\varepsilon_0,
\end{equation}
and
\begin{align}\label{id.cons}
\int_{\T}\int_{\R^3}[F_0(x,v)-G]\,dvdx=0,\quad \int_{\T}\int_{\R^3}vF_0(x,v)\,dvdx=0,
\end{align}
then the Cauchy problem \eqref{beF} and \eqref{beFid} admits a unique solution $F(t,x,v)\geq0$ satisfying the following estimates:
\begin{equation}\label{lif.decay}
\left\|w_l(v)\left[e^{3\beta t}F(t,x,e^{\beta t}v)-G(v)\right]\right\|_{L^\infty}\\
\leq C e^{-\la \beta t}\sum\limits_{0\leq \ga_0\leq2}\left\|w_l\mu^{-\frac{1}{2}}\pa_x^{\ga_0}\left[F_0(x,v)-G(v)\right]\right\|_{L^\infty},
\end{equation}
and
\begin{equation}\label{lif.decay.der}
\sum\limits_{1\leq\ga_0\leq2}\left\|w_l(v)e^{3\beta t}\pa_x^{\ga_0}F(t,x,e^{\beta t}v)\right\|_{L^\infty}
\leq Ce^{-\la  t}\sum\limits_{1\leq\ga_0\leq2}\left\|w_l(v)\mu^{-\frac{1}{2}}\pa_x^{\ga_0}F_0(x,v)\right\|_{L^\infty},
\end{equation}
for any $t\geq 0$.
\end{theorem}

\begin{remark}\label{thm2.rm}
We give a few remarks on Theorem \ref{ge.th} as follows.

\begin{itemize}
  
\item[(a)] Whenever $F$ is spatially homogeneous, as a direct consequence of  Theorem \ref{ge.th}, the large time asymptotics \eqref{lta.h} for solutions to the USF \eqref{usf}  towards the self-similar profile $G$ is also justified in the velocity weighted $L^\infty$ setting. In particular, from \eqref{lif.decay} one has
\begin{equation}
\label{lif.decay.sh}
\|w_l [e^{3\beta t} F(t,e^{\beta t}\cdot)-G]\|_{L^\infty}\leq C e^{-\la \beta t}\|w_l\mu^{-\frac{1}{2}} (F_0-G)\|_{L^\infty}
\end{equation} 
for any $t\geq 0$.

\item[(b)]  Estimate \eqref{lif.decay} or \eqref{lif.decay.sh} implies that the rate of convergence is exponential with the size proportional to $\beta\sim \al^2$. Such property features the shearing motion for small $\al>0$. In fact, when $\al=0$, the large time behavior of solutions to \eqref{usf} is the global Maxwellian equilibrium uniquely determined by initial data $F_0(v)$ through all the conservative fluid quantities, and the convergence rate is exponential with the size given by the spectral gap of the linearized Boltzmann operator. 

For $\al>0$, it is not necessary to impose that $F_0$ has the same energy as $G$ except the mass and momentum conservation \eqref{id.cons}, because in the self-similar setting the energy of the rescaled distribution function is dissipative with the magnitude of dissipation rate proportional to $\beta$ due to the linear relaxation effect arising from the term $-\beta\na_v\cdot (vF)$. Precisely, let $f(t,x,v)=e^{3\beta t}F(t,x,e^{\beta t}v)$, then it follows from \eqref{beF} that  
\begin{align}
\frac{d}{dt}\int_{\T}dx \int \frac{1}{2}|v|^2(f-G)\,dv&+\beta \int_{\T}dx \int |v|^2(f-G)\,dv\notag\\
&+\al \int_{\T}dx \int v_1v_2(f-G)\,dv=0.\label{rem.eit}
\end{align}  
This identity implies that the size of the exponential convergence rate in \eqref{lif.decay} or \eqref{lif.decay.sh} is optimal; we also will explain this point in more detail in Section \ref{sec.subsp} later.  

\item[(c)] Estimate \eqref{lif.decay.der} implies that the convergence rate of the higher order spatial derivatives is much faster than the one of the zero order, since the size of convergence is independent of the shear rate $\al$.  This indicates that the collision of particles dominates the long time asymptotics of the energy for the higher order spatial derivatives.

\item[(d)] The smallness assumption \eqref{ge.th.ids} on initial data implies that the initial perturbation has to admit an additional large velocity decay as $\mu^{1/2}$. This restriction is essentially due to the perturbation method of the proof. It is interesting to remove such restriction using an alternative approach, for instance, in \cite{GMM}.

\end{itemize} 

\end{remark}

\subsection{Literature} In what follows we mention some known results on the self-similar solutions to the Boltzmann equation in case of the Maxwell molecule model. When $\al=0$, namely, there is no shear effect, the mathematical study of the problem was initiated by Bobylev and Cercignani \cite{BC02b,BC02a,BC03}. Since the energy remains conservative, the self-similar profile exists only when it has infinite second order moments. The dynamical stability of such infinite energy self-similar profile was proved by Morimoto, Yang and Zhao \cite{MYZ} in the angular non-cutoff case; see also the previous investigation Cannone and Karch \cite{CK10, CK13} on this topic.

When $\al\neq 0$, the global-in-time existence of solutions to the Boltzmann equation \eqref{usf} for the USF was first established by Cercignani \cite{Cer89,Cer00,Cer02}. The group invariant property in the higher dimensional case was discussed in Bobylev, Caraffini and Spiga \cite{BCS}. Recently, in series of significant progress by James, Nota and Vel\'azquez \cite{JNV-ARMA,JNV-JNS,JNV19}, the existence of homoenergetic mild solutions as non-negative Radon measures was studied in a systematic way for a large class of initial data, where the admissible macroscopic shear velocity $u^{\rm sh}=L(t)x$ with $L(t):=A(I+tA)^{-1}$ for a constant matrix $A$ is characterized and the asymptotics of homoenergetic solutions that do not have self-similar profiles is also conjectured in certain situations. An interesting work by Matthies and Theil \cite{MT} also showed that the self-similar profile does not have the same exponential large-velocity tail as the global Maxwellian. Applying the Fourier transform method that is a fundamental analysis tool in the Boltzmann theory introduced by Bobylev \cite{Bo75,Bo88}, a recent progress Bobylev, Nota and Vel\'azquez \cite{BNV-2019} proved the self-similar asymptotics of solutions in large time for the Boltzmann equation with a general deformation of the form
\begin{equation*}
%\label{ }
\pa_t F-\na_v\cdot (AvF)=Q(F,F)
\end{equation*} 
under a smallness condition on the matrix $A$, and they also showed that the self-similar profile can have the finite polynomial moments of higher order as long as the norm of $A$ is smaller. 
It seems that \cite{BNV-2019} is the only known result on the large time asymptotics to the self-similar profile in weak topology. 

In the end, we remark that there have been extensive studies of stability of shear flow in an infinite channel domain $\T_{x_1}\times \R_{x_2}$ in the context of fluid dynamic equations, cf.~\cite{ScHe}, in particular, we mention great contributions \cite{BM-15,BMV-16,BGM-17} recently made by Bedrossian together with his collaborators. In fact, in comparison with \eqref{beF} under consideration, it would be more interesting to study the large time behavior of solutions to the original Boltzmann equation \eqref{2-D} in the 2D domain $\T\times \R$ in order to gain further understandings of the stability issue similar to those aforementioned works on fluid equations by taking the limit of either small or large Knudsen number.

\subsection{Strategy of the proof}\label{sec.subsp}
The main ideas and techniques used in the paper are outlined as follows.

\begin{itemize}
\item 
First of all, in the framework of perturbation, there is a severe velocity-growth term in the form of $v_1v_2G$ which is caused by the shearing motion. Specifically, setting the perturbation as 
$
G=\mu+\al\mu^{1/2}(G_1+G_R)
$ 
where $G_1$ is included to remove the zero-order inhomogeneous term, $G_R$ satisfies an equation of the form
\begin{equation*}
%\label{ }
\cdots+\frac{\al}{2}v_1v_2 G_R+LG_R=\cdots.
\end{equation*}
Here $v_1v_2 G_R$ becomes a trouble term to control in the basic energy estimate in term of the dissipation of the linearized self-adjoint operator $L$. 

To overcome the difficulty, we borrow the idea given by Caflisch \cite{Ca-1980}, where the solution is split into two parts: one includes the exponential weight while the other does not, namely, we set
\begin{equation*}
%\label{ }
\mu^{\frac{1}{2}}G_R=G_{R,1}+\mu^{\frac{1}{2}}G_{R,2}.
\end{equation*}
The key point here is that we put the bad terms mentioned above into the one without exponential weight, so as to eliminate the velocity growth. Roughly $G_{R,1}$ and $G_{R,2}$ satisfy the coupling equations of the form
\begin{eqnarray*}
 \cdots-\al v_2\pa_{v_1} G_{R,1} +\frac{\al}{2}v_1v_2 \mu^{\frac{1}{2}} G_{R,2}+\nu_0 G_R&=&\chi_M\CK G_{R,1}+\cdots,\\
 \cdots-\al v_2 \pa_{v_1} G_{R,2}+LG_{R,2}&=&(1-\chi_M)\mu^{-\frac{1}{2}}\CK G_{R,1}+\cdots,
\end{eqnarray*}
after ignoring the high order or nonlinear terms, where $\chi_M$ is a velocity cutoff function defined in \eqref{def.chiM}, and other notations are introduced in Section \ref{sec.app}.

We should point out that as confirmed by the numerical result, one may only expect the first part $G_{R,1}$ to decay polynomially in large velocity. To understand this issue mathematically, we consider the equation of the form
\begin{equation*}
%\label{ }
-\al v_2 \pa_{v_1}G_{R,1} +\nu_0 G_{R,1}=\cdots,
\end{equation*} 
where $\nu_0>0$ is the constant collision frequency corresponding to $\mu$ in case of the Maxwell molecules.  
Multiplying the above equation with the polynomial weight $w_l=(1+|v|^2)^l$ gives
\begin{equation}
\label{sopagr1}
-\al v_2\pa_{v_1}(w_lG_{R,1}) + \left(\nu_0+2\al l \frac{v_1v_2}{1+|v|^2}\right)w_l G_{R,1}=\cdots.
\end{equation}   
Therefore, given $l>0$ large, we need to require $0<\al<\al_0(l)\sim \frac{1}{l}$ that is small enough such that 
\begin{equation*}
%\label{ }
\inf_{v}\left(\nu_0+2\al l \frac{v_1v_2}{1+|v|^2}\right)\geq \frac{1}{2}\nu_0
\end{equation*}
holds true and hence $w_lG_{R,1}$ can be shown to be bounded in all $v$ in terms of \eqref{sopagr1}.

\item 
Although the Caflisch's decomposition provides us the great advantage above, it also prevents us from deducing the $L^2$ estimates of the solution, particularly for the first part of the decomposition, due to the decay-loss of the operator $\CK$. 

To treat the difficulty, we resort to the $L^\infty$-$L^2$ method developed recently by Guo \cite{Guo-2010}; see also \cite{DHWZ-19,EGKM-13,GL-2017}. One of the key points when applying this approach is the decay of the operator $K$ for large velocity. At the current stage, it is quite hard to achieve any decay rates of $\CK$. Fortunately, motivated by Arkeryd, Esposito and Pulvirenti \cite{AEP-87}, we justify the crucial estimates for such $\CK$ with the  algebraic velocity weight. More precisely, we find out the following ``decay" rate for the large power of the velocity weight
\begin{align*}
\sup_{|v|\geq M}w_{l}|\na_v^k\CK f|\leq \frac{C}{l} \sum\limits_{0\leq k'\leq k}\|w_{l}\na_v^{k'}f\|_{L^\infty},
\end{align*}
with $k\geq 1$ and $M\sim l^2$, where we refer to Proposition \ref{CK} for more details. We remark that such estimate holds true for the Maxwell molecules only, as seen from the derivation of \eqref{CI223} in the proof of Proposition \ref{CK} later on. 

\item  
In addition, the $L^2$ estimate for the second part of the decomposition is also difficult to obtain due to the inhomogeneous structure of the splitting equation. 

To deal with this difficulty, for the steady case, the conservation laws of solutions are essentially used, so that both the first order correction and the remainder of the steady solution are microscopic, then the $L^2$ estimate can be directly obtained by the energy estimate. 

As to the unsteady case, since the energy is no longer conserved, the argument for the steady problem is invalid. In fact, in the time-dependent situation, the zero order dissipation of the temperature is captured by exploring the structure of the macroscopic equations which contains the weak damping generated by the shear flow. More specifically, the dissipation of the temperature of the unsteady solution is derived from the macroscopic thirteen moments equations and the lower order terms are cancelled by introducing a new anti-derivative of a second order momentum. For instance, using the cancellation property
$$
(\pa_xc,\int_0^x d_{12}\,dy)+(d_{12},c)=0,
$$
one can derive from \eqref{ab}, \eqref{c} and \eqref{Aij} that   
$$
\frac{d}{dt}\left\{\|c\|^2+(b_1,\frac{1}{3}\al e^{-\beta t}\int_0^x d_{12}\,dy)\right\}
+2\beta\|c\|^2
\leq \cdots.
$$
We may refer to \eqref{dis.c} for more details. The above energy estimate in perturbation framework is consistent with the energy identity \eqref{rem.eit} mentioned in Remark \ref{thm2.rm}.

\end{itemize}
    
\subsection{Organization of the paper}
The rest structure of this paper is arranged as follows. In Section \ref{sec.app} we list the basic estimates on the linearized operator $L$ as well as the nonlinear operators $\Ga$ and $Q$, and also present an explicit formula of $L(v_iv_j\mu^{1/2})$ in the case of the Maxwell molecule model. In Section \ref{sec.CK}  we give a key estimate for the operator $\CK$. The existence of the self-similar stationary profile $G(v)$ for \eqref{eqG} is constructed in Section \ref{sta.prob}. In Section \ref{loc.sec},
we turn to  the unsteady problem of the spatially inhomogeneous Boltzmann equation \eqref{beF} and \eqref{beFid} and establish the local-in-time existence of solutions. 
Finally, Section \ref{ge.sec} is devoted to showing the global existence of solutions and large time asymptotic behavior for the Cauchy problem \eqref{beF} and \eqref{beFid}.

\subsection{Notations}
We now list some notations  used in the paper.
 \begin{itemize}
 \item
 Throughout this paper,  $C$ denotes some generic positive (generally large) constant and $\la$ denotes some generic positive (generally small) constants, where $C$ and $\la$  may take different values in different places. $D\lesssim E$ means that  there is a generic constant $C>0$
such that $D\leq CE$. $D\sim E$
means $D\lesssim E$ and $E\lesssim D$.
\item We denote $\Vert \,\cdot \,\Vert $ the $L^{2}(\T
\times \R^{3})-$norm or the $L^{2}(\T)-$norm or $L^{2}(\R^3)-$norm.
Sometimes, we use $\|\,\cdot \,\|_{L^\infty }$ to denote either the $L^{\infty }(\T
\times \R^{3})-$norm or the $L^{\infty }(\R^3)-$norm. Moreover, %we denote $\|\cdot \|_{\nu}\equiv \|\nu^{1/2}\cdot\|_2$, and
$(\cdot,\cdot)$ denotes the $L^{2}$ inner product in
$\T\times {\R}^{3}$  with
the $L^{2}$ norm $\|\cdot\|$ and $\langle\cdot\rangle$ denotes the $L^{2}$ inner product in $\R^3_v$.

\end{itemize}

%\appendix
\section{Basic estimates}\label{sec.app}
In this section, we collect some basic estimates which will be used in the next sections. Let us first give some elementary estimates for the linearized collision operator $L$ and nonlinear collision operator $\Ga$, defined by
\begin{align}\label{def.L}
Lg=-\mu^{-1/2}\left\{Q(\mu,\sqrt{\mu}g)+Q(\sqrt{\mu}g,\mu)\right\},
\end{align}
and
$$
\Gamma (f,g)=\mu^{-1/2} Q(\sqrt{\mu}f,\sqrt{\mu}g)=\int_{\R^3}\int_{\S^2}B_0\mu^{1/2}(v_\ast)[f(v_\ast')g(v')-f(v_\ast)g(v)]\,d\om dv_\ast ,
$$
respectively. Note that $Lf=\nu f-Kf$ with
\begin{align}\label{sp.L}
\nu=\int_{\R^3}\int_{\S^2}B_0(\cos \ta)\mu(v_\ast)\, d\om dv_\ast=\nu_0,\quad
Kf=\mu^{-\frac{1}{2}}\left\{Q(\mu^{\frac{1}{2}}f,\mu)+Q_{\textrm{gain}}(\mu,\mu^{\frac{1}{2}}f)\right\},
\end{align}
where $Q_{\textrm{gain}}$ denotes the positive part of $Q$ in \eqref{def.Q}. The kernel of $L$, denoted as $\ker L$, is a five-dimensional space spanned by $\{1,v,|v|^2-3\}\sqrt{\mu}:= \{\phi_i\}_{i=1}^5$. We further define a projection from $L^2$ to $\ker(L)$ by
\begin{align*}%\label{P0}
\FP_0 g=\left\{a_g+\Fb_g\cdot v+(|v|^2-3)c_g\right\}\sqrt{\mu}
\end{align*}
for $g\in L^2$, and correspondingly denote the operator $\FP_1$ by $\FP_1g=g-\FP_0 g$, which is orthogonal to $\FP_0$.

It is also convenient to define
\begin{align}\notag
\CL f=-\left\{Q(f,\mu)+Q(\mu,f)\right\}=\nu f-\CK f,
\end{align}
with
\begin{equation}\label{sp.cL}
\nu f=\nu_0f,\quad \CK f=Q(f,\mu)+Q_{\textrm{gain}}(\mu,f)=\sqrt{\mu}K(\frac{f}{\sqrt{\mu}}),
\end{equation}
according to \eqref{sp.L}.

The following lemma can be found in \cite[Lemmas 3.2, 3.3, pp.638-639]{Guo-2006}, where the more general hard sphere case is proved.

\begin{lemma}\label{es.L}
In the Maxwell molecular case, there is a constant $\de_0>0$ such that
\begin{align}\notag
\lag Lf,f\rag=\lag L\FP_1f,\FP_1f\rag\geq\de_0\|\FP_1f\|^2.
\end{align}
Moreover, for $\ga>0$ and $l\geq0$,
\begin{align}\notag
\lag w^2_l\pa_v^\ga L f,\pa_v^\ga  f\rag\geq\de_0\|w_l\pa_v^\ga f\|^2-C\|f\|^2.
\end{align}
\end{lemma}

The following lemma is concerned with the integral operator $K$ given by \eqref{sp.L}, and its proof in case of the hard sphere model has been given by \cite[Lemma 3, pp.727]{Guo-2010}.

\begin{lemma}\label{Kop}
Let $K$ be defined as \eqref{sp.L}, then it holds that
\begin{align}\notag
Kf(v)=\int_{\R^3}\Fk(v,v_\ast)f(v_\ast)\,dv_\ast
\end{align}
with
\begin{equation*}
|\Fk(v,v_\ast)|\leq C\{1+|v-v_\ast|^{-2}\}e^{-
\frac{1}{8}|v-v_\ast|^{2}-\frac{1}{8}\frac{\left||v|^{2}-|v_\ast|^{2}\right|^{2}}{|v-_\ast|^{2}}}. %\label{grad}
\end{equation*}
Moreover, let
$
\Fk_w(v,v_\ast)=w_{l}(v)\Fk(v,v_\ast)w_{-l}(v_\ast)
$
with $l\geq0$,
then it also holds that
\begin{equation*}
\int_{\R^3} \Fk_w(v,v_\ast)e^{\frac{\varepsilon|v-v_\ast|^2}{8}}dv_\ast\leq \frac{C}{1+|v|},
\end{equation*}
for $\varepsilon\geq 0$ small enough.
\end{lemma}

For the velocity weighted derivative estimates on the nonlinear operator $\Ga$, one has

\begin{lemma}\label{Ga}
In the Maxwell molecular case, it holds that
\begin{align}\label{es1.Ga}
\|w_l\pa_v^\ga\Ga(f,g)\|_{L^2}\leq C\sum\limits_{\ga_1\leq\ga}\|w_l\pa_v^{\ga_1}f\|_{L^2}\|w_l\pa_v^{\ga-\ga_1}g\|_{L^2},
\end{align}
and
\begin{align}\label{es11.Ga}
\|w_l\pa_v^\ga\Ga(f,g)\|_{L^\infty}\leq C\sum\limits_{\ga_1\leq\ga}\|w_l\pa_v^{\ga_1}f\|_{L^\infty}\|w_l\pa_v^{\ga-\ga_1}g\|_{L^\infty},
\end{align}
for any multiple index $\ga$ and any $l\geq0$.
\end{lemma}

\begin{proof}
The proof of \eqref{es1.Ga} and \eqref{es11.Ga} is similar as that of \cite[Lemma 2.3, pp.1111]{Guo-vpb} and \cite[Lemma 5, pp.730]{Guo-2010},
respectively. Thus we omit the details for brevity.
\end{proof}

The following Lemma on the velocity weighted derivative estimates for the original Boltzmann equation $Q$ can be verified by using the parallel argument as obtaining \cite[Proposition 3.1, pp.397]{AEP-87} where the hard potential case and the case $|\ga|=0$ were proved.

\begin{lemma}\label{op.es.lem}For $l>\frac{3}{2}$ and $|\ga|\geq0$, it holds that
\begin{equation}\notag
\|w_{l} \pa_v^\ga Q(F_1,F_2)\|_{L^\infty}\leq C\sum\limits_{\ga_1\leq\ga}\|w_{l} \pa_v^{\ga-\ga_1} F_1\|_{L^\infty}\|w_{l} \pa_v^{\ga_1}F_2\|_{L^\infty}.
\end{equation}

%and moreover, it holds that
%\begin{equation}\label{op.es2}
%\int_{\R^3}\frac{|\pa_v^\ga Q(F_1,F_2)|^2}{\mu}dv\leq C\sum\limits_{\ga_1\leq\ga}\int_{\R^3}\frac{|\pa_v^{\ga_1} F_1|^2}{\mu}dv\int_{\R^3}\frac{|\pa_v^{\ga-\ga_1} F_2|^2}{\mu}dv.
%\end{equation}
\end{lemma}

We now give the following two useful results concerning the second momentum invariant property of the linearized operator $L$ in the case of Maxwell molecules. The first one is due to \cite[Proposition 4.10, pp.804]{JNV-ARMA}.

\begin{lemma}\label{nice.lem}
Let $W_{ij}(v)$ be quadratic functions in the form of $W_{ij}(v)=v_iv_j$$(1\leq i,j\leq3)$ and define
\begin{align}\label{Tij}
T_{ij}=\frac{1}{2}\int_{\S^2}d\om \,B_0(\cos \ta)\left[W_{i,j}(v')+W_{i,j}(v_\ast')-W_{i,j}(v)-W_{i,j}(v_\ast)\right],
\end{align}
where $(v,v_\ast)$ and $(v',v_\ast')$ satisfies \eqref{v.re}. Then it holds that
\begin{align}\label{Tij.eq}
T_{ij}=-b_0\left[(v-v_\ast)_i(v-v_\ast)_j-\frac{\de_{ij}}{3}|v-v_\ast|^2\right],
\end{align}
with $b_0$ given in \eqref{b0}.
\end{lemma}

Based on the above nice lemma, we can obtain

\begin{lemma}\label{our.lem}
Let $L$ be defined as \eqref{def.L}, then it holds that for all $1\leq i,j\leq3$
\begin{align}\label{Lb0}
L(v_iv_j\mu^{1/2})=2b_0(v_iv_j-\frac{\de_{ij}}{3}|v|^2)\mu^{1/2}.
\end{align}
\end{lemma}
\begin{proof}
For $f=\mu^{1/2}W$ with a general function $W=W(v)$, one has
\begin{equation*}
Lf=-\mu^{1/2}\int \mu_\ast \,dv_\ast \int d\omega \,B_0(\cos\theta) [W'+W_\ast'-W-W_\ast].
\end{equation*}
In particular,
letting $W=W_{ij}(v)=v_iv_j$ and applying Lemma \ref{nice.lem}, we have
\begin{equation}
\label{p1}
L(\mu^{1/2}W_{ij})=- 2\mu^{1/2} \int \mu_\ast T_{ij} \,dv_\ast,
\end{equation}
where $T_{ij}$ is given by \eqref{Tij}. Plugging \eqref{Tij.eq} into \eqref{p1}, one sees that \eqref{Lb0} is valid. This completes the proof of Lemma \ref{our.lem}.
\end{proof}

%\newpage
\section{Large velocity decay of $\CK$}\label{sec.CK}

Recall \eqref{sp.cL} for the definition of the operator $\CK$, namely,
\begin{equation}
\label{def.ck}
\CK f=Q(f,\mu)+Q_{\textrm{gain}}(\mu,f)=\int_{\R^3}\int_{\S^2} B_0(\cos\theta) (f_\ast'\mu'-f_\ast\mu +\mu_\ast' f')\,d\omega dv_\ast.
\end{equation}
In this section we present a crucial estimate on $\CK$ meaning that the weighted velocity derivatives of $\CK$ are small for large velocities as long as the power of the polynomial velocity weight is large enough. Such property plays a vital role in the proof of the next sections.
%The key point here is that we make full use of the Peetre's inequality so as to surmount the difficulty stemming from the polynomial weights.

\begin{proposition}\label{CK}
Let $\CK$ be given by \eqref{def.ck}, then for any positive integer $k\geq 1$, there is $C>0$ such that for any arbitrarily large $l>0$, there is $M=M(l)>0$ such that it holds that 
\begin{align}\label{CK1}
\sup_{|v|\geq M} w_{l}|\na_v^k\CK f|\leq \frac{C}{l} \sum\limits_{0\leq k'\leq k}\|w_{l}\na_v^{k'}f\|_{L^\infty}.
\end{align}
In particular, one can choose $M=l^2$.
%where $C>0$ is independent of $l$.
%for $l$ is positive and suitably large.
\end{proposition}

\begin{proof}
Fix an integer $k\geq 1$, and take $l>0$ that can be arbitrarily large. Let $M>0$ be large to be suitably chosen in terms of $l$ in the later proof. We define $\chi_{M}(v)$ to be a non-negative smooth cutoff function such that
\begin{align}\label{def.chiM}
\chi_{M}(v)=\left\{\begin{array}{rll}
1,&\ |v|\geq M+1,\\[2mm]
0,&\ |v|\leq M.
\end{array}\right.
\end{align}
%for $M>0$ sufficiently large. 
In light of \eqref{def.ck}, we have
\begin{align}\notag
w_{l}&\chi_{M}\na_v^k\CK f\eqdef\CI_1+\CI_2,
\end{align}
with
\begin{align*}
\CI_1=-w_{l}\chi_{M}\int_{\R^3}\int_{\S^2}B_0f(v_\ast)\na_v^k\mu(v)\,d\om dv_\ast,
\end{align*}
and
\begin{align*}
\CI_2=w_{l}\chi_{M}\sum\limits_{k_1\leq k}C_k^{k_1}\Bigg\{&\int_{\R^3}\int_{\S^2}B_0\na_v^{k_1}f(v'_\ast)\na_v^{k-k_1}\mu(v')\,d\om dv_\ast\\
&\qquad\qquad+\int_{\R^3}\int_{\S^2}B_0\na_v^{k_1}f(v')\na_v^{k-k_1}\mu(v'_\ast)\,d\om dv_\ast\Bigg\}.
\end{align*}
We now compute $\CI_1$ and $\CI_2$. For $\CI_1$, one directly has
\begin{equation}\label{CI1}
\CI_1\leq Cw_{l}\chi_{M}\na_v^k\mu(v)\|w_{l}f\|_{L^\infty}\int_{\R^3}w_{l}^{-1}dv
\leq Ce^{-\frac{M^2}{16}}\|w_{l}f\|_{L^\infty},
\end{equation}
thanks to the assumption that $M\gg 1$ and $l>\frac{3}{2}$, for instance.

For $\CI_2$,
we first rewrite it as
\begin{align}\notag
\CI_2=w_l\chi_{M}\sum\limits_{k_1\leq k}C_k^{k_1}\int_{\R^3}\int_{\S^2}B^\ast_0\na_v^{k_1}f(v')\na_v^{k-k_1}\mu(v'_\ast)\,d\om dv_\ast,
\end{align}
where $B^\ast_0=\frac{1}{2}(B_0(\cos \ta)+B_0(\sin \ta)).$ As it is shown in \cite[(3.2), pp.397]{AEP-87}, we now resort to the Carleman's representation, i.e.
\begin{equation}\notag
\int_{\R^3}\int_{\S^2}B^\ast_0\na_v^{k_1}f(v')\na_v^{k-k_1}\mu(v'_\ast)\,d\om dv_\ast
=\int_{\R^3}\frac{\na_v^{k-k_1}\mu(v_\ast')}{|v-v'_\ast|^2}\int_{E(v,v_\ast')}\na_v^{k_1}f(v')B^\ast_0 \,d\Pi_{v'}dv_\ast',
\end{equation}
where
$$E(v,v_\ast')=\{v'\big| (v-v')\cdot(v-v_\ast')=0,\ |v-v'|\leq|v-v_\ast'|\}\subset\R^2,$$
and $\Pi_{v'}$ is the Lebesgue measure on the hyperplane $E(v,v_\ast')$. 
%\Red{In addition, one can check that}
%$$
%\Red{|\na_v^kf(v')|\leq|\na^k_{v'}f(v')|,\quad |\na_v^kf(v_\ast')|\leq|\na^k_{v_\ast'}f(v')|,}
%$$
%\Red{according to}
%\begin{equation*}
%\Red{\left\{\begin{aligned}
%v'=\frac{v+v_\ast}{2}+\frac{|v-v_\ast|}{2}\si,\\
%v_\ast'=\frac{v+v_\ast}{2}-\frac{|v-v_\ast|}{2}\si,
%\end{aligned}\right.\quad \si\in\S^2.}
%\end{equation*}
Next, we define
\begin{align*}
\chi_1=\chi_1(\xi)=\left\{\begin{array}{rll}
&1,\ |\xi|<\frac{|v|}{\sqrt{2}},\\[2mm]
&0,\ \textrm{otherwise},
\end{array}\right.
\end{align*}
and $\chi_0=1-\chi_1.$ %we also denote $f_i=f\chi_i$$(i=1,0)$ for brevity.
We then decompose $\CI_2$ into
\begin{equation*}%\label{I2.dc}
\begin{split}
\CI_2=&w_l\chi_{M}\sum\limits_{k_1\leq k}C_k^{k_1}\int_{\R^3}\frac{\na_v^{k-k_1}\mu(v_\ast')
\chi_0(v_\ast')}{|v-v'_\ast|^2}\int_{E(v,v_\ast')}\na_v^{k_1}f(v')\chi_0(v')B^\ast_0 \,d\Pi_{v'}dv_\ast'\\
&+w_l\chi_{M}\sum\limits_{k_1\leq k}C_k^{k_1}\int_{\R^3}\frac{\na_v^{k-k_1}\mu(v_\ast')
\chi_1(v_\ast')}{|v-v'_\ast|^2}\int_{E(v,v_\ast')}\na_v^{k_1}f(v')\chi_0(v')B^\ast_0 \,d\Pi_{v'}dv_\ast'\\
&+w_l\chi_{M}\sum\limits_{k_1\leq k}C_k^{k_1}\int_{\R^3}\frac{\na_v^{k-k_1}\mu(v_\ast')
\chi_0(v_\ast')}{|v-v'_\ast|^2}\int_{E(v,v_\ast')}\na_v^{k_1}f(v')\chi_1(v')B^\ast_0 \,d\Pi_{v'}dv_\ast'\\
\eqdef& \sum\limits_{n=1}^3\CI_{2,n}.
\end{split}
\end{equation*}
Note that the term simultaneously involving $\mu \chi_1$ and $f\chi_1$ has vanished due to the fact that $|v'|^2+|v'_\ast|^2=|v|^2+|v_\ast|^2.$
We now turn to estimate $\CI_{2,n}$$(1\leq n\leq 3)$ term by term. First of all, the direct computation gives us to
\begin{align}\label{WW}
\na_v^{k-k_1}\mu(v_\ast')
\leq C\mu^{\frac{1}{2}}(v_\ast'),\
w_{l}\chi_{|v|\geq M}\mu^{\frac{1}{4}}(v_\ast')
\chi_0(v_\ast')\leq Ce^{-\frac{M^2}{32}}.
\end{align}
Moreover, standard calculation yields
\begin{align}\label{st.ca}
\int_{\R^3}\frac{\mu^{\frac{1}{4}}(v_\ast')}{|v-v'_\ast|^2}\,dv_\ast'\leq C\lag v\rag^{-2}.
\end{align}
%provided $\rho<M.$
By using \eqref{WW} and \eqref{st.ca}, one sees that for $l>\frac{3}{2}$,
\begin{align}\label{CI213}
\CI_{2,1},\ \CI_{2,3}\leq Ce^{-\frac{M^2}{32}}\lag v\rag^{-2}\sum\limits_{k_1\leq k}\|w_l\na_v^{k_1}f\|_{L^\infty}\int_{\R^2}w_{-l}(v')\,d\Pi v'\leq Ce^{-\frac{M^2}{32}}\sum\limits_{k_1\leq k}\|w_l\na_v^{k_1}f\|_{L^\infty}.
\end{align}
It remains now to estimate the delicate term $\CI_{2,2}$ where the smallness is hard to be obtained. As \cite[Proposition 3.1, pp.397]{AEP-87}, we introduce
the following two cutoff functions
\begin{align*}
\chi_\de(v_\ast')=\left\{\begin{array}{rll}
&1,\ |v_\ast'|<\de|v|,\\[2mm]
&0,\ \textrm{otherwise},
\end{array}\right.\quad  \chi_\eta(v_\ast)=\left\{\begin{array}{rll}
&1,\ |v_\ast|<\eta|v|,\\[2mm]
&0,\ \textrm{otherwise},
\end{array}\right.
\end{align*}
where $0<\de<\eta<1.$ Then we spilt $\CI_{2,2}$ as
\begin{equation*}%\label{I22.dc}
\begin{split}
\CI_{2,2}=&w_l\chi_{M}\sum\limits_{k_1\leq k}C_k^{k_1}\int_{\R^3}\frac{\na_v^{k-k_1}\mu(v_\ast')
\chi_1(v_\ast')(1-\chi_\de(v_\ast'))}{|v-v'_\ast|^2}\int_{E(v,v_\ast')}\na_v^{k_1}f(v')\chi_0(v')B^\ast_0 \,d\Pi_{v'}dv_\ast'\\
&+w_l\chi_{M}\sum\limits_{k_1\leq k}C_k^{k_1}\int_{\R^3}\frac{\na_v^{k-k_1}\mu(v_\ast')
\chi_1(v_\ast')\chi_\de(v_\ast')}{|v-v'_\ast|^2}\int_{E(v,v_\ast')}\na_v^{k_1}f(v')
\chi_0(v')\chi_\eta(v_\ast)B^\ast_0 \,d\Pi_{v'}dv_\ast'\\
&+w_l\chi_{M}\sum\limits_{k_1\leq k}C_k^{k_1}\int_{\R^3}\frac{\na_v^{k-k_1}\mu(v_\ast')
\chi_1(v_\ast')\chi_{\de}(v_\ast')}{|v-v'_\ast|^2}\\&\qquad\qquad\qquad\qquad\times\int_{E(v,v_\ast')}\na_v^{k_1}f(v')\chi_0(v')(1-\chi_\eta(v_\ast)) B^\ast_0 \,d\Pi_{v'}dv_\ast'\\ \eqdef& \sum\limits_{n=1}^3\CI^n_{2,2}.
\end{split}
\end{equation*}
Performing the similar calculations as for obtaining \eqref{CI213}, one has
\begin{align}\label{CI221}
\CI^1_{2,2}\leq Ce^{-\frac{\de^2M^2}{16}}\sum\limits_{k_1\leq k}\|w_l\na_v^{k_1}f\|_{L^\infty}\int_{\R^2}w_{-l}(v')\,d\Pi v'\leq Ce^{-\frac{\de^2M^2}{16}}\sum\limits_{k_1\leq k}\|w_l\na_v^{k_1}f\|_{L^\infty}.
\end{align}
For $\CI^2_{2,2}$, we first have that if $|v_\ast'|<\de|v|$ and $|v_\ast|<\eta|v|$, then
\begin{align}\notag
|v'-v|=|v_\ast-v_\ast'|\leq(\eta+\de)|v|,\quad |v'|=|v|^2+|v_\ast|^2-|v_\ast'|^2\geq(1-\de^2)|v|^2,
\end{align}
which further implies that the measure of $E(v,v_\ast')$ is bounded as 
$$
|E(v,v_\ast')|\leq \pi(\eta+\de)^2|v|^2\leq 4\pi\eta^2|v|^2,
$$
and it holds true that 
$$
(1+|v'|^2)^{-l}\leq (1+(1-\de^2)|v|^2)^{-l}\leq(1+|v|^2)^{-l}(1-\de^2)^{-l}.
$$
Consequently, applying \eqref{WW} and \eqref{st.ca} again, we obtain
\begin{align}\label{CI222}
\CI^2_{2,2}\leq &C w_l\chi_{M}\sum\limits_{k_1\leq k}\int_{\R^3}\frac{\mu^{\frac{1}{2}}(v_\ast')}{|v-v_\ast'|^2}
\int_{E(v,v_\ast')}\na_v^{k_1}f(v')\chi_0(v')\chi_\eta(v_\ast)\,d\Pi_{v'}dv_\ast'\notag\\
\leq& Cw_l(v)\chi_{M}\lag v\rag^{-2}\sum\limits_{k_1\leq k}\|w_l\na_v^{k_1}f\|_{L^\infty}
\eta^2|v|^2(1+|v|^2)^{-l}(1-\de^2)^{-l}\notag\\
\leq &C\eta^2(1-\de^2)^{-l}\sum\limits_{k_1\leq k}\|w_l\na_v^{k_1}f\|_{L^\infty}.
\end{align}
We are now in a position to compute the last term $\CI^3_{2,2}$. %First of all, we denote $|\bar{v}|=\min\{|v'|,v'\in E(v,v_\ast')\}$.
Since for the case of $\CI^3_{2,2}$, we have $|v_\ast'|<\de|v|$ and $|v_\ast|\geq\eta|v|$, then it follows that 
\begin{align}\notag
|v'|^2=|v|^2+|v_\ast|^2-|v_\ast'|^2\geq |v|^2+\eta^2|v|^2-\de^2|v|^2=(1+\eta^2-\de^2)|v|^2,
\end{align}
which implies 
\begin{align}\label{CI223}
\CI^3_{2,2}\leq &C w_l\chi_{M}\lag v\rag^{-2}\sum\limits_{k_1\leq k}\|w_l\na_v^{k_1}f\|_{L^\infty}\int^{+\infty}_{|v|\sqrt{1+\eta^2-\de^2}}\frac{r}{(1+r^2)^l}\,dr\notag\\
\leq& C\sum\limits_{k_1\leq k}\|w_l\na_v^{k_1}f\|_{L^\infty}\chi_{M}\lag v\rag^{-2}w_l(v)\frac{1}{l-1}\left(1+(1+\eta^2-\de^2)|v|^2\right)^{-l+1}\notag\\
\leq &\frac{C}{l}\sum\limits_{k_1\leq k}\|w_l\na_v^{k_1}f\|_{L^\infty},
\end{align}
where the last inequality holds due to the fact that $1+\eta^2-\de^2\geq1$ and $l\gg1.$

Therefore, putting \eqref{CI221}, \eqref{CI222} and \eqref{CI223} together, we arrive at
\begin{align}\label{CI22.p2}
\CI_{2,2}\leq C\left\{e^{-\frac{\de^2M^2}{32}}+\eta^2(1-\de^2)^{-l}
+\frac{1}{l}\right\}\sum\limits_{k_1\leq k}\|w_l\na_v^{k_1}f\|_{L^\infty}.
\end{align}
Furthermore, if one chooses $\de=\frac{1}{l}$, $\eta=\frac{1}{\sqrt{l}}$ and $M= l^2\gg1,$
then
\begin{align}\label{CI22.p3}
(1-\de^2)^{-l}<e,\ e^{-\frac{\de^2M^2}{32}}\leq e^{-\frac{M}{32}}.
\end{align}
As a consequence, \eqref{CI22.p2} and \eqref{CI22.p3}
give rise to
\begin{align}\label{CI22.p4}
\CI_{2,2}\leq \left\{\frac{1}{M}+\frac{2}{l}\right\}\sum\limits_{k_1\leq k}\|w_l\na_v^{k_1}f\|_{L^\infty}.
\end{align}
%provided $\de=\frac{1}{l}$, $\eta=\frac{1}{\sqrt{l}}$ and $M= l^2\gg1.$
Finally, the desired estimate \eqref{CK1} follows from \eqref{CI1}, \eqref{CI213} and \eqref{CI22.p4}. This ends the proof of Lemma \ref{CK}.
\end{proof}

%\newpage
\section{Steady problem}\label{sta.prob}

This section is devoted to studying the steady problem
\begin{equation}
\label{sp}
-\beta \na_v\cdot (vG)-\al v_2 \pa_{v_1} G=Q(G,G)
\end{equation}
with 
\begin{equation}
\label{def.betac}
\beta=-\frac{\al}{3}\int_{\R^3} v_1v_2G\,dv,
\end{equation}
where the solution $G(v)$ is required to satisfy 
\begin{equation}
\label{cls1}
\int_{\R^3} G \,dv=1,\quad
\int_{\R^3}v_i G \,dv=0,i=1,2,3,\quad
\int_{\R^3} |v|^2 G \,dv=3,
\end{equation}
that are equivalent with the fact that $G$ has the same fluid quantities as $\mu$ in \eqref{def.mu} for any $\al>0$. Note that through the paper we have omitted the dependence of $G$ on the parameter $\al$. 

Since one expects $G\to \mu$ as $\al\to 0$,
to look for the solution, let us first set
\begin{equation}\label{Gexp}
G=\mu+\al\sqrt{\mu}\left\{G_1+G_R\right\},
\end{equation}
with $\FP_0G_1=0$ and $\FP_0G_R=0 $ such that \eqref{cls1} holds true, i.e.
\begin{eqnarray}\label{con.g1r}
\left\{\begin{array}{rll}
\begin{split}
&\int_{\R^3}G_1\sqrt{\mu}\,dv=\int_{\R^3}G_R\sqrt{\mu}\,dv=0,\\
&\int_{\R^3}G_1v_i\sqrt{\mu}\,dv=\int_{\R^3}G_Rv_i\sqrt{\mu}\,dv=0,\ i=1,2,3,\\
&\int_{\R^3}G_1|v|^2\sqrt{\mu}\,dv=\int_{\R^3}G_R|v|^2\sqrt{\mu}\,dv=0,
\end{split}
\end{array}\right.
\end{eqnarray}
where $G_1$ accounts for the first order correction and $G_R$ denotes the higher order remainder.  
We now turn to determine $G_1$ and derive the equation of the remainder $G_R$. In fact, plugging \eqref{Gexp} into  \eqref{def.betac} gives
\begin{equation}\label{beta.sum}
\beta=-\frac{\al}{3}\int_{\R^3} v_1v_2 G\,dv=-\frac{\al^2}{3}\int_{\R^3} v_1v_2\mu^{1/2} (G_1+G_R)\,dv,
%=\frac{\al^2}{6b_0} -\frac{\al^2}{3}\int_{\R^3} v_1v_2\mu^{1/2}G_R\,dv,
\end{equation}
which implies that $\beta$ is at least the second order of $\al$. Therefore, substituting \eqref{Gexp} into \eqref{sp}, one can write
\begin{align}
-\frac{\bet}{\al}&\mu^{-\frac{1}{2}}\na_v\cdot(v\mu)-\beta\mu^{-\frac{1}{2}}\na_v\cdot(v\mu^{\frac{1}{2}}(G_1+G_R))
+v_1v_2\mu^{\frac{1}{2}}-\al\mu^{\frac{1}{2}}v_2\pa_{v_1}(\mu^{\frac{1}{2}}(G_1+G_R))\notag
\\&+LG_1+LG_R
=\al\Ga(G_1,G_1)+\al\{\Ga(G_R,G_1)+\Ga(G_1,G_R)\}+\al\Ga(G_R,G_R).\label{gr0}
\end{align}
To remove the zero order term from \eqref{gr0}, we set 
\begin{align*}%\label{G1}
G_1=-L^{-1}\{v_1v_2\mu^{\frac{1}{2}}\},
%=-\frac{1}{2b_0}v_1v_2\mu^{\frac{1}{2}},
\end{align*}
where we have noticed $v_1v_2\mu^{\frac{1}{2}}\in (\ker L)^\perp$ so that $G_1$ is well-defined and $G_1\in (\ker L)^\perp$ is purely microscopic, satisfying \eqref{con.g1r}. Moreover, it follows from Lemma \ref{our.lem} that 
\begin{align}\label{G1}
G_1=-\frac{1}{2b_0}v_1v_2\mu^{\frac{1}{2}}
\end{align}
with the constant $b_0>0$ defined in \eqref{b0}. Then, \eqref{gr0} is further reduced to
\begin{align}
\beta\mu^{-\frac{1}{2}}&\na_v\cdot(v\mu^{\frac{1}{2}}G_R)
-\al\mu^{-\frac{1}{2}}v_2\pa_{v_1}(\mu^{\frac{1}{2}}G_R)
+LG_R\notag\\
=&\frac{\bet}{\al}\mu^{-\frac{1}{2}}\na_v\cdot(v\mu)+\beta\mu^{-\frac{1}{2}}\na_v\cdot(v\mu^{\frac{1}{2}}G_1)+
\al\mu^{-\frac{1}{2}}v_2\pa_{v_1}(\mu^{\frac{1}{2}}G_1)\notag\\&+
\al\Ga(G_1,G_1)+\al\{\Ga(G_R,G_1)+\Ga(G_1,G_R)\}+\al\Ga(G_R,G_R),\label{gr}
\end{align}
%where
%$$
%\beta=\frac{1}{6b_0}\al^2-\frac{\al^2}{3}\int v_1v_2\mu^{1/2}G_Rdv.
%$$
and in light of \eqref{G1}, $\beta$ in \eqref{beta.sum} is given as \begin{equation}\label{def.beta0}
\beta=\beta^0 -\frac{\al^2}{3}\int_{\R^3} v_1v_2\mu^{1/2}G_R\,dv,
\end{equation}
where for later use we have denoted 
$$
\beta^0=-\frac{\al^2}{3}\int_{\R^3} v_1v_2\mu^{1/2}G_1\,dv=\frac{\al^2}{6b_0}>0.
$$
To solve \eqref{gr} on $G_R$, it is necessary to use the decomposition
\begin{align*}%\label{def.demgr}
\sqrt{\mu}G_R=G_{R,1}+\sqrt{\mu}G_{R,2},
\end{align*}
where %, in addition to \eqref{con.g1r}, 
$G_{R,1}$ and $G_{R,2}$ are supposed to satisfy
\begin{align}
-\beta\na_v\cdot(vG_{R,1})&
-\al v_2\pa_{v_1}G_{R,1}+\frac{\beta}{2}|v|^2\mu^{\frac{1}{2}}G_{R,2}+\al\frac{v_1v_2}{2}\mu^{\frac{1}{2}}G_{R,2}
+\nu_0G_R-\chi_{M}\CK G_{R,1}\notag\\
=&\frac{\bet}{\al}\na_v\cdot(v\mu)+\beta\na_v\cdot(v\mu^{\frac{1}{2}}G_1)
+\al v_2\pa_{v_1}(\mu^{\frac{1}{2}}G_1)+
\al Q(\mu^{\frac{1}{2}}G_1,\mu^{\frac{1}{2}}G_1)\notag\\&+\al\{Q(\mu^{\frac{1}{2}}G_R,\mu^{\frac{1}{2}}G_1)
+Q(\mu^{\frac{1}{2}}G_1,\mu^{\frac{1}{2}}G_R)\}+\al Q(\mu^{\frac{1}{2}}G_R,\mu^{\frac{1}{2}}G_R),\label{gr1}
\end{align}
and
\begin{align}\label{gr2}
-\beta&\na_v\cdot(vG_{R,2})
-\al v_2\pa_{v_1}G_{R,2}
+LG_{R,2}-(1-\chi_{M})\mu^{-\frac{1}{2}}\CK G_{R,1}=0,
\end{align}
respectively. Here, we recall that $\nu_0$ and $\CK$ are defined in \eqref{sp.cL}. Moreover, in order for $G_R$ to satisfy \eqref{con.g1r}, we require
\begin{eqnarray}\label{con.g12}
\left\{\begin{array}{rll}
\begin{split}
&\int_{\R^3}G_{R,1}\,dv+\int_{\R^3}\sqrt{\mu}G_{R,2}\,dv=0,\\
&\int_{\R^3}v_iG_{R,1}\,dv+\int_{\R^3}v_i\sqrt{\mu}G_{R,2}\,dv=0,\ i=1,2,3,\\
&\int_{\R^3}|v|^2G_{R,1}\,dv+\int_{\R^3}|v|^2\sqrt{\mu}G_{R,2}\,dv=0.
\end{split}
\end{array}\right.
\end{eqnarray}

The existence of \eqref{gr1} and \eqref{gr2} under the restriction condition \eqref{con.g12} will be established through the approximation solution sequence by iteratively solving the following equations
\begin{eqnarray}\label{gr12.ls}
\left\{\begin{array}{rll}
\begin{split}
\eps G^{n+1}_{R,1}&-\beta^{n}\na_v\cdot (vG^{n+1}_{R,1})-\al v_2\pa_{v_1}G^{n+1}_{R,1}
+\nu_0G^{n+1}_{R,1}
-\chi_{M}\CK G^{n+1}_{R,1}\\&+\frac{\beta^n}{2}|v|^2\mu^{\frac{1}{2}}G^{n+1}_{R,2}+\al\frac{v_1v_2}{2}\mu^{\frac{1}{2}}G^{n+1}_{R,2}
-\frac{\bet^{n+1}-\frac{\al^2}{6b_0}}{\al}\na_v\cdot(v\mu)\\=&\beta^n\na_v\cdot(v\mu^{\frac{1}{2}}G_1)
+\frac{\al}{6b_0}\na_v\cdot(v\mu)+\al v_2\pa_{v_1}(\mu^{\frac{1}{2}}G_1)+
\al Q(\mu^{\frac{1}{2}}G_1,\mu^{\frac{1}{2}}G_1)\\&+\al\{Q(\mu^{\frac{1}{2}}G^{n}_R,\mu^{\frac{1}{2}}G_1)
+Q(\mu^{\frac{1}{2}}G_1,\mu^{\frac{1}{2}}G^{n}_R)\}+\al Q(\mu^{\frac{1}{2}}G^n_R,\mu^{\frac{1}{2}}G^n_R),\\
\eps G^{n+1}_{R,2}&-\beta^n\na_v\cdot (v \na_vG^{n+1}_{R,2})-\al v_2\pa_{v_1}G^{n+1}_{R,2}
+LG^{n+1}_{R,2}-(1-\chi_{M})\mu^{-\frac{1}{2}}\CK G^{n+1}_{R,1}=0,
\end{split}
\end{array}\right.
\end{eqnarray}
for a small parameter $\eps>0$, where we have denoted 
\begin{align}\label{n-sq}
\mu^{\frac{1}{2}}G^{n}_R=G^{n}_{R,1}+\mu^{\frac{1}{2}}G^{n}_{R,2},\quad \beta^n=\beta^0-\frac{\al^2}{3}\int_{\R^3} v_1v_2(G^{n}_{R,1}+\mu^{\frac{1}{2}}G^{n}_{R,2})\,dv,\quad n\geq0,
\end{align}
the constant $\beta^0$ is defined in \eqref{def.beta0}, and we set 
$$
G^{0}_{R,1}=G^{0}_{R,2}=0.
$$
For brevity we have omitted the explicit dependence of the approximation solution sequence $\{[G_{R,1}^n,G_{R,2}^n]\}_{n=1}^\infty$ on $\eps$. Note that we put the penalty terms $\eps G^{n+1}_{R,i}$$(i=1,2)$ on the left hand side of \eqref{gr12.ls} so as to guarantee the mass conservation in \eqref{con.g1r}.  In addition, since it holds that
\begin{align}
\lag \frac{\al}{6b_0}\na_v\cdot(v\mu)+\al v_2\pa_{v_1}(\mu^{\frac{1}{2}}G_1),|v|^2\mu^{\frac{1}{2}}\rag=0,\notag
\end{align}
and
\begin{align}
\lag \frac{\bet^{n+1}-\frac{\al^2}{6b_0}}{\al}\na_v\cdot(v\mu)+\al v_2\pa_{v_1}G^{n+1}_{R,1},|v|^2\rag
+\lag \al v_2\pa_{v_1}G^{n+1}_{R,2},|v|^2\mu^{\frac{1}{2}}\rag=0,\notag
\end{align}
one sees
that
\begin{align}\label{abr12n}
\lag G^{n+1}_{R,1},[1,v_i,|v|^2]\rag+\lag G^{n+1}_{R,2},[1,v_i,|v|^2]\mu^{\frac{1}{2}}\rag=0,\ i=1,2,3,
\end{align}
for any $\eps>0.$
%provided the summation $\frac{G^{n}_{R,1}}{\sqrt{\mu}}+G^{n}_{R,2}$ are orthogonal to $\ker(L).$

We first show that in an appropriate function space there exists a solution $[G_{R,1},G_{R,2}]$ satisfying
\begin{align}\label{abr12}
\lag G_{R,1},[1,v_i,|v|^2]\rag+\lag G_{R,2},[1,v_i,|v|^2]\mu^{\frac{1}{2}}\rag=0,\ i=1,2,3
\end{align}
to the coupled linear system corresponding to \eqref{gr12.ls}.
To do so, let us first define the following linear operator parameterized by $\si\in[0,1]$ (cf.~\cite{DHWZ-19}):
$$
\mathscr{L}_\si[\CG_1,\CG_2]=[\mathscr{L}^1_\si,\mathscr{L}^2_\si][\CG_1,\CG_2],
$$
where
\begin{eqnarray*}%\label{laop}
\left\{\begin{array}{rll}
\begin{split}
\mathscr{L}^1_{\si}[\CG_1,\CG_2]=&\eps\CG_1-\beta'\na_v\cdot (v\CG_1)-\al v_2\pa_{v_1}\CG_1
+\nu_0\CG_1
-\si\chi_{M}\CK\CG_1\\
&+\frac{\beta'}{2}|v|^2\sqrt{\mu}\CG_2
+\al\frac{v_1v_2}{2}\sqrt{\mu}\CG_2-\frac{\bet''(\CG)}{\al}\na_v\cdot(v\mu),\\
\mathscr{L}^2_\si[\CG_1,\CG_2]=&\eps\CG_2-\beta'\na_v\cdot (v\CG_2)-\al v_2\pa_{v_1}\CG_2
+\nu_0\CG_2-\si K\CG_2-\si(1-\chi_{M})\mu^{-\frac{1}{2}}\CK \CG_1.
\end{split}
\end{array}\right.
\end{eqnarray*}
Here $K$ is defined as \eqref{sp.L}, $\beta'$ is a given constant satisfying $\beta'\sim\al^2$, and
\begin{equation}\label{def.betaG}
\beta''(\CG)=-\frac{\al^2}{3}\int_{\R^3} v_1v_2(\CG_1+\mu^{\frac{1}{2}}\CG_2)\,dv.
\end{equation}
Then we consider the solvability of the general coupled linear system
\begin{align}\label{pals}
\left\{\begin{array}{rll}
&\mathscr{L}^1_{\si}[\CG_1,\CG_2]=\CF_1,\\[2mm]
&\mathscr{L}^2_\si[\CG_1,\CG_2]=\CF_2,
 \end{array}\right.
\end{align}
where $\CF_1$ and $\CF_2$ are given sources satisfying
\begin{eqnarray}
\left\{\begin{array}{rll}
 &\lag \CF_1,[1,v_i,|v|^2]\rag+\lag \CF_2,[1,v_i,|v|^2]\mu^{\frac{1}{2}}\rag=0,\ i=1,2,3,\\[3mm]
&\|w_{l}\na_v^k\CF_1\|_{L^\infty}+\|w_{l}\na_v^k\CF_2\|_{L^\infty}<+\infty,\ \textrm{for any}\ k\geq0.\label{F12}
 \end{array}\right.
\end{eqnarray}

In what follows, we look for solutions to the system \eqref{pals} in the Banach space
\begin{align*}
\FX_{\al,m}=\bigg\{[\CG_1,\CG_2]&\big|\sum\limits_{0\leq k\leq m}\|w_{l}\na_v^k[\CG_1,\CG_2]\|_{L^\infty}<+\infty,
%\ \textrm{for any}\ m\geq k\geq0,\notag
\\
&\quad\lag \CG_1,[1,v_i,|v|^2]\rag+\lag \CG_2,[1,v_i,|v|^2]\mu^{\frac{1}{2}}\rag=0,\ i=1,2,3
\bigg\},\notag
\end{align*}
associated with the norm
$$
\|[\CG_1,\CG_2]\|_{\FX_{\al},m}=\sum\limits_{0\leq k\leq m}\left\{\|w_{l}\na_v^k\CG_1\|_{L^\infty}+\|w_{l}\na_v^k\CG_2\|_{L^\infty}\right\}.
$$
Let us now deduce the {\it a priori} estimate for the parameterized linear system \eqref{pals}.

\begin{lemma}[{\it a priori} estimate]\label{lifpri}
Let $[\CG_1,\CG_2]\in\FX_{\al,m}$ with $\al>0$ and $m\geq 0$ be a solution to \eqref{pals} with $\eps>0$ suitably small, $\si\in[0,1]$ and $[\CF_1,\CF_2]$ satisfying \eqref{F12}. There is $l_0>0$ such that for any $l\geq l_0$ arbitrarily large, there are $\al_0=\al_0(l)>0$ and large $M=M(l)>0$ such that for any $0<\al<\al_0$,  the solution $[\CG_1,\CG_2]$ satisfies the following estimate 
\begin{align}\label{Lif.es1}
\|[\CG_1,\CG_2]\|_{\FX_{\al,m}}=\|\mathscr{L}_\si^{-1}[\CF_1,\CF_2]\|_{\FX_{\al,m}}\leq C_\mathscr{L}\sum\limits_{0\leq k\leq m}\left\{\|w_{l}\na_v^k\CF_1\|_{L^\infty}+\|w_{l}\na_v^k\CF_2\|_{L^\infty}\right\},
\end{align}
where the constant $C_\mathscr{L}>0$ is independent of $\si$, $\eps$ and $\al$.
\end{lemma}

\begin{proof}
The proof is divided into two steps.

\medskip
\noindent\underline{Step 1. $L^\infty$ estimates.}  Taking $0\leq k\leq m$ and $l$, we set $H_{1,k}=w_{l}\na_v^k\CG_1$ and $H_{2,k}=w_{l}\na_v^k\CG_2$. Then, $\FH_k=[H_{1,k},H_{2,k}]$ satisfies the following equations:
\begin{align}
\eps H_{1,k}&-\bet' \na_v\cdot(v H_{1,k})+2l \bet' \frac{|v|^2}{1+|v|^2} H_{1,k}
-\al v_2\pa_{v_1}H_{1,k}+2l \al \frac{v_2v_1}{{1+|v|^2}}H_{1,k}
+\nu_0H_{1,k}\notag\\
&-\si\chi_{M}w_l\CK \left(\frac{H_{1,k}}{w_{l}}\right)
-w_l\frac{\bet''(\frac{\FH_{0}}{w_l})}{\al}\na_v^k\na_v\cdot(v\mu)\notag\\
= &{\bf1}_{|\ga'|=1}w_{l}\bet' C_\ga^{\ga'}\na_v\cdot(\pa^{\ga'}v\pa_v^{\ga-\ga'}\CG_1)
+{\bf1}_{\ga'=(0,1,0)}\al C_\ga^{\ga'}w_l\pa_{v_1}\pa_v^{\ga-\ga'}\CG_1
\notag\\&-\frac{\beta'}{2}w_l\sum\limits_{\ga'\leq\ga}C_\ga^{\ga'}\pa_v^{\ga'}(|v|^2\mu^{\frac{1}{2}})\pa_v^{\ga-\ga'}\CG_2
-\frac{\al}{2}\sum\limits_{\ga'\leq\ga}w_{l} C_\ga^{\ga'}\pa_v^{\ga'}\left(v_1v_2\sqrt{\mu}\right)\pa_v^{\ga-\ga'}\CG_2
\notag\\&+\si\sum\limits_{0<\ga'\leq\ga}C_\ga^{\ga'}w_{l}(\pa_v^{\ga'}(\chi_{M}\CK)) \left(\pa_v^{\ga-\ga'}\CG_1\right)+w_{l}\na_v^k\CF_1,\label{H10}
\end{align}
and
\begin{align}
\eps H_{2,k}&-\bet' \na_v\cdot(v H_{2,k})+2l \bet' \frac{|v|^2}{1+|v|^2} H_{2,k}
-\al v_2\pa_{v_1}H_{2,k}+2l \al \frac{v_2v_1}{{1+|v|^2}}H_{2,k}
\notag\\
&+\nu_0 H_{2,k}-\si w_{l}K\left(\frac{H_{2,k}}{w_{l}}\right)
\notag\\=&{\bf1}_{|\ga'|=1}w_{l}\bet' C_\ga^{\ga'}\na_v\cdot(\pa^{\ga'}v\pa_v^{\ga-\ga'}\CG_2)
+{\bf1}_{\ga'=(0,1,0)}\al C_\ga^{\ga'}w_{l}\pa_{v_1}\pa_v^{\ga-\ga'}\CG_2
\notag\\&+\si w_{l}\sum\limits_{0<\ga'\leq\ga}C_\ga^{\ga'}(\pa^{\ga'}K)\left(\pa^{\ga-\ga'}_v\CG_2\right)\notag\\
&+\si\sum\limits_{\ga'\leq\ga}C_\ga^{\ga'}w_{l}\pa_v^{\ga'}((1-\chi_{M})\mu^{-\frac{1}{2}}\CK) \left(\pa_v^{\ga-\ga'}\CG_1\right)+w_{l}\na_v^k\CF_2,\label{H20}
\end{align}
where $\FH_0\eqdef[H_1,H_2]=[H_{1,0},H_{2,0}]=w_{l}[\CG_1,\CG_2].$
The characteristic method will be employed to construct the existence of solutions to \eqref{H10} and \eqref{H20} in $L^\infty$ space (cf.~\cite{EGKM-13}). To do so, we first introduce a uniform parameter $t$, and regard $H_{i,k}(v)=H_{i,k}(t,v)$$(i=1,2)$, then define the characteristic line $[s,V(s;t,v)]$ for both the equations \eqref{H10} and \eqref{H20} going through $(s,v)$ such that
\begin{align}\label{H1CL}
\left\{\begin{array}{rll}
&\frac{d V_1}{ds}=-\beta' V_1(s;t,v)-\al V_2(s;t,v),\\[2mm]
&\frac{d V_i}{ds}=-\beta' V_i(s;t,v),\ i=2,3,\\[2mm]
&V(t;t,v)=v,
\end{array}\right.
\end{align}
which is equivalent to
\begin{equation*}%\label{H1CLp1}
\begin{split}
&V_1(s;t,v)=e^{\beta'(t-s)}(v_1+\al v_2(t-s)),\\[2mm]
&V_i(s;t,v)=e^{\beta'(t-s)}v_i,\ \ i=2,3.
\end{split}
\end{equation*}
Integrating along the backward trajectory \eqref{H1CL}, one can write the solutions of \eqref{H10} and \eqref{H20} as the mild form of
\begin{align}\label{H10m}
H_{1,k}
=&e^{-\int_{0}^t\CA^\eps(\tau,V(\tau))d\tau}H_{1,k}(V(0))+\si\int_0^{t}e^{-\int_{s}^t\CA^\eps(\tau,V(\tau))d\tau}\left\{\chi_{M}w_l\CK
\left(\frac{H_{1,k}}{w_{l}}\right)\right\}(V(s))\,ds\notag
\\&-\int_0^{t}e^{-\int_{s}^t\CA^\eps(\tau,V(\tau))d\tau}
\left\{w_l\frac{\bet''(\frac{\FH_0}{w_l})}{\al}\na_v^k\na_v\cdot(v\mu)\right\}(V(s))\,ds
\notag\\
&+\int_0^{t}e^{-\int_{s}^t\CA^\eps(\tau,V(\tau))d\tau}
\bigg\{{\bf1}_{|\ga'|=1}w_{l}\bet' C_\ga^{\ga'}\na_v\cdot(\pa^{\ga'}v\pa_v^{\ga-\ga'}\CG_1)
\notag\\&\qquad\qquad
+{\bf1}_{\ga'=(0,1,0)}\al C_\ga^{\ga'}w_l\pa_{v_1}\pa_v^{\ga-\ga'}\CG_1\bigg\}(V(s))\,ds
\notag\\
&-\int_0^{t}e^{-\int_{s}^t\CA^\eps(\tau,V(\tau))d\tau}\bigg\{\frac{\beta'}{2}
w_l\sum\limits_{\ga'\leq\ga}C_\ga^{\ga'}\pa_v^{\ga'}(|v|^2\mu^{\frac{1}{2}})\pa_v^{\ga-\ga'}\CG_2
\notag\\&\qquad\qquad+\frac{\al}{2}\sum\limits_{\ga'\leq\ga}w_{l} C_\ga^{\ga'}\pa_v^{\ga'}\left(v_1v_2\sqrt{\mu}\right)\pa_v^{\ga-\ga'}\CG_2\bigg\}(V(s))\,ds\notag\\
&+\si\int_0^{t}e^{-\int_{s}^t\CA^\eps(\tau,V(\tau))d\tau}\left\{{\bf1}_{|\ga|\geq1}\sum\limits_{0<\ga'\leq\ga}C_\ga^{\ga'}w_{l}(\pa_v^{\ga'}(\chi_{M}\CK)) \left(\pa_v^{\ga-\ga'}\CG_1\right)\right\}(V(s))\,ds\notag\\
&+\int_0^{t}e^{-\int_{s}^t\CA^\eps(\tau,V(\tau))d\tau}\left(w_{l}\na_v^k\CF_1\right)(V(s))\,ds\eqdef \sum\limits_{i=1}^7\CI_{i},
\end{align}
and
\begin{align}\label{H20m}
H_{2,k}
=&e^{-\int_{0}^t\CA^\eps(\tau,V(\tau))d\tau}H_{2,k}(V(0))+\si\int_0^{t}e^{-\int_{s}^t\CA^\eps(\tau,V(\tau))d\tau}\left[w_{l}K\left(\frac{H_{2,k}}{w_{l}}\right)\right](V(s))\,ds
\notag\\&
+\int_0^{t}e^{-\int_{s}^t\CA^\eps(\tau,V(\tau))d\tau}\bigg\{{\bf1}_{|\ga'|=1}w_{l}\bet' C_\ga^{\ga'}\na_v\cdot(\pa^{\ga'}v\pa_v^{\ga-\ga'}\CG_2)
\notag\\&\qquad\qquad+{\bf1}_{\ga'=(0,1,0)}\al C_\ga^{\ga'}w_{l}\pa_{v_1}\pa_v^{\ga-\ga'}\CG_2\bigg\}(V(s))\,ds\notag\\&
+\si\int_0^{t}e^{-\int_{s}^t\CA^\eps(\tau,V(\tau))d\tau}\left\{{\bf1}_{|\ga|\geq1}
w_{l}\sum\limits_{0<\ga'\leq\ga}C_\ga^{\ga'}(\pa^{\ga'}K)\left(\pa^{\ga-\ga'}_v\CG_2\right)\right\}(V(s))\,ds
\notag\\&+\si\int_0^{t}e^{-\int_{s}^t\CA^\eps(\tau,V(\tau))d\tau}
\left\{w_{l}\sum\limits_{\ga'\leq\ga}C_\ga^{\ga'}\pa_v^{\ga'}
((1-\chi_{M})\mu^{-\frac{1}{2}}\CK)\left(\pa_v^{\ga-\ga'}\CG_1\right)\right\}(V(s))\,ds \notag\\&
+\int_0^{t}e^{-\int_{s}^t\CA^\eps(\tau,V(\tau))d\tau}\left(w_{l}\na_v^k\CF_2\right)(V(s))\,ds
\eqdef \sum\limits_{i=8}^{13}\CI_{i},
\end{align}
where
\begin{align}\label{sglw}
\CA^\eps(\tau,V(\tau))=\nu_0+\eps-3\beta'+2l \bet' \frac{|V(\tau)|^2}{1+|V(\tau)|^2} +2l \al \frac{V_2(\tau)V_1(\tau)}{{1+|V(\tau)|^2}}\geq\nu_0/2,
\end{align}
provided that $\eps>0$, $\beta'\sim \al^2$, $l\beta'\sim l\al^2$, and $l\al $ is suitably small.
In what follows, we will estimate $\CI_i$$(1\leq i\leq13)$ term by term. %The estimations will be divided into two case.

%\noindent{\it Case 1. $k=0$.} In this case, the terms $I_4, I_5, I_6$, $I_{10}, I_{11}$ and $I_{12}$ vanish.
Since the parameter $t$ here is arbitrary, we may take $t$
sufficiently large such that
\begin{align*}
e^{-\int_{0}^t\CA^\eps(\tau,V(\tau))\,d\tau}\leq e^{-\frac{\nu_0t}{2}}\leq\frac{1}{8},
\end{align*}
from which, one sees that
\begin{align}\notag
\CI_1\leq \frac{1}{8}\|H_{1,k}\|_{L^\infty},\ \CI_8\leq \frac{1}{8}\|H_{2,k}\|_{L^\infty}.
\end{align}
Next, Lemma \ref{CK} and \eqref{sglw} give that
\begin{align}\notag
\CI_2\leq \frac{C}{l }\|H_{1,k}\|_{L^\infty}\int_{0}^te^{-\frac{\nu_0}{2}(t-s)}\,ds\leq \frac{C}{l }\|H_{1,k}\|_{L^\infty}.
\end{align}
In view of \eqref{def.betaG}, one has
\begin{align}\notag
\CI_3\leq C\al\|H_{1,0}\|_{L^\infty}+C\al\|H_{2,0}\|_{L^\infty}.
\end{align}
It is straightforward to see that
\begin{align}\notag
 \CI_4\leq C\al\sum\limits_{k'\leq k}\|H_{1,k'}\|_{L^\infty},\ \CI_5,\CI_{10}\leq C\al\sum\limits_{k'\leq k}\|H_{2,k'}\|_{L^\infty}.
\end{align}
For $\CI_6$, we first rewrite $\pa_v^{\ga'}(\chi_{M}\CK)(\pa_v^{\ga-\ga'}\CG_1)$ as
\begin{align*}
\pa_v^{\ga'}(\chi_{M}\CK) (\pa_v^{\ga-\ga'}\CG_1)=&\sum\limits_{\ga''\leq\ga'}C_{\ga'}^{\ga''}\pa_v^{\ga'-\ga''}\chi_{M}\pa_v^{\ga''}\CK (\pa_v^{\ga-\ga'}\CG_1)\notag\\
=&\sum\limits_{\ga''\leq\ga'}C_{\ga'}^{\ga''}\pa_v^{\ga'-\ga''}\chi_{M}\left\{Q (\pa_v^{\ga''}\mu,\pa_v^{\ga-\ga'}\CG_1)
+Q (\pa_v^{\ga-\ga'}\CG_1,\pa_v^{\ga''}\mu)\right\},
\end{align*}
then one sees that
\begin{align}\notag
\CI_6{\bf1}_{k\geq1}\leq C\sum\limits_{k'<k}\|H_{1,k'}\|_{L^\infty},
\end{align}
according to Lemma \ref{op.es.lem}. And likewise, we also have
\begin{align}\notag
\CI_{12}\leq C\sum\limits_{k'\leq k}\|H_{1,k'}\|_{L^\infty}.
\end{align}
Next, Lemma \ref{Ga} leads us to have
\begin{align}\notag
\CI_{11}{\bf1}_{k\geq1}\leq C\sum\limits_{k'<k}\|H_{2,k'}\|_{L^\infty}.
\end{align}
For $\CI_7$ and $\CI_{13}$, one directly has
%in view of Lemma \ref{op.es.lem}, \eqref{g1g2.es} and \eqref{sglw}, we have
\begin{align}\notag
\CI_7
\leq C\|w_{l}\na_v^k\CF_1\|_{L^\infty}, \ \CI_{13}\leq C\|w_{l}\na_v^k\CF_2\|_{L^\infty}.
\end{align}
Finally, for the delicate term $\CI_{9}$, we divide our computations into the following three cases.

\medskip
\noindent{\it Case 1.} $|V|\geq M$ with $M$ suitably large.
From Lemma \ref{Kop}, it follows that
$$
\int\mathbf{k}_w(V,v_\ast)\,dv_\ast\leq \frac{C}{(1+|V|)}\leq \frac{C}{M}.
$$
Using this, it follows that
\begin{equation}\label{I81}
\CI_{9}\leq \sup\limits_{0\leq s\leq t}\int_{\R^3}\mathbf{k}_w(V,v_\ast)\,dv_\ast\|H_{2,k}\|_{L^\infty}\leq \frac{C}{M}\|H_{2,k}\|_{L^\infty}.
\end{equation}

\medskip
\noindent{\it Case 2.} $|V|\leq M$ and $|v_\ast|\geq 2M$. In this situation, we have
$|V-v_\ast|\geq M$, then
\begin{equation*}
\mathbf{k}_w(V,v_\ast)
\leq Ce^{-\frac{\vps M^2}{8}}\mathbf{k}_w(V,v_\ast)e^{\frac{\vps |V-v_\ast|^2}{8}}.
\end{equation*}
By virtue of Lemma \ref{Kop}, one sees that
$\int\mathbf{k}_w(V,v_\ast)e^{\frac{\vps |V-v_\ast|^2}{8}}\,dv_\ast$ is still bounded. In this stage,
we have by a similar argument as for obtaining \eqref{I81} that
\begin{equation*}%\label{J32}
\begin{split}
\CI_{9}\leq Ce^{-\frac{\vps M^2}{8}}\|H_{2,k}\|_{L^\infty}.
\end{split}
\end{equation*}
To obtain the final bound for $\CI_{9}$, we are now in a position to handle the last case.

\medskip
\noindent{\it Case 3.} $|V|\leq M$, $|v_\ast|\leq 2M$. In this case, our strategy is to convert the bound in $L^\infty$-norm to the one in $L^2$-norm which will be established later on. To do so, for any large $M>0$,
we choose a number $p=p(M)$ to define
\begin{equation}\label{km}
\mathbf{k}_{w,p}(V,v_\ast)\equiv \mathbf{1}_{|V-v_\ast|\geq\frac{1}{p},|v_\ast|\leq p}\mathbf{k}_{w}(V,v_\ast),
\end{equation}
such that $\sup\limits_{V}\int_{\mathbf{R}^{3}}|\mathbf{k}_{w,p}(V,v_\ast)
-\mathbf{k}_{w}(V,v_\ast)|\,dv_\ast\leq
\frac{1}{M}.$ One then has
\begin{align*}
%\begin{split}
\CI_{9}&\leq C\sup\limits_{s}\int_{|v_\ast|\leq 2M}\mathbf{k}_{w,p}(V,v_\ast)|\na_v^k\CG_2(v_\ast)|dv_\ast+\frac{1}{M}\|H_{2,k}\|_{L^\infty}\\
&\leq C(p)\sup\limits_{s}\|\na_v^k\CG_2\|+\frac{1}{M}\|H_{2,k}\|_{L^\infty},
%\end{split}
\end{align*}
according to H\"older's inequality and the fact that $\int_{\R^3}\mathbf{k}^2_{w,p}(V,v_\ast)dv_\ast<\infty.$

Therefore, it follows that for any large $M>0$,
\begin{equation}\label{CI8}
\begin{split}
\CI_{9}\leq C\left(e^{-\frac{\vps M^2}{8}}+\frac{1}{M}\right)\|H_{2,k}\|_{L^\infty}+C\sup\limits_{s}\|\na_v^k\CG_2\|.
\end{split}
\end{equation}
Combing all the estimates above together, we now conclude to have
\begin{align}\label{lifn}
\left\{\begin{array}{rll}
&\|H_{1,k}\|_{L^\infty}\leq \left(\frac{1}{8}+\frac{C}{l}+C\al\right)\|H_{1,k}\|_{L^\infty}+C\al\|H_{1,0}\|_{L^\infty}
\\[4mm]&\qquad\qquad\quad+C\al\sum\limits_{k'\leq k}\|H_{2,k'}\|_{L^\infty}
+{\bf1}_{k\geq1}C\sum\limits_{k'<k}\|H_{1,k'}\|_{L^\infty}
+C\|w_{l}\na_v^k\CF_1\|_{L^\infty},\\[4mm]
&\|H_{2,k}\|_{L^\infty}\leq \left(\frac{1}{8}+\frac{C}{M}+C\al\right)\|H_{2,k}\|_{L^\infty}+{\bf1}_{k\geq1}C\sum\limits_{k'<k}\|H_{2,k'}\|_{L^\infty}
\\[4mm]&\qquad\qquad\quad+C\sum\limits_{k'\leq k}\|H_{1,k'}\|_{L^\infty}+C\|\na_v^k\CG_2\|+C\|w_{l}\na_v^k\CF_2\|_{L^\infty}.\end{array}\right.
\end{align}
It should be pointed out that the constant $C$ in \eqref{lifn} is independent of $\si$ and $\eps$.

\medskip
\noindent\underline{Step 2. $L^2$ estimates.}
We now deduce the $L^2$ estimate on $\CG_2$ which is necessary due to \eqref{lifn}. Let us start from the macroscopic part of $(\CG_{1},\CG_2)$. Recalling the definition of $\FP_0$, at this stage, we may write
$$
\FP_0\CG_{2}=(a_{2}+\Fb_{2}\cdot v+c_{2}(|v|^2-3))\sqrt{\mu},
$$
and define the projection $\bar{\FP}_0$, from $L^2$ to $\ker(\CL)$, as
$$
\bar{\FP}_0\CG_{1}=(a_{1}+\Fb_{1}\cdot v+c_{1}(|v|^2-3))\mu,
$$
and because $[\CG_1,\CG_2]\in\FX_{\al,m}$, it also follows that 
\begin{align}\label{abR}
a_{1}+a_{2}=0,\ \Fb_{1}+\Fb_{2}=0,\ c_1+c_2=0.
\end{align}
The following significant observation will be used in the later deductions
\begin{align}\label{keyob}
\nu_0f-\si Kf=(1-\si)\nu_0f+\si Lf,\ \nu_0f-\si \CK f=(1-\si)\nu_0f+\si\CL f.
\end{align}
By applying \eqref{keyob} and \eqref{abR}, for any $k\geq0$, we get from $\lag \na_v^k\FP_1\eqref{pals}_2, \na_v^k\FP_1\CG_2\rag$ and Lemma \ref{es.L} that
\begin{align}\label{ip.Gr2}
\frac{1}{2}\min\{\nu_0,\de_0\}\|\na_v^k\FP_1\CG_{2}\|-C{\bf 1}_{k\geq1}\|\FP_1\CG_{2}\|\leq&
C\left|\left\lag\na_v^k\FP_1[(1-\chi_{M})\mu^{-\frac{1}{2}}\CK\CG_1], \na_v^k\FP_1\CG_2\right\rag\right|^{\frac{1}{2}}\notag\\&+C\|\na_v^k\CF_2\|+C\al|[a_1,\Fb_1,c_1]|,
\end{align}
and using Lemma \ref{op.es.lem}, one also has
\begin{align}\label{CKl2}
\bigg|\bigg\lag&\na_v^k\FP_1[(1-\chi_{M})\mu^{-\frac{1}{2}}\CK\CG_1], \na_v^k\FP_1\CG_2\bigg\rag\bigg|\notag\\
\leq& \eta\|\na_v^k\FP_1\CG_2\|^2+C_{\eta}\left\|\na_v^k\FP_1[(1-\chi_{M})\mu^{-\frac{1}{2}}\CK\CG_1]\right\|^2\notag\\
\leq& \eta\|\na_v^k\FP_1\CG_2\|^2+C_{\eta}\left\|\na_v^k[(1-\chi_{M})\mu^{-\frac{1}{2}}\CK\CG_1]\right\|^2
+C_{\eta}\left\|\na_v^k\FP_0[(1-\chi_{M})\mu^{-\frac{1}{2}}\CK\CG_1]\right\|^2\notag\\
\leq& \eta\|\na_v^k\FP_1\CG_2\|^2+C_\eta\sum\limits_{k'\leq k}\|w_{l}\na_v^{k'}\CG_1\|^2_{L^\infty}.
\end{align}
For $l > 5/2$, it follows
\begin{align}\label{wG2}
|[a_1,\Fb_1,c_1]|\leq C\|w_{l}\CG_{1}\|_{L^\infty}.
\end{align}
Now, \eqref{ip.Gr2}, \eqref{CKl2} and \eqref{wG2} give rise to
\begin{align}\label{disG2}
\sum\limits_{0\leq k\leq m}\|\na_v^k\FP_1\CG_{2}\|+|[a_1,\Fb_1,c_1]|
\leq C\sum\limits_{0\leq k'\leq k}\|w_{l}\na_v^{k'}\CG_{1}\|_{L^\infty}+C\sum\limits_{0\leq k\leq m}\|w_{l}\na_v^k\CF_2\|_{L^\infty},
\end{align}
for $l > 5/2$, where $C>0$ is independent of $\eps$.

Consequently, taking the linear combination of \eqref{lifn} and \eqref{disG2} for $0\leq k\leq m$ and adjusting constants, we arrive at
\begin{align*}
\sum\limits_{0\leq k\leq m}\left\{\|H_{1,k}\|_{L^\infty}+\|H_{2,k}\|_{L^\infty}\right\}\leq
C\sum\limits_{0\leq k\leq m}\|w_{l}\na_v^k[\CF_1,\CF_2]\|_{L^\infty}.
\end{align*}
%which further gives rise to \eqref{Lif.es1}.
This shows the desired estimate \eqref{Lif.es1} and ends the proof of Lemma \ref{lifpri}.
\end{proof}

With Lemma \ref{lifpri} in hand, we now turn to prove the existence of solutions to \eqref{pals} in $L^\infty$ framework by the contraction mapping method. We employ the continuity technique in the parameter $\si$ developed in \cite{DHWZ-19}. 

\begin{lemma}\label{ex.pals}
%For any $m\geq0$, Let  $[\CF_1,\CF_2]$ satisfy \eqref{F12} with $0\leq k\leq m$, then 
Under the same assumption of Lemma \ref{lifpri},
there exists a unique solution $[\CG_1,\CG_2]\in\FX_{\al,m}$
to \eqref{pals} with $\si=1$ satisfying
\begin{align}\label{Lif.es2}
\sum\limits_{0\leq k\leq m}\left\{\|w_{l}\na_v^k\CG_1\|_{L^\infty}+\|w_{l}\na_v^k\CG_2\|_{L^\infty}\right\}\leq
C\sum\limits_{0\leq k\leq m}\left\{\|w_{l}\na_v^k\CF_1\|_{L^\infty}+\|w_{l}\na_v^k\CF_2\|_{L^\infty}\right\}.
\end{align}
\end{lemma}

\begin{proof}
If $\si=0$, then \eqref{H10} and \eqref{H20} can be reduced to
\begin{align}
\eps H_1&-\bet' \na_v\cdot(v H_1)+2l \bet' \frac{|v|^2}{1+|v|^2} H_1
-\al v_2\pa_{v_1}H_1
+2l \al \frac{v_2v_1}{{1+|v|^2}}H_1
+\nu_0H_1\notag\\&+\frac{\beta'}{2}|v|^2\sqrt{\mu}H_2
+\al\frac{v_1v_2}{2}\sqrt{\mu}H_2
-w_l\frac{\bet''(\frac{\FH_0}{w_l})}{\al}\na_v\cdot(v\mu)
=w_{l}\CF_1,\notag
\end{align}
and
\begin{align}
\eps H_2&-\bet' \na_v\cdot(v H_2)+2l \bet' \frac{|v|^2}{1+|v|^2} H_2
-\al v_2\pa_{v_1}H_2+2l \al \frac{v_2v_1}{{1+|v|^2}}H_2
+\nu_0 H_2
\notag\\=&w_{l}\CF_2,\notag
\end{align}
respectively. Then, in this case of $\si=0$, the existence of $L^\infty$-solutions can be easily proved by the characteristic method and the contraction mapping theorem, since there is no trouble term involving $K$ or
$\CK$. That is, one can directly show that 
\begin{align}\label{L0}
\|\mathscr{L}_0^{-1}[\CF_1,\CF_2]\|_{\FX_{\al,m}}\leq C_\mathscr{L}\|[\CF_1,\CF_2]\|_{\FX_{\al,m}}.
\end{align}
We now define a operator
\begin{align}
&\CT_\si[\CG_1,\CG_2]=\mathscr{L}^{-1}_{0}\Big[\si \chi_{M}\CK \CG_1+\CF_1,
\ \si (1-\chi_{M})\mu^{-\frac{1}{2}}\CK \CG_1+\si K\CG_2+\CF_2\Big].\notag
\end{align}
Moreover, since $[\CG_1,\CG_2]\in\FX_{\al,m}$, one also has
\begin{align}
\lag \CK\CG_1,[1,v_i,|v|^2]\rag+\lag K\CG_2,[1,v_i,|v|^2]\mu^{\frac{1}{2}}\rag=0,\ 1\leq i\leq3,\notag
\end{align}
according to \eqref{keyob}, which further implies
$$
\lag \CT_\si[\CG_1,\CG_2],[\{1,v_i,|v|^2\},\{1,v_i,|v|^2]\}\mu^{\frac{1}{2}}]\rag=0,
$$
for any $\eps>0.$
Then \eqref{L0} yields
\begin{align}\label{cT}
&\|\CT_\si[\CG_1,\CG_2]-\CT_\si[\CG'_1,\CG'_2]\|_{\FX_{\al,m}}\notag\\[2mm]
\quad&=\Big\|\mathscr{L}^{-1}_{0}\Big[\si\chi_{M}\CK \left(\CG_1-\CG_1'\right),
\si (1-\chi_{M})w_{l}\mu^{-\frac{1}{2}}\CK \left(\CG_1-\CG_1'\right)
\notag\\&\qquad\qquad+\si w_{l}K\left(\CG_2-\CG_2'\right)\Big]\Big\|_{\FX_{\al,m}}\notag\\[2mm]
\quad&=\si\Big\|\mathscr{L}^{-1}_{0}\Big[ w_{l}\chi_{M}\CK \left(\CG_1-\CG_1'\right),
 (1-\chi_{M})w_{l}\mu^{-\frac{1}{2}}\CK \left(\CG_1-\CG_1'\right)
\notag\\&\qquad\qquad+ w_{l}K\left(\CG_2-\CG_2'\right)\Big]\Big\|_{\FX_{\al,m}}\notag\\[2mm]
\quad&\leq \si C_\mathscr{L}\|[\CG_1-\CG_1',\CG_2-\CG_2']\|_{\FX_{\al,m}}\leq \frac{1}{2}\|[\CG_1-\CG_1',\CG_2-\CG_2']\|_{\FX_{\al,m}},
\end{align}
provided that $\si\in[0,\si_\ast]$ with $0<\si_\ast\leq\min\{\frac{1}{2C_\mathscr{L}}\}.$

Thus, we obtain a unique fixed point $[\CG_1,\CG_2]$ in $\FX_{\al,m}$ such that
\begin{align}\notag
\CT_\si[\CG_1,\CG_2]=[\CG_1,\CG_2],
\end{align}
which is equivalent to
\begin{align}
\mathscr{L}_0[\CG_1,\CG_2]=&\Big[\si \chi_{M}\CK \CG_1+\CF_1,
\si (1-\chi_{M})\mu^{-\frac{1}{2}}\CK \CG_1
+\si K\CG_2+\CF_2\Big].\notag
\end{align}
Therefore $[\CG_1,\CG_2]$ is a unique solution to the system
\begin{align}
\mathscr{L}_\si[\CG_1,\CG_2]=\left[\CF_1,\CF_2\right],\ \si\in[0,\si_\ast].\notag
\end{align}
Next, we define
\begin{align}
\CT_{\si_\ast+\si}=&\mathscr{L}_{\si_\ast}^{-1}\Big[\si \chi_{M}\CK \CG_1+\CF_1,
\si (1-\chi_{M})\mu^{-\frac{1}{2}}\CK \CG_1+\si K\CG_2+\CF_2\Big].\notag
\end{align}
Since the constant $C_\mathscr{L}$ in \eqref{Lif.es1} is uniform in $\si\in[0,1]$, one can further verify that $\CT_{\si_\ast+\si}$ with
$\si\in[0,\si_\ast]$ is also a contraction mapping on $\FX_{\al,m}$ by using the similar argument as for obtaining \eqref{cT}.
Namely, we have shown the existence of $\mathscr{L}^{-1}_{2\si_\ast}$ on $\FX_{\al,m}$ and \eqref{Lif.es1} holds true for $\si=2\si_\ast$. Hence, step by step,
one can see that $\mathscr{L}^{-1}_1$ exists in case $\si=1$ and \eqref{Lif.es2} also follows simultaneously. This completes the proof of Lemma \ref{ex.pals}.
\end{proof}

Once Lemma \ref{ex.pals} has been obtained, we can now turn to complete the

\begin{proof}[Proof of Theorem \ref{st.sol}]
We prove the existence of $W^{m,\infty}$ solution to the coupled system \eqref{gr1} and \eqref{gr2} under the condition \eqref{abr12}.

According to Lemma \ref{ex.pals}, there indeed exists a unique solution $[G^{n+1}_{R,1},G^{n+1}_{R,2}]$ to the system \eqref{gr12.ls} satisfying \eqref{abr12n}, provided that
$[G^{n}_{R,1},G^{n}_{R,2}]\in\FX_{\al,m}$ for any $m\geq0$. We now show that the solution sequence $\{[G^{n}_{R,1},G^{n}_{R,2}]\}_{n=0}^{\infty}$
is a Cauchy sequence in $\FX_{\al,m-1}$ with $m\geq1$, hence it is convergent and the limit is the unique solution of the following system
\begin{align}
\eps G_{R,1}&-\bet\na_v\cdot (vG_{R,1})-\al v_2\pa_{v_1}G_{R,1}+\frac{\beta}{2}|v|^2\sqrt{\mu}G_{R,2}+\al\frac{v_1v_2}{2}\sqrt{\mu}G_{R,2}
+\nu_0G_{R,1}
-\chi_{M}\CK G_{R,1}\notag\\
=&\frac{\bet}{\al}\na_v\cdot(v\mu)+\beta\na_v\cdot(v\mu^{\frac{1}{2}}G_1)
+\al v_2\pa_{v_1}(\mu^{\frac{1}{2}}G_1)+
\al Q(\mu^{\frac{1}{2}}G_1,\mu^{\frac{1}{2}}G_1)\notag\\&+\al\{Q(\mu^{\frac{1}{2}}G_R,\mu^{\frac{1}{2}}G_1)
+Q(\mu^{\frac{1}{2}}G_1,\mu^{\frac{1}{2}}G_R)\}+\al Q(\mu^{\frac{1}{2}}G_R,\mu^{\frac{1}{2}}G_R),\label{eps.gr21}
\end{align}
and
\begin{align}\label{eps.gr22}
\eps G_{R,2}&-\bet\na_v\cdot (vG_{R,2})-\al v_2\pa_{v_1}G_{R,2}+LG_{R,2}-(1-\chi_{M})\mu^{-\frac{1}{2}}\CK G_{R,1}=0.
\end{align}
To do this, we first denote
$$
\tilde{\beta}^{n+1}=\beta^{n+1}-\beta^{n}=-\frac{\al^2}{3}\int v_1v_2\mu^{1/2}\tilde{G}^{n+1}_R\,dv, \ 
$$
with $\beta^n$ given by \eqref{n-sq},
and  set
$$\mu^{1/2}\tilde{G}_R^{n+1}=\tilde{G}^{n+1}_{R,1}+\mu^{1/2}\tilde{G}^{n+1}_{R,2},$$
with
$$
[\tilde{G}^{n+1}_{R,1},\tilde{G}^{n+1}_{R,2}]=[G_{R,1}^{n+1}-G_{R,1}^{n},G_{R,2}^{n+1}-G_{R,2}^{n}].
$$
Then  $\tilde{G}_{R,1}^{n+1}$ and $\tilde{G}_{R,2}^{n+1}$ satisfy the following equations:
\begin{eqnarray*}%\label{dG}
\left\{\begin{array}{rll}
\begin{split}
\eps\tilde{G}^{n+1}_{R,1}&-\beta^{n}\na_v\cdot (v\tilde{G}^{n+1}_{R,1})-\al v_2\pa_{v_1}\tilde{G}^{n+1}_{R,1}
+\nu_0\tilde{G}^{n+1}_{R,1}
-\chi_{M}\CK \tilde{G}^{n+1}_{R,1}\\&+\frac{\beta^n}{2}|v|^2\sqrt{\mu}\tilde{G}^{n+1}_{R,2}
+\al\frac{v_1v_2}{2}\sqrt{\mu}\tilde{G}^{n+1}_{R,2}-\frac{\tilde{\bet}^{n+1}}{\al}\na_v\cdot(v\mu)\\
=&\tilde{\beta}^{n}\na_v\cdot (vG_{R,1}^{n})+\tilde{\beta}^{n}\na_v\cdot(v\mu^{\frac{1}{2}}G_1)+\al\{Q(\mu^{\frac{1}{2}}\tilde{G}^{n}_R,\mu^{\frac{1}{2}}G_1)
+Q(\mu^{\frac{1}{2}}G_1,\mu^{\frac{1}{2}}\tilde{G}^{n}_R)\}\\&+\al \left\{Q(\mu^{\frac{1}{2}}\tilde{G}_R^{n},\mu^{\frac{1}{2}}\tilde{G}^{n})
+Q(\mu^{\frac{1}{2}}\tilde{G}_R^{n},\mu^{\frac{1}{2}}G^n_R)+Q(\mu^{\frac{1}{2}}G^n_R,\mu^{\frac{1}{2}}\tilde{G}_R^{n})\right\}
\\ \eqdef& \CM(\tilde{G}_R^{n},\tilde{G}_R^{n}),\\
\eps\tilde{G}^{n+1}_{R,2}&-\beta^{n}\na_v\cdot (v \tilde{G}^{n+1}_{R,2})-\al v_2\pa_{v_1}\tilde{G}^{n+1}_{R,2}
+L\tilde{G}^{n+1}_{R,2}-(1-\chi_{M})\mu^{-\frac{1}{2}}\CK \tilde{G}^{n+1}_{R,1}=0,
\end{split}
\end{array}\right.
\end{eqnarray*}
with
\begin{align}
\lag \tilde{G}^{n}_{R,1},[1,v_i,|v|^2]\rag+\lag \tilde{G}^{n}_{R,2},[1,v_i,|v|^2]\mu^{\frac{1}{2}}\rag=0,\ i=1,2,3.\notag
\end{align}
Our goal next is to prove
\begin{align}\label{cau}
\|[\tilde{G}^{n+1}_{R,1},\tilde{G}^{n+1}_{R,2}]\|_{\FX_{\al,m-1}}\leq C\al\|[\tilde{G}^{n}_{R,1},\tilde{G}^{n}_{R,2}]\|_{\FX_{\al,m-1}},
\end{align}
under the condition that
\begin{align}\label{tGbd}
\|[G^n_{R,1},G^n_{R,2}]\|_{\FX_{\al,m}}<C_{\al,m}, 
%\ \textrm{for all}\ n\geq0, \ m\geq1.
\end{align}
where the constant $C_{\al,m}$ is finite independent of $n$. 
In fact, on the one hand, thanks to Lemma \ref{lifpri}, it follows that
\begin{align*}
\|[\tilde{G}^{n+1}_{R,1},\tilde{G}^{n+1}_{R,2}]\|_{\FX_{\al,m-1}}\leq C\sum\limits_{0\leq k\leq m-1}\|w_{l}\na_x^k\CN(\tilde{G}_R^{n},\tilde{G}_R^{n})\|_{\infty}.
\end{align*}
On the other hand, we get from Lemma \ref{op.es.lem} that
\begin{align}
\sum\limits_{0\leq k\leq m-1}\|w_{l}\na_x^k\CM(\tilde{G}_R^{n},\tilde{G}_R^{n})\|_{\infty}\leq&
C\al\sum\limits_{0\leq k\leq m-1}\left\{\|w_{l}\na_x^k\tilde{G}^{n}_{R,1}\|_{L^\infty}^2+\|w_{l}\na_x^k\tilde{G}^{n}_{R,2}\|_{L^\infty}^2\right\}\notag
\\&+
C\al\sum\limits_{0\leq k\leq m-1}\|w_{l}\na_x^k\tilde{G}_R^{n}\|_{L^\infty}\sum\limits_{0\leq k\leq m}\|w_{l}\na_x^kG^n_R\|_{L^\infty},\notag
\end{align}
which is further bounded by
\begin{align}
C\al\sum\limits_{0\leq k\leq m-1}\left\{\|w_{l}\na_v^k\tilde{G}^{n}_{R,1}\|_{L^\infty}
+\|w_{l}\na_v^k\tilde{G}^{n}_{R,2}\|_{L^\infty}\right\},\notag
\end{align}
due to \eqref{tGbd}. Thus, \eqref{cau} is valid, in other words, $\{[G^{n}_{R,1},G^{n}_{R,2}]\}_{n=0}^{\infty}$
is a Cauchy sequence in $\FX_{\al,m-1}$ for $\al>0$ suitably small. Hence,
$$
[G^{n}_{R,1},G^{n}_{R,2}]\rightarrow[G_{R,1}^\eps,G_{R,2}^\eps]
$$
strongly in $\FX_{\al,m-1}$ as $n\rightarrow+\infty$, and
\begin{align}\label{betae}
\beta^n\rightarrow\beta^{\eps}
%=-\frac{\al^2}{3}\int_{\R^3} v_1v_2\mu^{1/2}(G_1+G^\eps_R)dv
=\frac{\al^2}{6b_0}-\frac{\al^2}{3}\int_{\R^3} v_1v_2(G^\eps_{R,1}+\mu^{1/2}G^\eps_{R,2})\,dv.
\end{align}
And the limit $[G^{\eps}_{R,1},G^{\eps}_{R,2}]$ is a unique solution to \eqref{eps.gr21} and \eqref{eps.gr22}.
Furthermore, it can be directly shown that $[G^{\eps}_{R,1},G^{\eps}_{R,2}]$ enjoys the estimate
\begin{align}\label{Gebd}
\|[G^\eps_{R,1},G^\eps_{R,2}]\|_{\FX_{\al,m}}\leq C\al.
\end{align}
Furthermore, taking the limit $\eps\to 0$, we may repeat the same procedure as for letting $n\to \infty$, so that the limit function 
%which implies
%$$
%\Red{[G^{\eps}_{R,1},G^{\eps}_{R,2}]\rightarrow [G_{R,1},G_{R,2}]}
%$$
%\Red{weakly in $H^m(\R^3_v)$ as $\eps\rightarrow0$}. Therefore, 
$[G_{R,1},G_{R,2}]\in\FX_{\al,m}$ is the unique solution of \eqref{gr1} and \eqref{gr2} and enjoys the same bound as \eqref{Gebd}.

%which also yields $\beta^\eps>0$ provided $\al$ is suitably small.

Moreover, for any $\eps>0$, it follows that
\begin{align*}
\eps\lag G^{\eps}_{R,1},1\rag+\eps\lag G^{\eps}_{R,2},\mu^{\frac{1}{2}}\rag=0,
\end{align*}
\begin{align*}
(\eps+\beta^\eps)\lag G^{\eps}_{R,1},v_1\rag+(\eps+\beta^\eps)\lag G^{\eps}_{R,2},v_1\mu^{\frac{1}{2}}\rag
+\al\lag G^{\eps}_{R,1},v_2\rag+\al\lag G^{\eps}_{R,2},v_2\mu^{\frac{1}{2}}\rag=0,
\end{align*}
\begin{align*}
(\eps+\beta^\eps)\lag G^{\eps}_{R,1},v_i\rag+(\eps+\beta^\eps)\lag G^{\eps}_{R,2},v_i\mu^{\frac{1}{2}}\rag=0,\ i=2,3,
\end{align*}
and
\begin{align*}
(\eps+\beta^\eps)\lag G^{\eps}_{R,1},|v|^2\rag+(\eps+\beta^\eps)\lag G^{\eps}_{R,2},|v|^2\mu^{\frac{1}{2}}\rag=0,
\end{align*}
consequently,
\begin{align}\label{Ge.cons}
\lag G^{\eps}_{R,1},[1,v_i,|v|^2]\rag+\lag G^{\eps}_{R,2},[1,v_i,|v|^2]\mu^{\frac{1}{2}}\rag=0,\ i=1,2,3,
\end{align}
since $\beta^\eps>0$ due to \eqref{Gebd} and \eqref{betae} and $\al$ can be suitably small. Taking $\eps\rightarrow0$ in \eqref{Ge.cons} gives
rise to \eqref{cls1}. This finishes the proof of Theorem \ref{st.sol}.
\end{proof}

\section{Local Existence}\label{loc.sec}
From this section we turn to study the time-dependent equation  \eqref{beF} in the one-dimensional spatially inhomogeneous setting. The  goal of this section is to first establish the unique local-in-time solution of the Cauchy problem \eqref{beF} and \eqref{beFid}, and the proof of the global existence as well as the large time behavior of solutions is left for the next section. In light of \eqref{def.trans}, we let 
$$
F(t,x,v)=e^{-3\bet t}f(t,x,\frac{v}{e^{\bet t}})\eqdef e^{-3\bet t}f(t,x,\xi),
$$
then the Cauchy problem \eqref{beF} and \eqref{beFid} is converted to
\begin{eqnarray}\label{bef}
\left\{
\begin{array}{rll}
\begin{split}
&\pa_t f+e^{\beta t}\xi_1\pa_xf-\bet \na_\xi\cdot(\xi f)-\al \xi_2\pa_{\xi_1}f =Q(f,f),\ t>0,\ x\in\T,\ \xi\in \R^3,\\
&f(0,x,\xi)=F_0(x,\xi),\ x\in\T,\ \xi\in \R^3.
\end{split}
\end{array}\right.
\end{eqnarray}
Now, setting $\tilde{f}(t,x,\xi)=f(t,x,\xi)-G(\xi)$, where the self-similar profile $G$ is determined in Theorem \ref{st.sol}, one can see that $\tilde{f}$ satisfies
\begin{eqnarray*}%\label{betf}
\left\{
\begin{array}{rll}
\begin{split}
&\pa_t \tilde{f}+e^{\beta t}\xi_1\pa_x\tilde{f}-\bet \na_\xi\cdot(\xi \tilde{f})-\al \xi_2\pa_{\xi_1}\tilde{f}\\
&\qquad\qquad=Q(\tilde{f},\tilde{f})+Q(\tilde{f},G)+Q(G,\tilde{f}),\ t>0,\ x\in\T,\ \xi\in \R^3,\\
&\tilde{f}(0,x,\xi)=F_0(x,\xi)-G(\xi),\ x\in\T,\ \xi\in \R^3.
\end{split}
\end{array}\right.
\end{eqnarray*}
Defining next $\tilde{f}(t,x,\xi)=\sqrt{\mu}\tilde{g}$ and recalling $G=\mu+\sqrt{\mu}\{\al G_1+\al G_R\}$, we have
\begin{eqnarray}\label{betg}
\left\{
\begin{array}{rll}
\begin{split}
&\pa_t \tilde{g}+e^{\beta t}\xi_1\pa_x\tilde{g}-\bet \na_\xi\cdot(\xi \tilde{g})+\frac{\bet}{2}|\xi|^2\tilde{g}-\al \xi_2\pa_{\xi_1}\tilde{g}+\frac{\al}{2}\xi_1\xi_2\tilde{g}+L\tilde{g}\\
&\qquad=\Ga(\tilde{g},\tilde{g})+\Ga(\tilde{g},\al G_1+\al G_R)+\Ga(\al G_1+\al G_R,\tilde{g}),\quad t>0,\ x\in\T,\ \xi\in \R^3,\\[2mm]
&\tilde{g}(0,x,\xi)=\tilde{g}_0=\frac{F_0(x,\xi)-G(\xi)}{\sqrt{\mu}},\ x\in\T,\ \xi\in \R^3.
\end{split}
\end{array}\right.
\end{eqnarray}
To solve \eqref{betg}, since there is a strong growth term $\frac{\al}{2}\xi_1\xi_2\tilde{g}$ in \eqref{betg}, it is more convenient to consider the decomposition $\sqrt{\mu}\tilde{g}=g_1+\sqrt{\mu}g_2$, where $g_1$ and $g_2$ satisfy
\begin{eqnarray}\label{beg1}
\left\{
\begin{array}{rll}
&\pa_t g_1+e^{\beta t}\xi_1\pa_xg_1-\bet \na_\xi\cdot(\xi g_1)-\al \xi_2\pa_{\xi_1}g_1+\nu_0g_1\\[2mm]
&\quad=\chi_{M}\CK g_1-\frac{\bet}{2}|\xi|^2\mu^{\frac{1}{2}}g_2-\frac{\al}{2}\mu^{\frac{1}{2}}\xi_1\xi_2g_2+\tilde{H}(g_1,g_2),\
t>0,\ x\in\T,\ \xi\in \R^3,\\[2mm]
&g_1(0,x,\xi)=0,\ x\in\T,\ \xi\in \R^3,\end{array}\right.
\end{eqnarray}
and
\begin{eqnarray}\label{beg2}
\left\{
\begin{array}{rll}
\begin{split}
&\pa_t g_2+e^{\beta t}\xi_1\pa_xg_2-\bet \na_\xi\cdot(\xi g_2)-\al \xi_2\pa_{\xi_1}g_2+Lg_2
\\[2mm]
&\quad= \mu^{-1/2}(1-\chi_{M})\CK g_1,\ t>0,\ x\in\T,\ \xi\in \R^3,\\[2mm]
&g_2(0,x,\xi)=\frac{F_0(x,\xi)-G(\xi)}{\sqrt{\mu}}\eqdef \tilde{g}_{0}(x,\xi),\ x\in\T,\ \xi\in \R^3,
\end{split}
\end{array}\right.
\end{eqnarray}
respectively. Here
\begin{align}
\tilde{H}(g_1,g_2)=&Q(g_1+\mu^{\frac{1}{2}}g_2,g_1+\mu^{\frac{1}{2}}g_2)
+Q(g_1+\mu^{\frac{1}{2}}g_2,\mu^{\frac{1}{2}}(\al G_1+\al G_R)
\notag\\&+Q(\mu^{\frac{1}{2}}(\al G_1+\al G_R),g_1+\mu^{\frac{1}{2}}g_2).\notag
\end{align}
We shall look for solutions of \eqref{beg1} and \eqref{beg2} in the following function space
\begin{equation*}%\label{fsp}
\begin{split}
\FY_{\al,T}=&\Big\{(\CG_1,\CG_2)\bigg| \sum\limits_{|\ga_0\leq2}\sup\limits_{0\leq t\leq T}\left\{\|w_{l}\pa_x^{\ga_0}\CG_1(t)\|_{L^\infty}+\al\|w_{l}\pa_x^{\ga_0}\CG_2(t)\|_{L^\infty}\right\}<+\infty,\\
&\qquad\qquad\qquad\lag\CG_1,[1,v_i]\rag+\lag\CG_2,[1,v_i]\mu^{\frac{1}{2}}\rag=0,\ i=1,2,3
\Big\},
\end{split}
\end{equation*}
supplemented with the norm
$$
\|[\CG_1,\CG_2]\|_{\FY_{\al,T}}=\sum\limits_{\ga_0\leq2}\sup\limits_{0\leq t\leq T}\left\{\|w_{l}\pa_x^{\ga_0}\CG_1(t)\|_{L^\infty}+\al\|w_{l}\pa_x^{\ga_0}\CG_2(t)\|_{L^\infty}\right\}.
$$
%Let $\CH=[\CH_1,\CH_2](t,x,\xi)\in\FY_\al$ is given, we then consider the following

\begin{theorem}[Local existence]\label{loc.th}
Under the conditions listed in Theorem \ref{ge.th}, there exits  $T_\ast>0$ which may depend on $\al$ such that the coupling systems \eqref{beg1} and \eqref{beg2} admit a unique local in time solution $[g_1(t,x,\xi),g_2(t,x,\xi)]$ satisfying
\begin{align*}
\|[g_1,g_2]\|_{\FY_{\al,T_\ast}}\leq2\varepsilon_0\al.
\end{align*}
\end{theorem}
\begin{proof}
Our proof is based on the Duhamel's principle and contraction mapping method.
We first consider the following approximation equations
\begin{eqnarray}\label{ap.beg1}
\left\{
\begin{array}{rll}
&\pa_t g_1+e^{\beta t}\xi_1\pa_xg_1-\bet \na_\xi\cdot(\xi g_1)-\al \xi_2\pa_{\xi_1}g_1+\nu_0g_1\\[2mm]
&\quad=\chi_{M}\CK g'_1-\frac{\bet}{2}|\xi|^2\mu^{\frac{1}{2}}g'_2-\frac{\al}{2}\mu^{\frac{1}{2}}\xi_1\xi_2g'_2+\tilde{H}(g'_1,g'_2),\
t>0,\ x\in\T,\ \xi\in \R^3,\\[2mm]
&g_1(0,x,\xi)=0,\ x\in\T,\ \xi\in \R^3,\end{array}\right.
\end{eqnarray}
and
\begin{eqnarray}\label{ap.beg2}
\left\{
\begin{array}{rll}
\begin{split}
&\pa_t g_2+e^{\beta t}\xi_1\pa_xg_2-\bet \na_\xi\cdot(\xi g_2)-\al \xi_2\pa_{\xi_1}g_2+\nu_0g_2
\\[2mm]
&\quad=Kg'_2+ \mu^{-\frac{1}{2}}(1-\chi_{M})\CK g'_1,\ t>0,\ x\in\T,\ \xi\in \R^3,\\[2mm]
&g_2(0,x,\xi)=\frac{F_0(x,\xi)-G(\xi)}{\sqrt{\mu}}\eqdef \tilde{g}_{0}(x,\xi),\ x\in\T,\ \xi\in \R^3.
\end{split}
\end{array}\right.
\end{eqnarray}
Let $[g_1,g_2]$ be a solution of the pair of \eqref{ap.beg1} and \eqref{ap.beg2} with $[g'_1,g'_2]$ being given.
Then the nonlinear operator $\CN$ is formally defined as
$$
\CN([g'_1,g'_2])=[g_1,g_2].
$$
Our aim is to prove that there exists a sufficiently small $T_\ast>0$ such that $\CN[g'_1,g'_2]$ has a unique fixed point in
some Banach space by adopting the contraction mapping method. In fact, since
\begin{align*}
\sum\limits_{\ga_0\leq2}\|w_{l}\pa^{\ga_0}_x\tilde{g}_{0}\|_{L^\infty}\leq \vps_0,
\end{align*}
we can define the following Banach space
\begin{equation*}%\label{Yvp}
\begin{split}
\FY^{\vps_0}_{\al,T}=&\Big\{(\CG_1,\CG_2)\bigg| \sum\limits_{\ga_0\leq2}\sup\limits_{0\leq t\leq T}\left\{\|w_{l}\pa_x^{\ga_0}\CG_1(t)\|_{L^\infty}+\al\|w_{l}\pa_x^{\ga_0}\CG_2(t)\|_{L^\infty}\right\}
\leq2\vps_0\al,\ \vps_0>0,\\
&\qquad\qquad\qquad\CG_1(0)=0,\ \CG_2(0)=\tilde{g}_{0}
\Big\},
\end{split}
\end{equation*}
associated with the norm
$$
\|[\CG_1,\CG_2]\|_{\FY^{\vps_0}_{\al,T}}=\sum\limits_{\ga_0\leq2}\sup\limits_{0\leq t\leq T}\left\{\|w_{l}\pa_x^{\ga_0}\CG_1(t)\|_{L^\infty}+\al\|w_{l}\pa_x^{\ga_0}\CG_2(t)\|_{L^\infty}\right\}.
$$
We now show that
\begin{align*}
\CN: \FY^{\vps_0}_{\al,T}\rightarrow \FY^{\vps_0}_{\al,T},
\end{align*}
is well-defined and $\CN$ is a contraction mapping for some $T>0$.

Let us denote $h_1=w_{l}\pa_x^{\ga_0}g_1$ and $h_2=w_{l}\pa_x^{\ga_0}g_2$ with $\ga_0\leq2$, then $[h_1,h_2]$ satisfies
\begin{eqnarray}\label{ap.beh1}
\left\{
\begin{array}{rll}
\begin{split}
&\pa_th_1+e^{\beta t}\xi_1\pa_xh_1-\bet \na_\xi\cdot(\xi h_1)-\al \xi_2\pa_{\xi_1}h_1
+2l \bet \frac{|\xi|^2}{1+|\xi|^2} h_1
+2l \al \frac{\xi_2\xi_1}{{1+|\xi|^2}}h_1+\nu_0h_1\\
&\quad=\chi_{M}w_{l}\CK \left(\frac{h'_1}{w_{l}}\right)
-\frac{\bet}{2}|\xi|^2\mu^{\frac{1}{2}}h'_2-\frac{\al}{2}\mu^{\frac{1}{2}}\xi_1\xi_2h'_2\\
&\qquad+w_{l}\pa_x^{\ga_0}\tilde{H}(g'_1,g'_2),\quad t>0,\ x\in\T,\ \xi\in \R^3,\\
&h_1(0,x,\xi)=0,\ x\in\T,\ \xi\in \R^3,
\end{split}
\end{array}\right.
\end{eqnarray}
and
\begin{eqnarray}\label{ap.beh2}
\left\{
\begin{array}{rll}
\begin{split}
&\pa_t h_2+e^{\beta t}\xi_1\pa_xh_2-\bet \na_\xi\cdot(\xi h_2)-\al \xi_2\pa_{\xi_1}h_2
+2l \bet \frac{|\xi|^2}{1+|\xi|^2} h_2
\\&\qquad+2l \al \frac{\xi_2\xi_1}{{1+|\xi|^2}}h_2+\nu_0h_2-w_{l}K\left(\frac{h'_2}{w_{l}}\right)
\\
&\qquad = w_{l}\mu^{-1/2}(1-\chi_{M})\CK \left(\frac{h'_1}{w_{l}}\right),t>0,\ x\in\T,\ \xi\in \R^3,\\
&h_2(0,x,\xi)= w_{l}\pa_x^{\ga_0}\tilde{g}_{0}(x,\xi),\ x\in\T,\ \xi\in \R^3,
\end{split}
\end{array}\right.
\end{eqnarray}
where $h'_i=w_{l}\pa_x^{\ga_0}g'_i$$(i=1,2)$.

Next, we define the characteristic line $[s,X(s;t,x,\xi),V(s;t,x,\xi)]$ for equations \eqref{ap.beh1} and \eqref{ap.beh2} passing through $(t,x,\xi)$ such that
\begin{align*}%\label{chl}
\left\{\begin{array}{rll}
&\frac{dX}{ds}=e^{\beta s}V_1(s;t,x,\xi),\\[2mm]
&\frac{d V_1}{ds}=-\beta V_1(s;t,x,\xi)-\al V_2(s;t,x,\xi),\\[2mm]
&\frac{d V_i}{ds}=-\beta V_i(s;t,x,\xi),\ i=2,3,\\[2mm]
&X(t;t,x,\xi)=x,\ V(t;t,x,\xi)=\xi,
\end{array}\right.
\end{align*}
which is equivalent to
\begin{equation}\label{chlp2}
\left\{\begin{split}
&X(s;t,x,\xi)=e^{\beta(t-s)}\Big(x-(t-s)\xi_1-\frac{1}{2}\al(t-s)^2\xi_2\Big),\\[2mm]
&V_1(s;t,x,\xi)=e^{\beta(t-s)}(\xi_1+\al\xi_2(t-s)),\\[2mm]
&V_i(s;t,x,\xi)=e^{\beta(t-s)}\xi_i.\ \ i=2,3.
\end{split}\right.
\end{equation}
Using this, as \eqref{H10m} and \eqref{H20m}, we can write the solution of \eqref{ap.beh1} as
\begin{align*}
[h_1,h_2]=\CQ(g'_1,g'_2)=[\CQ_{1}(g'_1,g'_2),\CQ_{2}(g'_1,g'_2)],
\end{align*}
with
\begin{align}\label{h10m}
\CQ_{1}(g'_1,g'_2)
=&\int_0^{t}e^{-\int_{s}^t\CA(\tau,V(\tau))\,d\tau}\left\{\chi_{M}w_l\CK
\left(\frac{h'_1}{w_{l}}\right)\right\}(V(s))\,ds\notag\\
&+\frac{\beta}{2}\int_0^{t}e^{-\int_{s}^t\CA(\tau,V(\tau))\,d\tau}
|V(s)|^2\sqrt{\mu}(V(s))h'_2(V(s))\,ds\notag\\
&+\al\int_0^{t}e^{-\int_{s}^t\CA(\tau,V(\tau))\,d\tau}
\frac{V_1(s)V_2(s)}{2}\sqrt{\mu}(V(s))h'_2(V(s))\,ds\notag\\
&+\int_0^{t}e^{-\int_{s}^t\CA(\tau,V(\tau))\,d\tau}\left(w_{l}\tilde{H}\right)(V(s))\,ds,
\end{align}
and
\begin{align}\label{h20m}
\CQ_{2}(g'_1,g'_2)
=&e^{-\int_{0}^t\CA(\tau,V(\tau))\,d\tau}(w_{l}\pa^{\ga_0}_x\tilde{g}_{0})(X(0),V(0))\notag\\
&+\int_0^{t}e^{-\int_{s}^t\CA(\tau,V(\tau))\,d\tau}\left\{(1-\chi_{M})\mu^{-\frac{1}{2}}w_{l}\CK
\left(\frac{h'_1}{w_{l}}\right)\right\}(V(s))\,ds\notag\\
&+\int_0^{t}e^{-\int_{s}^t\CA(\tau,V(\tau))\,d\tau}\left[w_{l}K\left(\frac{h'_2}{w_{l}}\right)\right](V(s))\,ds,
\end{align}
where $\CA=\CA^\eps-\eps.$

Let $[g'_1,g'_2]\in\FY^{\vps_0}_{\al,T_\ast}$. In light of \eqref{sglw}, taking $L^\infty$ estimates of $\CQ[g'_1,g'_2]$ and applying Lemmas \ref{Kop}, \ref{op.es.lem} and \ref{CK}, one directly has
\begin{align}\label{Q1}
\sum\limits_{\ga_0\leq2}\sup\limits_{0\leq t\leq T_\ast}\|\CQ_1[g'_1,g'_2]\|_{L^\infty}
\leq& (\frac{C}{l }+C\al+C\vps_0)T_\ast\sum\limits_{\ga_0\leq2} \sup\limits_{0\leq t\leq T_\ast}\|h'_1\|_{L^\infty}
\notag\\&+C(\al+\vps_0) T_\ast\sum\limits_{\ga_0\leq2} \sup\limits_{0\leq t\leq T_\ast}\|h'_2\|_{L^\infty}
\leq \frac{\vps_0\al}{2},
\end{align}
and
\begin{align}\label{Q2}
\sum\limits_{\ga_0\leq2}&\sup\limits_{0\leq t\leq T_\ast}\al\|\CQ_2[g'_1,g'_2]\|_{L^\infty}\notag\\
\leq& \vps_0\al+C T_\ast\al \sum\limits_{\ga_0\leq2}\bigg\{\sup\limits_{0\leq t\leq T_\ast}\|h'_1\|_{L^\infty}+\sup\limits_{0\leq t\leq T_\ast}\|h'_2\|_{L^\infty}\bigg\}\leq \frac{3\vps_0\al}{2},
\end{align}
provided that $T_\ast>0$ is suitably small.
And similarly, for $[g'_1,g'_2]\in\FY^{\vps_0}_{\al,T_\ast}$ and
$[g''_1,g''_2]\in\FY^{\vps_0}_{\al,T_\ast}$, it follows that 
\begin{align}\label{dfh1}
\sum\limits_{\ga_0\leq2}&\sup\limits_{0\leq t\leq T_\ast}\|\CQ_1[g'_1,g'_2]-\CQ_1[g''_1,g''_2]\|_{L^\infty}\notag\\
\leq& (\frac{C}{l }+C\al+C\vps_0)T_\ast \sum\limits_{\ga_0\leq2}\sup\limits_{0\leq t\leq T_\ast}\|h'_1-h_1''\|_{L^\infty}
\notag\\&+C(\al+\vps_0) T_\ast \sum\limits_{\ga_0\leq2}\sup\limits_{0\leq t\leq T_\ast}\|h'_2-h_2''\|_{L^\infty}\notag\\
\leq&\frac{1}{4}\sum\limits_{\ga_0\leq2}\left\{\sup\limits_{0\leq t\leq T_\ast}\|h'_1-h_1''\|_{L^\infty}+\al\sup\limits_{0\leq t\leq T_\ast}\|h'_2-h_2''\|_{L^\infty}\right\},
\end{align}
and
\begin{align}\label{dfh2}
\sum\limits_{\ga_0\leq2}&\al\sup\limits_{0\leq t\leq T_\ast}\|\CQ_2[g'_1,g'_2]-\CQ_2[g''_1,g''_2]\|_{L^\infty}\notag\\
\leq& C T_\ast \al\sum\limits_{\ga_0\leq2}\left\{\sup\limits_{0\leq t\leq T_\ast}\|h'_1-h_1''\|_{L^\infty}+\sup\limits_{0\leq t\leq T_\ast}\|h'_2-h_2''\|_{L^\infty}\right\}\notag\\
\leq& \frac{1}{4}\sum\limits_{\ga_0\leq2}\left\{\sup\limits_{0\leq t\leq T_\ast}\|h'_1-h_1''\|_{L^\infty}+\al\sup\limits_{0\leq t\leq T_\ast}\|h'_2-h_2''\|_{L^\infty}\right\},
\end{align}
for $T_\ast>0$ small enough.  Here, the following type of estimates have been also used:
\begin{align}
\|w_{l}&[Q(\mu^{\frac{1}{2}}g'_2,\mu^{\frac{1}{2}}g'_2)-Q(\mu^{\frac{1}{2}}g''_2,\mu^{\frac{1}{2}}g''_2)]\|_{L^\infty}\notag\\ \leq&\|w_{l}Q(\mu^{\frac{1}{2}}(g'_2-g_2''),\mu^{\frac{1}{2}}(g'_2-g_2''))\|_{L^\infty}
+\|w_{l}Q(\mu^{\frac{1}{2}}(g'_2-g_2''),\mu^{\frac{1}{2}}g_2'')\|_{L^\infty}
\notag\\&+\|w_{l}Q(\mu^{\frac{1}{2}}g_2'',\mu^{\frac{1}{2}}(g'_2-g_2''))\|_{L^\infty}\notag\\
\leq& C\left\{\|w_{l}[g'_2-g_2'']\|_{L^\infty}^2+2\|w_{l}g''_2\|_{L^\infty}\|w_{l}[g'_2-g_2'']\|_{L^\infty}\right\}.\notag
\end{align}
Consequently, \eqref{dfh1} and \eqref{dfh2} lead to
\begin{align}\notag
\|\CQ[g'_1,g'_2]-\CQ[g''_1,g''_2]\|_{\FY^{\vps_0}_{\al,T_\ast}}\leq\frac{1}{2}\|[g'_1,g'_2]-[g''_1,g''_2]\|_{\FY^{\vps_0}_{\al,T_\ast}}.
\end{align}
This together with \eqref{Q1} and \eqref{Q2} imply that  there exists $T_\ast>0$ such that $\CN$ is a contraction mapping on $\FY^{\vps_0}_{\al,T_\ast}$. Hence, there exists a unique $[g_1,g_2]\in\FY^{\vps_0}_{\al,T_\ast}$ such that
$$
[g_1,g_2]=\CN(g_1,g_2).
$$
This completes the proof of Theorem \ref{loc.th}.
\end{proof}

\section{Convergence to the steady state}\label{ge.sec}
Following the previous section regarding the local existence, the goal of this section is to establish the global existence of the Cauchy problem \eqref{beF} and \eqref{beFid}. More precisely, we shall construct a unique global-in-time solution around the self-similar profile, and also prove its large time asymptotic behavior with the exponential rate of convergence.

As in the previous sections, we denote the macroscopic part of $\tilde{g}$ by
$$\FP_0\tilde{g}=\{a+\Fb\cdot \xi+c(|\xi|^2-3)\}\sqrt{\mu}.$$
%The inner products of $\eqref{betg}_1$ with $\{1,\xi,\frac{1}{6}(|\xi|^2-3)\}\sqrt{\mu}$ with respect to $v$ give rise to
By taking the following velocity moments
\begin{equation*}
    \mu^{\frac{1}{2}}, \xi_j\mu^{\frac{1}{2}}, \frac{1}{6}(|\xi|^2-3)\mu^{\frac{1}{2}},
    A_{ij}=(\xi_i\xi_j-\frac{\de_{ij}}{3}|\xi|^2)\mu^{\frac{1}{2}}, B_i=\frac{1}{10}(|\xi|^2-5)\xi_i \mu^{\frac{1}{2}}
\end{equation*}
with {$1\leq i,j\leq 3$} for the equation
\begin{eqnarray}\label{betg.s4}
\left\{
\begin{array}{rll}
\begin{split}
&\pa_t \tilde{g}+e^{\beta t}\xi_1\pa_x\tilde{g}-\bet \na_\xi\cdot(\xi \tilde{g})+\frac{\bet}{2}|\xi|^2\tilde{g}-\al \xi_2\pa_{\xi_1}\tilde{g}+\frac{\al}{2}\xi_1\xi_2\tilde{g}+L\tilde{g}\\
&\qquad=\underbrace{\Ga(\tilde{g},\tilde{g})+\Ga(\tilde{g},\al G_1+\al G_R)
+\Ga(\al G_1+\al G_R,\tilde{g})}_{\tilde{\CF}},
\\&\qquad t>0,\ x\in\T,\ \xi\in \R^3,\\[2mm]
&\tilde{g}(0,x,\xi)=\tilde{g}_0=\frac{F_0(x,\xi)-G(\xi)}{\sqrt{\mu}},\ x\in\T,\ \xi\in \R^3,
\end{split}
\end{array}\right.
\end{eqnarray}
one sees that the coefficient functions $[a,\Fb,c]=[a,\Fb,c](t,x)$ satisfy
the fluid-type system
\begin{eqnarray}\label{ab}
\left\{
\begin{array}{rll}
\begin{split}
&\pa_t a+e^{\beta t}\pa_xb_1=0,\\
&\pa_t b_1+\beta b_1+e^{\beta t}\pa_x(a+2c)+\al b_2+e^{\beta t}\pa_x\int_{\R^3}\xi_1^2\sqrt{\mu}\FP_1\tilde{g}\,d\xi=0,\\
&\pa_t b_i+\beta b_i+e^{\beta t}\pa_x\lag A_{1i},\FP_1\tilde{g}\rag=0,\ i=2,3,
\end{split}
\end{array}\right.
\end{eqnarray}
\begin{align}\label{c}
\pa_t c&+\beta a+2\beta c+\frac{e^{\beta t}}{3}\pa_xb_1+\frac{e^{\beta t}}{6}\pa_x\int_{\R^3}\xi_1(|\xi|^2-3)\sqrt{\mu}\FP_1\tilde{g}\,d\xi
%\notag\\&
+\frac{\al}{3}\lag A_{12},\FP_1\tilde{g}\rag=0,
\end{align}
\begin{eqnarray}\label{Aij}
\left\{\begin{array}{rll}
\begin{split}
\pa_t&\lag A_{11},\FP_1\tilde{g}\rag +\frac{4e^{\beta t}}{3}\pa_xb_1
+e^{\beta t}\pa_x\lag \xi_1A_{11},\FP_1\tilde{g}\rag+2\beta\lag A_{11},\FP_1\tilde{g}\rag
\\&\qquad\qquad+\frac{4\al}{3}\lag A_{12},\FP_1\tilde{g}\rag+\lag L\tilde{g},A_{11}\rag=\lag \tilde{\CF},A_{11}\rag,\\
\pa_t&\lag A_{12},\FP_1\tilde{g}\rag +e^{\beta t}\pa_xb_2
+e^{\beta t}\pa_x\lag \xi_1A_{12},\FP_1\tilde{g}\rag+2\beta\lag A_{12},\FP_1\tilde{g}\rag\\
&\qquad\qquad+\al(a+2c)
+\al\lag A_{22},\FP_1\tilde{g}\rag+\lag L\tilde{g},A_{12}\rag=\lag \tilde{\CF},A_{12}\rag,\\
\pa_t&\lag A_{13},\FP_1\tilde{g}\rag +e^{\beta t}\pa_xb_3
+e^{\beta t}\pa_x\lag \xi_1A_{13},\FP_1\tilde{g}\rag+2\beta\lag A_{13},\FP_1\tilde{g}\rag\\
&\qquad\qquad
+\al\lag A_{23},\FP_1\tilde{g}\rag+\lag L\tilde{g},A_{13}\rag=\lag \tilde{\CF},A_{13}\rag,
\end{split}
\end{array}\right.
\end{eqnarray}
and
\begin{align}\label{B1}
\pa_t&\lag B_1,\FP_1\tilde{g}\rag +e^{\beta t}\pa_xc
+e^{\beta t}\pa_x\lag \xi_1B_1,\FP_1\tilde{g}\rag+\frac{\bet}{5}b_1+\frac{\beta}{10}\lag (|\xi|^2-3)\xi_1\sqrt{\mu},\FP_1\tilde{g}\rag\notag\\
&
+\frac{\al}{5}b_2+\frac{\al}{5}\int_{\R^3}\xi_1^2\xi_2\sqrt{\mu}\FP_1\tilde{g}\,d\xi
+\al\lag B_2,\FP_1\tilde{g}\rag+\lag L\tilde{g},B_1\rag=\lag \tilde{\CF},B_1\rag,
\end{align}
respectively. Here and in the sequel, we have denoted $[b_1,b_2,b_3]=\Fb.$
From \eqref{ab} and the initial condition \eqref{id.cons}, it follows that
\begin{align}\notag
\int_{\T}a\,dx=\int_{\T}b_i\,dx=0,\ 1\leq i\leq3.
\end{align}
We are now in a position to complete the

\begin{proof}[Proof of Theorem \ref{ge.th}]
The global existence of \eqref{betg.s4} follows from the standard continuation argument based on the local existence
which has been established in Section \ref{loc.sec} and the {\it a priori} estimate. In what follows, we intend to obtain
the {\it a priori} estimate
\begin{align}\label{apes}
\sup\limits_{0\leq t \leq T}&
e^{\la\beta_{\ga_0} t}\left\{\|w_{l}\pa_x^{\ga_0}g_1(t)\|_{L^\infty}+\al_{\ga_0} \|w_{l}\pa_x^{\ga_0}g_2(t)\|_{L^\infty}\right\}\notag\\
\leq& C\sum\limits_{\ga_0\leq2}\al_{\ga_0}\|w_{l}\pa^{\ga_0}_x\tilde{g}_{0}\|_{L^\infty},\ \ga_0\leq2,
\end{align}
under the {\it a priori} assumption that $[g_1,g_2]$ is a unique solution to the coupled system \eqref{beg1} and \eqref{beg2} and satisfies
\begin{align}\label{apas}
\sup\limits_{0\leq t \leq T}
e^{\la\beta_{\ga_0} t}\left\{\|w_{l}\pa_x^{\ga_0}g_1(t)\|_{L^\infty}+\al_{\ga_0} \|w_{l}\pa_x^{\ga_0}g_2(t)\|_{L^\infty}\right\}\leq \ka_0\al_{\ga_0},\ \ga_0\leq2.
\end{align}
Here, $0<\ka_0\ll 1$, and
\begin{align}\label{la}
0<\la<\frac{1}{4}\min\{1,\la_0\}
\end{align}
with $\la_0$ being determined as \eqref{dr.abc}.
Moreover, $\al_{\ga_0}$ and $\bet_{\ga_0}$  are defined as
\begin{align*}
\al_{\ga_0}=\left\{\begin{array}{rll}
\al,&\ \ga_0=0,\\
1,&\ \ga_0=1,2,\ \
\end{array}\right.
\beta_{\ga_0}=\left\{\begin{array}{rll}
\beta,&\ \ga_0=0,\\
1,&\ \ga_0=1,2.
\end{array}\right.
\end{align*}

\noindent\underline{Step 1. $L^\infty$ estimates.}
Recalling \eqref{h10m} and \eqref{h20m}, one has
\begin{align}\label{h1h2}
e^{\la\beta_{\ga_0} t}|h_1|\leq \sum\limits_{i=1}^3\CJ_i,\ \ e^{\la\beta_{\ga_0} t}|h_2|\leq \sum\limits_{i=4}^6\CJ_i,
\end{align}
with
\begin{align}
\CJ_1=e^{\la\beta_{\ga_0} t}\int_0^{t}e^{-\int_{s}^t\CA(\tau,V(\tau))\,d\tau}\left\{\chi_{M}w_l\CK
\left(\frac{h_1}{w_{l}}\right)\right\}(V(s))\,ds,\notag
\end{align}
\begin{align}
\CJ_2=&\frac{\beta}{2} e^{\la\beta_{\ga_0} t}\int_0^{t}e^{-\int_{s}^t\CA(\tau,V(\tau))\,d\tau}
|V(s)|^2\sqrt{\mu}(V(s))h_2(V(s))\,ds\notag\\&+\al e^{\la\beta_{\ga_0} t}\int_0^{t}e^{-\int_{s}^t\CA(\tau,V(\tau))\,d\tau}
\frac{V_1(s)V_2(s)}{2}\sqrt{\mu}(V(s))h_2(V(s))\,ds,\notag
\end{align}
\begin{align}
\CJ_3=e^{\la\beta_{\ga_0} t}\int_0^{t}e^{-\int_{s}^t\CA(\tau,V(\tau))\,d\tau}\left(w_{l}\tilde{H}\right)(V(s))\,ds,\notag
\end{align}
\begin{align}
\CJ_4=e^{\la\beta_{\ga_0} t}e^{-\int_{0}^t\CA(\tau,V(\tau))\,d\tau}(w_{l}\pa^{\ga_0}_x\tilde{g}_{0})(X(0),V(0)),\notag
\end{align}
\begin{align}
\CJ_5=e^{\la\beta_{\ga_0} t}\int_0^{t}e^{-\int_{s}^t\CA(\tau,V(\tau))\,d\tau}\left\{(1-\chi_{M})\mu^{-\frac{1}{2}}w_{l}\CK
\left(\frac{h_1}{w_{l}}\right)\right\}(V(s))\,ds,\notag
\end{align}
and
\begin{align}
\CJ_6=e^{\la\beta_{\ga_0} t}\int_0^{t}e^{-\int_{s}^t\CA(\tau,V(\tau))\,d\tau}\left[w_{l}K\left(\frac{h_2}{w_{l}}\right)\right](V(s))\,ds.
\notag
\end{align}
In what follows, we will compute $\CJ_i$$(1\leq i\leq6)$ term by term. Since $\beta>0$ is sufficiently small and $\la$ satisfies \eqref{la}, in view of \eqref{sglw}, \eqref{apas} and Lemmas \ref{op.es.lem} and \ref{CK},
one directly has
\begin{align}
\CJ_1\leq \frac{C}{l }\sum\limits_{0\leq s\leq t}e^{\la\beta_{\ga_0} s}\|h_1(s)\|_{L^\infty}
e^{\la\beta_{\ga_0} t}\int_0^{t}e^{-\frac{\nu_0}{2}(t-s)}e^{\la\beta_{\ga_0} s}\,ds
\leq \frac{C}{l }\sum\limits_{0\leq s\leq t}e^{\la\beta_{\ga_0} s}\|h_1(s)\|_{L^\infty},\notag
\end{align}
\begin{align}
\CJ_2\leq C\al\sum\limits_{0\leq s\leq t}e^{\la\beta_{\ga_0} s}\|h_2(s)\|_{L^\infty},\notag
\end{align}
\begin{align}
\CJ_3\leq C(\al+\ka_0)\sum\limits_{0\leq s\leq t}e^{\la\beta_{\ga_0} s}\left(\|h_1(s)\|_{L^\infty}+\|h_2(s)\|_{L^\infty}\right),\notag
\end{align}
\begin{align}
\CJ_4\leq e^{\la\beta_{\ga_0} t-\frac{\nu_0}{2}t}\|w_{l}\pa^{\ga_0}_x\tilde{g}_{0}\|_{L^\infty}\leq \|w_{l}\pa^{\ga_0}_x\tilde{g}_{0}\|_{L^\infty},\notag
\end{align}
\begin{align}
\CJ_5\leq C\sum\limits_{0\leq s\leq t}e^{\la\beta_{\ga_0} s}\|h_1(s)\|_{L^\infty}.\notag
\end{align}
For the delicate term $\CJ_6$, we first split it as
\begin{align}\notag
\CJ_6=e^{\la\beta_{\ga_0} t}\left\{\int_0^{t-\eps}+\int_{t-\eps}^{t}\right\}ds~e^{-\int_{s}^t\CA(\tau,V(\tau))d\tau}
\left[w_{l}K\left(\frac{h_2}{w_{l}}\right)\right](V(s))\eqdef \CJ_{6,1}+\CJ_{6,2},
\end{align}
where $\eps>0$ is small enough.
Then, for $\CJ_{6,2}$, by applying \eqref{sglw} and Lemma \ref{Ga}, one has
\begin{align}
|\CJ_{6,2}|\leq& C\sum\limits_{0\leq s\leq t}e^{\la\beta_{\ga_0} s}\|h_2(s)\|_{L^\infty}\int_{t-\eps}^{t}
e^{\la\beta_{\ga_0} t}e^{-\frac{\nu_0}{2} (t-s)}e^{-\la\beta_{\ga_0} s}\,ds\notag\\
\leq & C\eps\sum\limits_{0\leq s\leq t}e^{\la\beta_{\ga_0} s}\|h_2(s)\|_{L^\infty}\notag.
\end{align}
%by choosing $\eps$ such that $0<C\eps<\frac{1}{4}.$
However, $\CJ_{6,1}$ needs more attentions.  We rewrite $\CJ_{6,1}$ as
\begin{align}\notag
\CJ_{6,1}=e^{\la\beta_{\ga_0} t}\int_0^{t-\eps}ds~e^{-\int_{s}^t\CA(\tau,V(\tau))\,d\tau}
w_{l}(V(s))\int_{\R^3}\Fk(V(s),\xi_\ast)\frac{h_2(s,X(s),\xi_\ast)}{w_{l}(\xi_\ast)}\,d\xi_\ast.
\end{align}
Then as for obtaining \eqref{CI8}, we divide the computations in the following three cases.

\medskip
\noindent{\it Case 1.} $|V(s)|\geq M$ with $M$ suitably large.
In view of Lemma \ref{Kop}, it follows that
$$
\int\mathbf{k}_w(V,\xi_\ast)\,d\xi_\ast\leq \frac{C}{(1+|V|)}\leq \frac{C}{M},
$$
which implies
\begin{align}
\CJ_{6,1}\leq& \sum\limits_{0\leq s\leq t}e^{\la\beta_{\ga_0} s}\|h_2(s)\|_{L^\infty}\int_{0}^{t-\eps}ds~e^{\la\beta_{\ga_0} t}e^{-\frac{\nu_0}{2} (t-s)}e^{-\la\beta_{\ga_0} s}\int_{\R^3}\mathbf{k}_w(V(s),\xi_\ast)\,d\xi_\ast\notag\\
\leq&\frac{C}{M}\sum\limits_{0\leq s\leq t}e^{\la\beta_{\ga_0} s}\|h_2(s)\|_{L^\infty}\int_{0}^{t-\eps}e^{-\frac{\nu_0}{4} (t-s)} \,ds\notag \\
\leq & \frac{C}{M}\sum\limits_{0\leq s\leq t}e^{\la\beta_{\ga_0} s}\|h_2(s)\|_{L^\infty}.\label{CJ611}
\end{align}

\noindent{\it Case 2.} $|V(s)|\leq M$ and $|\xi_\ast|\geq 2M$. In this situation one has
$|V(s)-\xi_\ast|\geq M$, so it holds that
\begin{equation*}
\mathbf{k}_w(V,\xi_\ast)
\leq Ce^{-\frac{\vps M^2}{16}}\mathbf{k}_w(V,v_\ast)e^{\frac{\vps |V-\xi_\ast|^2}{16}},
\end{equation*}
where one sees that the integral $\int\mathbf{k}_w(V,\xi_\ast)e^{\frac{\vps |V-\xi_\ast|^2}{16}}\,d\xi_\ast$ is further bounded according to Lemma \ref{Kop}.
Thus as for obtaining \eqref{CJ611}, one has
\begin{equation}\notag
\begin{split}
\CJ_{6,1}\leq Ce^{-\frac{\vps M^2}{16}}\sum\limits_{0\leq s\leq t}e^{\la\beta_{\ga_0} s}\|h_2(s)\|_{L^\infty}.
\end{split}
\end{equation}
%To obtain the final bound for $\CI_8$, we are now in a position to handle the last case

\noindent{\it Case 3.} $|V|\leq M$, $|\xi_\ast|\leq 2M$. In this bad case, one possible way is to convert the bound in $L^\infty$-norm to the one in $L^2$-norm by an iteration approach. As what we have done in Case 3 in the proof of Lemma \ref{lifpri},
we compute $\CJ_{6,1}$ as follows
\begin{equation*}%\label{CJ613}
\begin{split}
\CJ_{6,1}\leq& e^{\la\beta_{\ga_0} t}\int_0^{t-\eps}ds~e^{-\int_{s}^t\CA(\tau,V(\tau))\,d\tau}\int_{|\xi_\ast|\leq 2M}\mathbf{k}_{w,p}(V(s),\xi_\ast)|h_2(s,X(s),\xi_\ast)|\,d\xi_\ast
\\&+\frac{1}{M}\sum\limits_{0\leq s\leq t}e^{\la\beta_{\ga_0} s}\|h_2(s)\|_{L^\infty},
\end{split}
\end{equation*}
where $\mathbf{k}_{w,p}$ is given by \eqref{km}. Next, by plugging the above estimates for $\CJ_4$, $\CJ_5$ and $\CJ_6$ into the second inequality of \eqref{h1h2},
one has
\begin{align}
e^{\la\beta_{\ga_0} t}|h_2|\leq& C\|w_{l}\pa^{\ga_0}_x\tilde{g}_{0}\|_{L^\infty}+
C\sum\limits_{0\leq s\leq t}e^{\la\beta_{\ga_0} s}\|w_{l}h_1(s)\|_{L^\infty}
\notag\\&+Ce^{\la\beta_{\ga_0} t}\int_0^{t-\eps}ds~{\bf1}_{|V(s)|\leq M}e^{-\int_{s}^t\CA(\tau,V(\tau))d\tau}e^{-\la\beta_{\ga_0} s}\notag\\
&\qquad\qquad\qquad\times\int_{|\xi_\ast|\leq 2M}\mathbf{k}_{w,p}(V(s),v_\ast)e^{\la\beta_{\ga_0} s}|h_2(s,X(s),\xi_\ast)|\,d\xi_\ast.\notag
\end{align}
Substituting it again, we get
\begin{align}
e^{\la\beta_{\ga_0} t}|h_2|\leq& Ce^{\la\beta_{\ga_0} t}\int_0^{t-\eps}ds~{\bf1}_{|V|\leq M}e^{-\int_{s}^t\CA(\tau,V(\tau))d\tau}
e^{-\la\beta_{\ga_0} s}\int_{|\xi_\ast|\leq 2M}\mathbf{k}_{w,p}(V(s),\xi_\ast)e^{\la\beta_{\ga_0} s}\notag
\\&\times\int_0^{s-\eps}ds_1~{\bf1}_{|V(s_1)|\leq M}e^{-\int_{s_1}^s\CA(\tau,V(\tau))d\tau}e^{-\la\beta_{\ga_0} s_1}\notag
\\&\qquad\times\int_{|\xi_\ast'|\leq2M}\mathbf{k}_{w,p}(\xi_\ast,\xi'_\ast)e^{\la\beta_{\ga_0} s_1}|h_2(s_1,X(s_1),\xi'_\ast)|\,d\xi_\ast d\xi'_\ast\notag\\
&+C\|w_{l}\pa^{\ga_0}_x\tilde{g}_{0}\|_{L^\infty}+
C\sup\limits_{0\leq s\leq t}e^{\la\beta_{\ga_0} s}\|w_{l}h_1(s)\|_{L^\infty}.
\label{h2ifit2}
\end{align}
On the other hand, thanks to \eqref{chlp2}, it follows
\begin{align}
X(s_1)&=X(s_1;s,X(s),\xi_\ast)\notag\\
&=e^{\beta(s-s_1)}\left(X(s;t,x,\xi)-(s-s_1)\xi_{\ast1}-\frac{1}{2}\al(s-s_1)^2\xi_{\ast2}\right)\notag\\
&= e^{\beta(s-s_1)}\left(e^{\beta(t-s)}\Big(x-(t-s)\xi_1-\frac{1}{2}\al(t-s)^2\xi_2\Big)-(s-s_1)\xi_{\ast1}-\frac{1}{2}\al(s-s_1)^2\xi_{\ast2}\right),\notag
\end{align}
and
\begin{align}
V_1(s_1)&=V_1(s_1;s,X(s),\xi_\ast)=e^{\beta(s-s_1)}\left(\xi_{\ast1}+\al\xi_{\ast2}(t-s)\right),\notag\\
V_i(s_1)&=V_i(s_1;s,X(s),\xi_\ast)=e^{\beta(s-s_1)}\xi_{\ast i},\ \ i=2,3.\notag
\end{align}
Therefor, for $s-s_1\geq\eps$, one has
\begin{align}\notag
\left|\frac{\pa\xi_{\ast1}}{\pa X(s_1)}\right|=\frac{e^{-\beta(s-s_1)}}{s-s_1}\leq\eps^{-1}e^{-\beta(s-s_1)}.
\end{align}
Let $y=X(s_1)$, then it holds that
\begin{align}\notag
&\left|e^{\beta(s-s_1)}X(s;t,x,\xi)-\frac{1}{2}\al e^{\beta(s-s_1)}(s-s_1)^2\xi_{\ast2}-y\right|\\
&\qquad\qquad\qquad\leq e^{\beta(s-s_1)}|s-s_1||\xi_{\ast1}|\leq 2Me^{\beta(s-s_1)}|s-s_1|,\notag
\end{align}
where we have used the fact that $|\xi_\ast|\leq 2M$ in order to further estimate the integral term on the right hand side of \eqref{h2ifit2}. Consequently, if $\ga_0=0$, we can bound the integral term on the right hand side of \eqref{h2ifit2} as
\begin{align}
Ce^{\la\beta t}&\int_0^{t}ds\int_0^{s-\eps}ds_1e^{-\frac{\nu_0}{2}(t-s)}e^{-\frac{\nu_0}{2}(s-s_1)}e^{-\la\beta s_1}
\notag\\&\times\int_{|\xi'_\ast|\leq 2M}\int_{|\xi_{\ast2}|^2+|\xi_{\ast3}|^2\leq 4M^2}\left(\int_{|\xi_{\ast1}|\leq 2M}|
e^{\la\beta s_1}h_2(s_1,y,\xi'_\ast)|^2d\xi_{\ast_1}\right)^{\frac{1}{2}}d\xi_{\ast_2}d\xi_{\ast_3}d\xi'_{\ast}\notag\\
\leq&Ce^{\la\beta t}\int_0^{t}ds\int_0^{s-\eps}ds_1e^{-\frac{\nu_0}{2}(t-s)}e^{-\frac{\nu_0}{2}(s-s_1)}e^{-\la\beta s_1}
\frac{e^{-\frac{\beta}{2}(s-s_1)}}{(s-s_1)^{\frac{1}{2}}}
\notag\\&\times\int_{|\xi'_\ast|\leq 2M}\left(\int_{\Om'}|e^{\la\beta s_1}
g_2(s_1,y,\xi'_\ast)|^2\,dy\right)^{\frac{1}{2}}d\xi'_{\ast}\notag\\
\leq&Ce^{\la\beta t}\int_0^{t}ds\int_0^{s-\eps}ds_1e^{-\frac{\nu_0}{2}(t-s)}e^{-\frac{\nu_0}{2}(s-s_1)}e^{-\la\beta s_1}
\frac{e^{-\frac{\beta}{2}(s-s_1)}}{(s-s_1)^{\frac{1}{2}}}\notag\\
&\times \left(M^{\frac{1}{2}}e^{\frac{\beta}{2}(s-s_1)}|s-s_1|^{\frac{1}{2}}+1\right)
\left(\int_{\R^3}\int_{\T}|e^{\la\beta s_1}g_2(s_1,y,\xi'_\ast)|^2\,dyd\xi'_{\ast}\right)^{\frac{1}{2}}\notag\\
\leq&C\sup\limits_{0\leq s\leq t}e^{\la\beta s}\|g_2(s)\|,\notag
\end{align}
where we have denoted
$$
\Om'=\left\{y\bigg||e^{\beta(s-s_1)}X(s;t,x,\xi)-\frac{1}{2}\al e^{\beta(s-s_1)}(s-s_1)^2\xi_{\ast2}-y|\leq2M e^{\beta(s-s_1)}|s-s_1|\right\}.
$$

While for $\ga_0=1,2$, because
%\begin{align}\notag
$\left|\frac{\pa X(s_1)}{\pa x}\right|=e^{\beta(t-s_1)}$,
%\end{align}
we can also bound the integral term on the right hand side of \eqref{h2ifit2} as
\begin{align}
Ce^{\la t}&\int_0^{t}ds\int_0^{s-\eps}ds_1\,e^{-\frac{\nu_0}{2}(t-s)}e^{-\frac{\nu_0}{2}(s-s_1)}e^{-\la s_1}
\int_{|\xi'_\ast|\leq 2M}\left(\int_{|\xi_{\ast1}|\leq 2M}
|e^{\la s_1}h_2(s_1,y,\xi'_\ast)|^2\,d\xi_{\ast_1}\right)^{\frac{1}{2}}d\xi'_{\ast}\notag\\
\leq&Ce^{\la t}\int_0^{t}ds\int_0^{s-\eps}ds_1e^{-\frac{\nu_0}{2}(t-s)}e^{-\frac{\nu_0}{2}(s-s_1)}e^{-\la s_1}
\frac{e^{-\frac{\beta}{2}(s-s_1)}}{(s-s_1)^{\frac{1}{2}}}e^{\ga_{0}\beta(t-s_1)}\notag\\&\times\int_{|\xi'_\ast|\leq 2M}\left(\int_{\Om'}
|e^{\la s_1}\pa^{\ga_0}_yg_2(s_1,y,\xi'_\ast)|^2\,dy\right)^{\frac{1}{2}}d\xi'_{\ast}\notag\\
\leq&Ce^{\la t}\int_0^{t}ds\int_0^{s-\eps}ds_1\,e^{-\frac{\nu_0}{2}(t-s)}e^{-\frac{\nu_0}{2}(s-s_1)}e^{-\la s_1}
\frac{e^{-\frac{\beta}{2}(s-s_1)}}{(s-s_1)^{\frac{1}{2}}}e^{\ga_{0}\beta(t-s_1)}\notag\\
&\times \left(M^{\frac{1}{2}}e^{\frac{\beta}{2}(s-s_1)}|s-s_1|^{\frac{1}{2}}+1\right)\left(\int_{\R^3}\int_{\T}
|e^{\la s_1}\pa^{\ga_{0}}_yg_2(s_1,y,\xi'_\ast)|^2\,dyd\xi'_{\ast}\right)^{\frac{1}{2}}\notag\\
\leq&C\sup\limits_{0\leq s\leq t}e^{\la s}\|\pa^{\ga_{0}}_yg_2(s)\|,\notag
\end{align}
where we notice that $0<\beta\sim\al^2 \ll \nu_0$.

By plugging the above estimates into \eqref{h1h2},
we then conclude
\begin{align}\label{h1if}
e^{\la\beta_{\ga_0} t}|h_1|\leq C(\al+\ka_0)\sum\limits_{0\leq s\leq t}e^{\la\beta_{\ga_0} s}\|w_{l}h_2(s)\|_{L^\infty},
\end{align}
and
%and by choosing $\eps$ such that $0<C\eps<\frac{1}{4}$, it follows
\begin{align}\label{h2if}
e^{\la\beta_{\ga_0} t}|h_2|\leq C\|w_{l}\pa^{\ga_0}_x\tilde{g}_{0}\|_{L^\infty}+
C\sum\limits_{0\leq s\leq t}e^{\la\beta_{\ga_0} s}\|w_{l}h_1(s)\|_{L^\infty}+C\sup\limits_{0\leq s\leq t}e^{\la\beta_{\ga_0} s}\|\pa_x^{\ga_0}g_2(s)\|.
\end{align}

\medskip
\noindent\underline{Step 2. $L^2$ estimates.}
Recall $\sqrt{\mu}\tilde{g}=g_1+\sqrt{\mu}g_2$.
We now denote
\begin{align}\notag
d_{ij}=\lag A_{ij},\FP_1\tilde{g}\rag=\lag A_{ij}\mu^{-\frac{1}{2}},\bar{\FP}_1g_1\rag+\lag A_{ij},\FP_1g_2\rag,
\end{align}
and we also use the notations
\begin{align}\notag
\FP_1\tilde{g}=\bar{\FP}_1g_1+\FP_1g_2, \ \FP_0\tilde{g}=\bar{\FP}_0g_1+\FP_0g_2,
\end{align}
with
$$
\bar{\FP}_1g_1=g_1-\bar{\FP}_0g_1,\ \textrm{and}\ \ \bar{\FP}_0g_1=(a_{1}+\Fb_{1}\cdot \xi+c_{1}(|\xi|^2-3))\mu(\xi).
$$
We now clarify the relation for $\FP_0g_2$, $\FP_1g_2$, $\FP_0\tilde{g}$ and $\FP_1\tilde{g}$. Noticing
$\sqrt{\mu}\tilde{g}=g_1+\sqrt{\mu}g_2$, one sees that
\begin{align*}
\FP_0g_2=\FP_0\tilde{g}-\FP_0\left(\frac{g_1}{\sqrt{\mu}}\right).
\end{align*}
Therefore it holds that
\begin{align}\label{pp1}
\|\FP_0g_2\|\leq\|\FP_0\widetilde{g}\|+\left\|\FP_0\left(\frac{g_1}{\sqrt{\mu}}\right)\right\|\leq
\|\FP_0\widetilde{g}\|+C\|w_{l}g_1\|_{L^\infty},\ \textrm{for} \ l >\frac{5}{2},
\end{align}
in particular,
\begin{align}\label{pp2}
&\|[a_2,\Fb_2]\|=\|[a-a_1,\Fb-\Fb_1]\|\leq \|\pa_x[a,\Fb]\|+C\|w_{l}g_1\|_{L^\infty},\notag\\
&\|c_2\|\leq \|c\|+C\|w_{l}g_1\|_{L^\infty},\ \textrm{for} \ l >\frac{5}{2}.
\end{align}
Likewise, one obtains that 
\begin{align}\label{pp3}
\|\lag\FP_1\tilde{g},|\xi|^3\mu^{\frac{1}{2}}\rag\|\leq\|\FP_1g_2\|+C\|w_{l}g_1\|_{L^\infty},\ \textrm{for} \ l >3.
\end{align}
Taking the summation of
$(\eqref{ab},\frac{\al e^{-\beta t}}{6}\int_0^x d_{12}\,dy)$ and
$(\eqref{c},c)$ and using integration by parts, one has
\begin{align}
\frac{d}{dt}&(b_1,\frac{\al e^{-\beta t}}{6}\int_0^x d_{12}\,dy)-( b_1,\frac{\al e^{-\beta t}}{6}\int_0^x \pa_td_{12}\,dy)
+\beta( b_1,\frac{\al e^{-\beta t}}{6}\int_0^x d_{12}\,dy)\notag
\\&+\beta (b_1,\frac{\al e^{-\beta t}}{6}\int_0^x d_{12}\,dy)
+\frac{\al}{6}(\pa_xa, \int_0^x d_{12}\,dy)+\frac{\al^2}{6}(b_2, \int_0^x d_{12}\,dy)\notag\\&-\frac{\al}{6}(\int_{\R^3}\xi_1^2\sqrt{\mu}\FP_1\tilde{g}\,d\xi, d_{12})
+\frac{1}{2}\frac{d}{dt}\|c\|^2+\beta (a,c)+2\beta \|c\|^2-\frac{e^{\beta t}}{3}(b_1,\pa_xc)
\notag\\&-\frac{e^{\beta t}}{6}(\int_{\R^3}\xi_1(|\xi|^2-3)\sqrt{\mu}\FP_1\tilde{g}\,d\xi,\pa_xc)=0,\label{es.c.p1}
\end{align}
where we have used the cancellation
\begin{align}
(\pa_xc,\int_0^x d_{12}\,dy)+(d_{12},c)=0.\notag
\end{align}
On the other hand, using the second equation of \eqref{Aij} and Lemma \ref{our.lem}, we have
\begin{align}\label{es.c.p2}
-( b_1,\frac{\al e^{-\beta t}}{6}\int_0^x \pa_td_{12}\,dy)=&\frac{2\al}{9}( b_1,b_2)
+\frac{\al}{6}(b_1,\lag \xi_1A_{12},\FP_1\tilde{g}\rag)+\frac{\beta\al e^{-\beta t}}{3}(b_1,\int_0^xd_{12} \,dy)\notag\\
&+\frac{\al^2 e^{-\beta t}}{6}(b_1,\int_0^x(a+2c)\,dy)
+\frac{\al^2 e^{-\beta t}}{6}(b_1,\int_0^x\lag A_{22},\FP_1\tilde{g}\rag \,dy)
\notag\\&+\frac{\al b_0 e^{-\beta t}}{3}(b_1,\int_0^xd_{12}\,dy)-\frac{\al e^{-\beta t}}{6}(b_1,\int_0^x\lag \tilde{\CF},A_{12}\rag \,dy).
\end{align}
As a consequence, \eqref{es.c.p1}, \eqref{es.c.p2} and \eqref{pp3} lead us to
\begin{align}\label{dis.c}
\frac{d}{dt}&\left\{\|c\|^2+(b_1,\frac{\al e^{-\beta t}}{3}\int_0^x d_{12}\,dy)\right\}
+2\bet\|c\|^2\notag\\
\leq& \eta_0 \|\lag B_1,\FP_1\tilde{g}\rag\|^2+C\al\|\lag \xi_1A_{12},\FP_1\tilde{g}\rag\|^2+\al\|d_{22}\|^2+\al\|d_{11}\|^2+\al\|d_{12}\|^2\notag\\&+\frac{C}{\eta_0}e^{2\beta t}\|\pa_x[a,\Fb,c]\|^2+C\al\|\lag \tilde{\CF},A_{12}\rag\|^2\notag\\
\leq& (\eta_0 +\al)(\|\FP_1g_2\|^2+\|w_{l}g_1\|^2_{L^\infty})+\frac{C}{\eta_0}e^{2\beta t}\|\pa_x[a,\Fb,c]\|^2+C\al^2\|\lag \tilde{\CF},A_{12}\rag\|^2,
\end{align}
where $\eta_0>0$ is an arbitrary constant and we also have used the   Poincar\'e inequality $\|u-\int_{\T}udx\|\leq C\|\pa_x u\|.$

We next compute carefully the last term on the right hand side of \eqref{dis.c}. First of all,
recalling the definition for $\tilde{\CF}$, one has
\begin{align*}%\label{HA12}
\|\lag \tilde{\CF},A_{12}\rag\|^2\lesssim& \int_{\T}\left(\int_{\R^3}|Q(\mu^{\frac{1}{2}}g_2,\mu^{\frac{1}{2}}g_2)||\xi_1\xi_2|\,d\xi\right)^2 dx
+\int_{\T}\left(\int_{\R^3}|Q(g_1,g_1)||\xi_1\xi_2|\,d\xi\right)^2 dx\notag\\
&+\int_{\T}\left(\int_{\R^3}|Q(g_1,\mu^{\frac{1}{2}}g_2)||\xi_1\xi_2|\,d\xi\right)^2 dx
+\int_{\T}\left(\int_{\R^3}|Q(\mu^{\frac{1}{2}}g_2,g_1)||\xi_1\xi_2|\,d\xi\right)^2 dx
\\
&+\al^2\int_{\T}\left(\int_{\R^3}Q(g_1,\mu^{\frac{1}{2}} (G_1+G_R)
+Q(\mu^{\frac{1}{2}} (G_1+G_R),g_1)\xi_1\xi_2\,d\xi\right)^{2} dx
\notag\\
&+\al^2\int_{\T}\left(\int_{\R^3}Q(\mu^{\frac{1}{2}}g_2,\mu^{\frac{1}{2}} (G_1+G_R)
+Q(\mu^{\frac{1}{2}} (G_1+G_R),\mu^{\frac{1}{2}}g_2)\xi_1\xi_2\,d\xi\right)^{2} dx\notag\\
\eqdef&\sum\limits_{i=1}^6\CH_i.\notag
\end{align*}
We then compute $\CH_i$$(1\leq i\leq6)$ term by term.
For $\CH_1$, Since $l >3$, thanks to Lemma \ref{op.es.lem} and the {\it a priori} assumption \eqref{apas}, it follows that
\begin{align}\notag
\CH_1\lesssim&  \|w_{l}Q(\mu^{\frac{1}{2}}g_2,\mu^{\frac{1}{2}}g_2)\|^2_{L^\infty}
\int_{\T}\left(\int_{\R^3}w_{-l}(\xi)|\xi_1\xi_2|\,d\xi\right)^2dx\lesssim \|w_l\mu^{\frac{1}{2}}g_2\|_{L^\infty}^4\lesssim \ka_0^2\|w_lg_2\|^2_{L^\infty}.
\end{align}
Similarly, it holds that
\begin{equation}\notag
\CH_2\lesssim \|w_{l}Q(g_1,g_1)\|^2_{L^\infty}
\int_{\T}\left(\int_{\R^3}w_{-l }|\xi_1\xi_2|\,d\xi\right)^2dx\lesssim \|w_{l}g_1\|^4_{L^\infty}\lesssim (\ka_0\al)^2\|w_{l}g_1\|^2_{L^\infty},
\end{equation}
and
\begin{align}
\CH_3,\CH_4\lesssim& \left(\|w_{l}Q(g_1,\mu^{\frac{1}{2}}g_2)\|^2_{L^\infty}+\|w_{l}Q(\mu^{\frac{1}{2}}g_2,g_1)\|^2_{L^\infty}\right)
\int_{\T}\left(\int_{\R^3}w_{-l }|\xi_1\xi_2|\,d\xi\right)^2dx\notag
\\ \lesssim& \|w_{l}g_1\|^2_{L^\infty}\|w_{l}g_2\|^2_{L^\infty}\lesssim \ka_0^2\|w_{l}g_1\|^2_{L^\infty}.\notag
\end{align}
Next, applying Lemma \ref{op.es.lem} and Theorem \ref{st.sol}, one directly has
\begin{equation}
\CH_5, \CH_6\lesssim\al^2\|w_{l}g_1\|_{L^\infty}^2+\al^2\|w_{l}g_2\|_{L^\infty}^2.\notag
\end{equation}

Putting now the above estimates for $\CH_i$$(1\leq i\leq6)$ into \eqref{dis.c}, one has
\begin{align}
\frac{d}{dt}&\left\{\|c\|^2+(b_1,\frac{\al e^{-\beta t}}{3}\int_0^x d_{12}\,dy)\right\}+2\beta\|c\|^2\notag\\
\leq& C(\eta_0+\al^2)\|\FP_1g_2\|^2+C(\al^4+\al^2\ka_0^2)\|w_{l}g_2\|_{L^\infty}^2\notag\\&+C(\al^2+\ka_0^2+\eta_0)\|w_{l}g_1\|_{L^\infty}^2
+\frac{C}{\eta_0}e^{2\beta t}\|\pa_x[a,\Fb,c]\|^2,\label{dis.c.p2}
\end{align}
where the Poincar\'e inequality has been also used. Thus \eqref{dis.c.p2} further implies that for $0<\la\leq \frac{1}{4}$,
\begin{align}
\sup\limits_{0\leq s\leq t}&e^{2\la\beta s}\|c(s)\|^2\notag\\
\leq& \|c(0)\|^2+\al\|[b_1,d_{12}](0)\|^2+\al\sup\limits_{0\leq s\leq t}e^{2\la\beta s}\|[b_1,d_{12}]\|^2
\notag\\&+C(\eta_0+\al^2)e^{2\la\beta t}\int_0^te^{-2\beta (t-s)}e^{-2\la\beta s}e^{2\la\beta s}\|\FP_1g_2\|^2(s)\,ds
\notag\\&+C(\al^2+\ka_0^2+\eta_0)e^{2\la\beta t}\int_0^te^{-2\beta (t-s)}e^{-2\la\beta s}e^{2\la\beta s}\|w_{l}g_1\|_{L^\infty}^2\,ds
\notag\\&+C(\al^4+\ka_0^2)e^{2\la\beta t}\int_0^te^{-2\beta (t-s)}e^{-2\la\beta s}e^{2\la\beta s}\|w_{l}g_2\|_{L^\infty}^2\,ds
\notag\\&+\frac{C}{\eta_0}e^{2\la\beta t}\int_0^te^{-2\beta (t-s)}e^{4\beta s}e^{-2\la s}e^{2\la s}\|\pa_x[a,\Fb,c]\|^2\,ds
\notag\\ \leq& \|c(0)\|^2+\al\|[b_1,d_{12}](0)\|^2
+\al\sup\limits_{0\leq s\leq t}e^{2\la\beta s}\|[b_1,d_{12}]\|^2\notag
\\&+\frac{C(\al^2+\eta_0)}{\beta}\sup\limits_{0\leq s\leq t}e^{2\la\beta s}\|\FP_1g_2(s)\|^2
+\frac{C(\al^2+\ka_0^2+\eta_0)}{\beta}\sup\limits_{0\leq s\leq t}e^{2\la\beta s}\|w_{l}g_1(s)\|_{L^\infty}^2
\notag\\&+\frac{C(\al^4+\al^2\ka_0^2)}{\beta}\sup\limits_{0\leq s\leq t}e^{2\la\beta s}\|w_{l}g_2(s)\|_{L^\infty}^2
+\frac{C}{\eta_0}\sup\limits_{0\leq s\leq t}e^{2\la s}\|\pa_x[a,\Fb,c]\|^2.\notag
\end{align}
Here we have used the following estimate
\begin{align}\label{decay.ip}
e^{2\la\beta t}&\int_{0}^te^{-2\beta(t-s)}e^{4\beta s}e^{-2\la s}e^{-2\la s}\|\pa_x[a,\Fb,c]\|^2\,ds\notag\\
\leq& \sup\limits_{0\leq s\leq t}e^{2\la s}\|\pa_x[a,\Fb,c]\|^2
\int_{0}^te^{-2\beta(t-s)}e^{2\la\beta t}e^{4\beta s}e^{-2\la s}\,ds\notag\\ \leq& C\sup\limits_{0\leq s\leq t}e^{2\la s}\|\pa_x[a,\Fb,c]\|^2,
\end{align}
due to the fact that $0<\beta\sim \al^2\ll 1$. Consequently, it follows that
\begin{align}\label{cd12.decay}
\al\sup\limits_{0\leq s\leq t}&e^{\la\beta s}\|c(s)\|\notag\\
\leq& \al\|c(0)\|^2+\al^{\frac{3}{2}}\|[b_1,d_{12}](0)\|+\al^{\frac{3}{2}}\sup\limits_{0\leq s\leq t}e^{\la\beta s}\|[b_1,d_{12}]\|
\notag\\&+ C(\al^2+\eta_0)^{\frac{1}{2}}\sup\limits_{0\leq s\leq t}e^{\la\beta s}\|\FP_1g_2(s)\|+C(\al^2+(\ka_0\al)^2+\eta_0)^{\frac{1}{2}}\sup\limits_{0\leq s\leq t}e^{\la\beta s}\|w_{l}g_1(s)\|_{L^\infty}
\notag\\&+C(\al^2+\al\ka_0)\sup\limits_{0\leq s\leq t}e^{\la\beta s}\|w_{l}g_2(s)\|_{L^\infty}
+\frac{C\al}{\sqrt{\eta_0}}\sup\limits_{0\leq s\leq t}e^{\la s}\|\pa_x[a,\Fb,c]\|.
\end{align}

On the other hand, taking the inner product of the first equation of \eqref{beg2} and $\FP_1g_2$ with respect to $(x,v)$
over $\T\times \R^3$, we also have by Lemma \ref{es.L}, \eqref{pp1}  and \eqref{pp2} that
\begin{align}
\frac{d}{dt}\|\FP_1g_2\|^2+\frac{\de_0}{2}\|\FP_1g_2\|^2
\leq& C\|w_{l}g_1\|_{L^\infty}^2+Ce^{2\beta t}\|\pa_x\FP_0g_2\|^2+C\al^2\|\FP_0g_2\|^2\notag\\
\leq& C\|w_{l}g_1\|_{L^\infty}^2+Ce^{2\beta t}\|\pa_x\FP_0g_2\|^2+C\al^2\|c\|^2,\notag
\end{align}
which further yields
\begin{align}\label{mic.decay}
\sup\limits_{0\leq s\leq t}e^{\la\beta s}\|\FP_1g_2(s)\|
\leq& \|\FP_1g_2(0)\|+C\sup\limits_{0\leq s\leq t}e^{\la\beta s}\|w_{l}g_1(s)\|_{L^\infty}
+C\sup\limits_{0\leq s\leq t}e^{\la s}\|\pa_x\FP_0g_2\|
\notag\\&+C\al\sup\limits_{0\leq s\leq t}e^{\la\beta s}\|c(s)\|.
\end{align}
Therefore, putting \eqref{cd12.decay} and \eqref{mic.decay} together, we have that for $\al^2=\eta_0$,
\begin{align}\label{zerol2}
\al\sup\limits_{0\leq s\leq t}&e^{\la\beta s}\|c(s)\|
+\al\sup\limits_{0\leq s\leq t}e^{\la\beta s}\|\FP_1g_2(s)\|\notag\\
\leq& \al\|c(0)\|^2+\al^{\frac{3}{2}}\|[b_1,d_{12}](0)\|+C\al\|\FP_1\tilde{g}_0\|
+C(\al+\ka_0\al)\sup\limits_{0\leq s\leq t}e^{\la\beta s}\|w_{l}g_1(s)\|_{L^\infty}
\notag\\&+C(\al^2+\ka_0\al)\sup\limits_{0\leq s\leq t}e^{\la\beta s}\|w_{l}g_2(s)\|_{L^\infty}
+C\al\sup\limits_{0\leq s\leq t}e^{\la s}\|w_l\pa_xg_1(s)\|_{L^\infty}
\notag\\&+C\sup\limits_{0\leq s\leq t}e^{\la s}\|\pa_x[a,\Fb,c]\|.
\end{align}

Let us now turn to deduce the higher order energy estimate. For this, we claim that
\begin{align}\label{1st.decay}
\sum\limits_{1\leq\ga_0\leq2}\sup\limits_{0\leq s\leq t}e^{\la s}\|\pa^{\ga_0}_xg_2\|(s)
\leq&\sum\limits_{1\leq\ga_0\leq2}\|\pa^{\ga_0}_x\tilde{g}_0\|
+C\sum\limits_{1\leq\ga_0\leq2}\sup\limits_{0\leq s\leq t}e^{\la s}\|w_{l}\pa^{\ga_0}_xg_1(s)\|_{L^\infty}
\notag\\&+C(\al+\ka_0)\sup\limits_{0\leq s\leq t}e^{\la s}\|w_{l}\pa_xg_2(s)\|_{L^\infty}.
%\notag\\&+C\al\sup\limits_{0\leq s\leq t}e^{\la\beta s}\|c(s)\|.
\end{align}
Indeed, for some $\la_0>0$, the inner products 
\begin{center}
$(\pa_x\eqref{ab}_2,e^{\beta t}\pa^2_x a)$, $(\pa_x\eqref{B1},e^{\beta t}\pa^2_x c)$ and $(\pa_x\eqref{Aij}_i,e^{\beta t}\pa^2_xb_i)$
\end{center} 
together with \eqref{ab} and \eqref{c} give rise to
\begin{align}\label{dr.abc}
\frac{d}{dt}\CE_{int}+\la_0 \sum\limits_{1\leq \ga_0\leq2}e^{2\beta t}\|\pa^{\ga_0}_x[a,\Fb,c]\|^2\lesssim&
\sum\limits_{1\leq \ga_0\leq2}e^{2\beta t}\|\lag \varsigma_{i}, \pa^{\ga_0}_x\FP_1\tilde{g}\rag\|^2
\notag\\&+\|\lag L\pa_x\tilde{g},\varsigma_{i}\rag\|^2+\|\lag \pa_x\tilde{\CF},\varsigma_{i}\rag\|^2,
\end{align}
where we have set
\begin{align}\notag
\CE_{int}=\sum\limits_{i=1}^3(\pa_x\lag A_{1i},\FP_1\tilde{g}\rag,e^{\beta t}\pa^2_xb_i)+(\pa_x\lag B_1,\FP_1\tilde{g}\rag,e^{\beta t}\pa^2_xc)+\ka_1(\pa_xb_1,e^{\beta t}\pa^2_xa),
\end{align}
and $\varsigma_i$ denote all $A_{ij}$, $B_i$ and $\xi_1^2\xi_2\mu^{\frac{1}{2}}$ etc. appearing
in \eqref{ab}, \eqref{c}, \eqref{Aij} and \eqref{B1}. Moreover, the Poincar\'e inequality
$$
\|\pa_x[a,\Fb,c] \|\leq C\|\pa^2_x[a,\Fb,c] \|
$$
has been used here.
Furthermore, performing the similar calculations as for treating
$\|\lag \tilde{\CF},\varsigma_{i}\rag\|^2$ in \eqref{dis.c} above, one has
\begin{align}\label{Hzi}
\|\lag \pa_x\tilde{\CF},\varsigma_{i}\rag\|^2\lesssim& (\ka^2_0+\al^2)\left(\|w_{l}\pa_xg_2\|^2_{L^\infty}+\|w_{l}\pa_xg_1\|^2_{L^\infty}\right).
\end{align}
%it should be pointed out that the only difference between \eqref{Hzi} and \eqref{HA12} is that there is a
%$\al^2\|c\|^2$ while computing the term like $\CH_5$ in the former.
Lemma \ref{Ga} and
Lemma \ref{op.es.lem} with $l>3$ imply that
\begin{align}\label{Lzi}
\|\lag L\pa_x\tilde{g},\varsigma_{i}\rag\|\lesssim\|w_{l}\pa_xg_1\|_{L^\infty}+\|\pa_x\FP_1g_2\|.
\end{align}
Consequently, plugging \eqref{Hzi} and \eqref{Lzi} into \eqref{dr.abc} and employing \eqref{pp3}, we arrive at
\begin{align}\label{dr.abc.p1}
e^{-2\beta t}\frac{d}{dt}\CE_{int}+\la_0 \sum\limits_{1\leq\ga_0\leq2}\|\pa^{\ga_0}_x[a,\Fb,c]\|^2\leq&
C \sum\limits_{1\leq\ga_0\leq2}\|\pa^{\ga_0}_x\FP_1g_2\|^2
+C\sum\limits_{1\leq\ga_0\leq2}\|w_{l}\pa^{\ga_0}_xg_1\|_{L^\infty}^2
\notag\\&+C(\ka^2_0+\al^2)\|w_{l}\pa_xg_2\|^2_{L^\infty}.
\end{align}

On the other hand, energy estimate on \eqref{beg2} leads us to
\begin{align}\label{eng.g2p1}
\frac{d}{dt}&\|\pa^{\ga_0}_xg_2\|^2+\la\|\pa^{\ga_0}_x\FP_1g_2\|^2
\leq C_{\eta_1}\|\pa^{\ga_0}_x g_1\|^2+(C\al^2+\eta_1)\|\pa^{\ga_0}_x\FP_0g_2\|^2,
\end{align}
for $\ga_0=1,2$, where $\eta_1$ is positive and suitably small.

Combing \eqref{dr.abc.p1} and \eqref{eng.g2p1} together, one has that for $\ka_2>0$ and suitably small,
\begin{align}
\frac{d}{dt}\big\{\sum\limits_{1\leq\ga_0\leq2}\|\pa^{\ga_0}_xg_2\|^2&+\ka_2e^{-2\beta t}\CE_{int}\big\}+\sum\limits_{1\leq\ga_0\leq2}\la_0\left\{\|\pa^{\ga_0}_x\FP_1g_2\|^2+\|\pa^{\ga_0}_x[a,\Fb,c]\|^2\right\}
\notag\\
\leq& C \sum\limits_{1\leq\ga_0\leq2}\|w_{l}\pa^{\ga_0}_x g_1\|_{L^\infty}^2+C(\ka^2_0+\al^2)\|w_{l}\pa_xg_2\|^2_{L^\infty}.\notag
\end{align}
From the above energy inequality, we further obtain that for $0<\la\leq\frac{\la_0}{4}$,
\begin{align}\label{1st.decay.p1}
\sup\limits_{0\leq s\leq t}&e^{2\la s}\sum\limits_{1\leq\ga_0\leq2}\|\pa^{\ga_0}_xg_2(s)\|^2
\notag\\ \leq&\sum\limits_{1\leq\ga_0\leq2}\|\pa^{\ga_0}_xg_2(0)\|^2
+C \sum\limits_{1\leq\ga_0\leq2}e^{2\la t}\int_{0}^te^{-\la_0(t-s)}\|w_{l}\pa^{\ga_0}_x g_1(s)\|_{L^\infty}^2\,ds
\notag\\&+C(\ka^2_0+\al^2) e^{2\la t}\int_{0}^te^{-\la_0(t-s)}\|w_{l}\pa_x g_2(s)\|_{L^\infty}^2ds.
%\notag\\ \leq&\|\pa_xg_2(0)\|^2+C\sup\limits_{0\leq s\leq t}e^{2\la s}\int_{0}^te^{-\la_0(t-s)}\|w_{l}\pa_x g_1(s)\|_{L^\infty}^2ds.
\end{align}
Therefore \eqref{1st.decay} follows from \eqref{1st.decay.p1} and the similar estimate as \eqref{decay.ip}.

\medskip
\noindent\underline{Step 3. Combination.}
We are now in position to obtain our final estimates \eqref{apes}. To do so, for $\ga_0=0$, we get from
the summation of \eqref{h1if}, $\al\times\eqref{h2if}$ and \eqref{zerol2} that
\begin{align}\label{hzero}
e^{\la\beta t}|w_{l}g_1|+\al e^{\la\beta t}|w_{l}g_2|\leq& C\al\|w_{l}\tilde{g}_{0}\|_{L^\infty}
+\al\|[c,d_{12}](0)\|+\al\|\FP_1\tilde{g}_0\|\notag
\\&+C\sup\limits_{0\leq s\leq t}e^{\la s}\|w_{l}\pa_xg_1(s)\|_{L^\infty}
+C\sup\limits_{0\leq s\leq t}e^{\la s}\|\pa_x[a,\Fb,c]\|.
\end{align}
As to $\ga_0=1,2$, we set $\ka_3>0$ sufficiently small, take the summation of \eqref{h1if} and $\ka_3\times\eqref{h2if}$, and plug \eqref{1st.decay} into the resultant inequality, so as to obtain
\begin{align}\label{h1st}
\sum\limits_{1\leq\ga_0\leq2}e^{\la t}\left\{|w_{l}\pa^{\ga_0}_xg_1|+ |w_{l}\pa^{\ga_0}_xg_2|\right\}
\leq&C\sum\limits_{1\leq\ga_0\leq2}\left\{\|\pa^{\ga_0}_x\tilde{g}_0\|+
\|w_{l}\pa^{\ga_0}_x\tilde{g}_{0}\|_{L^\infty}\right\}.
\end{align}
On the other hand, it follows that
\begin{align}\label{pl1}
\|\FP_1g_2\|\leq C\|w_{l}g_2\|_{L^\infty},\ \textrm{for}\ l >\frac{3}{2},
\end{align}
and for $l >\frac{5}{2}$, one has
\begin{align}\label{pl2}
\|\pa_x[a,\Fb,c]\|\leq& C\|\pa_x\FP_0\tilde{g}\|\leq C\|\pa_x g_2\|+C\|w_{l}\pa_xg_1\|_{L^\infty}
\notag\\ \leq& C\|w_{l}\pa_x g_2\|_{L^\infty}+C\|w_{l}\pa_xg_1\|_{L^\infty},
\end{align}
and
\begin{align}\label{pl3}
\|[c,b_1,d_{12}](0)\|\leq C\|w_{l}g_2\|_{L^\infty}+C\|w_{l}g_1(0)\|_{L^\infty}\leq C\|w_{l}g_2\|_{L^\infty}.
\end{align}
Finally, putting \eqref{hzero}, \eqref{h1st}, \eqref{pl1}, \eqref{pl2} and \eqref{pl3} together and adjusting constants, we have
\begin{align}\notag
\sum\limits_{\ga_0\leq2}e^{\la\beta_{\ga_0} t}\left\{|w_{l}\pa_x^{\ga_0}g_1|+ \al^{\ga_0}|w_{l}\pa_x^{\ga_0}g_2|\right\}
\leq C\sum\limits_{\ga_0\leq2}\al^{\ga_0}\|w_{l}\pa_x^{\ga_0}\tilde{g}_{0}\|_{L^\infty}.
\end{align}
Thus \eqref{apes} is valid.

\medskip
\noindent \underline{Step 4. Non-negativity.} We now turn to prove that the unique global solution constructed above is non-negative, i.e.
$e^{3\beta t}F(t,x,e^{\beta t}\xi)=G(\xi)+\tilde{f}(t,x,\xi)\geq0$ under the condition that $F_0(x,\xi)=G+\tilde{f}(0,x,\xi)\geq0$, which also indicates the non-negativity of the self-similar solution $G(v)$ obtained in Theorem \ref{st.sol} due to the large time asymptotic behavior \eqref{lif.decay}. To do so, let us start from the following linearized equation of \eqref{bef} in Section \ref{loc.sec}
\begin{eqnarray}\label{ap.bef}
\left\{
\begin{array}{rll}
\begin{split}
&\pa_t f+e^{\beta t}\xi_1\pa_xf-\bet \na_\xi\cdot(\xi f)-\al \xi_2\pa_{\xi_1}f +f\mathcal {V}(f')\\
&\qquad\qquad\qquad\qquad=Q_+(f',f'),\ t>0,\ x\in\T,\ \xi\in \R^3,\\
&f(0,x,\xi)=F_0(x,\xi),\ x\in\T,\ \xi\in \R^3,
\end{split}
\end{array}\right.
\end{eqnarray}
where
$$
\mathcal {V}(f')=\int_{\R^3}\int_{\S^2}B_0(\cos \ta)f'(\xi_\ast)\,d\om d\xi_\ast.
$$
One can see that if $F_0(x,\xi)\geq0$ and $f'(t,x,\xi)\geq0$, then any solution of \eqref{ap.bef} should be non-negative.
Denote $f=G+\tilde{f}$ and $f'=G+\tilde{f}'$, and decompose $\tilde{f}$ and $\tilde{f}'$ as
\begin{align}\notag
\tilde{f}=f_1+\sqrt{\mu}f_2,\ \ \tilde{f}'=f'_1+\sqrt{\mu}f'_2.
\end{align}
We now verify that there exists a unique solution in the form of $G+f_1+\sqrt{\mu}f_2$  to \eqref{ap.bef} under the condition that $[f'_1,f'_2]$ belongs to the function space
\begin{equation*}%\label{Wsp}
\begin{split}
\FW^{\vps_0}_{\al,T}=&\Big\{(\CG_1,\CG_2)\bigg| \sup\limits_{0\leq t\leq T}\left\{\|w_{l}\CG_1(t)\|_{L^\infty}+\al\|w_{l}\CG_2(t)\|_{L^\infty}\right\}
\leq4\vps_0\al,\\
&\qquad\qquad\qquad\CG_1(0)=0,\ \CG_2(0)=\tilde{g}_{0},\ G+\CG_1+\mu^{\frac{1}{2}}\CG_2\geq0
\Big\}.
\end{split}
\end{equation*}
We now consider the coupled equations for $f_1$ and $f_2$
\begin{eqnarray}\label{bef1}
\left\{
\begin{array}{rll}
\begin{split}
&\pa_t f_1+e^{\beta t}\xi_1\pa_xf_1-\bet \na_\xi\cdot(\xi f_1)-\al \xi_2\pa_{\xi_1}f_1
+\frac{\bet}{2}|\xi|^2\mu^{\frac{1}{2}}f_2+\frac{\al}{2}\mu^{\frac{1}{2}}\xi_1\xi_2f_2+\nu_0f_1
\\[2mm]
&\qquad+(f_1+\mu^{\frac{1}{2}}f_2)\CV(\mu^{\frac{1}{2}}(\al G_1+\al G_R))
+(f_1+\mu^{\frac{1}{2}}f_2)\CV(f'_1+\mu^{\frac{1}{2}}f'_2)\\[2mm]
&\quad=\underbrace{\chi_{M}\CK f'_1+\CF_3(f'_1,f'_2)}_{\CF_4},\
t>0,\ x\in\T,\ \xi\in \R^3,\\[2mm]
&f_1(0,x,\xi)=0,\ x\in\T,\ \xi\in \R^3,
\end{split}
\end{array}\right.
\end{eqnarray}
and
\begin{eqnarray}\label{bef2}
\left\{
\begin{array}{rll}
\begin{split}
&\pa_t f_2+e^{\beta t}\xi_1\pa_xf_2-\bet \na_\xi\cdot(\xi f_2)-\al \xi_2\pa_{\xi_1}f_2+\nu_0f_2
\\[2mm]
&\quad=\underbrace{Kf_2'+\mu^{-1/2}(1-\chi_{M})\CK f'_1}_{\CF_5},\ t>0,\ x\in\T,\ \xi\in \R^3,\\[2mm]
&f_2(0,x,\xi)=\frac{F_0(x,\xi)-G(\xi)}{\sqrt{\mu}}\eqdef \tilde{g}_{0}(x,\xi),\ x\in\T,\ \xi\in \R^3,
\end{split}
\end{array}\right.
\end{eqnarray}
where
\begin{align}
\CF_3(f'_1,f'_2)=&
-\mu^{\frac{1}{2}}(\al G_1+\al G_R)\CV(f'_1+\mu^{\frac{1}{2}}f'_2)
+Q_+(f'_1+\mu^{\frac{1}{2}}f'_2,f'_1+\mu^{\frac{1}{2}}f'_2).\notag
\end{align}
Let $[f_1,f_2]$ be a solution of the pair of \eqref{bef1} and \eqref{bef2} with $[f'_1,f'_2]\in\FW^{\vps_0}_{\al,T}$.
Then the nonlinear operator $\CW$ is formally defined as
$$
\CW([f'_1,f'_2])=[\CW_1,\CW_2]([f'_1,f'_2])=[f_1,f_2].
$$
We next show that $\CW$ is a contraction mapping on $\FW^{\vps_0}_{\al,T}$.
To do this, let us first rewrite the solution of \eqref{bef1} and \eqref{bef2} as
\begin{align}\label{f12.sol}
w_l[f_1,\al f_2]=&\dis e^{-\int_{0}^t\CM(\tau,V(\tau))\,d\tau}[0,\al w_{l}\tilde{g}_{0}(X(0),V(0))]
\notag\\&+\int_0^{t}e^{-\int_{s}^t\CM(\tau,V(\tau))\,d\tau}w_l[\CF_4, \al\CF_5]\,ds,
\end{align}
where $\CM$ is a $2\times 2$  matrix given by
\begin{eqnarray*}%\label{def.matrM}
\begin{aligned} &\left[\begin{array} {cccc}
M_{11} \  \  \  & M_{12} \\[3mm]
0 \  \  \ & \al M_{22}
\end{array} \right],
\end{aligned}
\end{eqnarray*}
with
\begin{align}
M_{11}=&\nu_0-3\beta+2l \bet \frac{|V(\tau)|^2}{1+|V(\tau)|^2} +2l \al\frac{V_2(\tau)V_1(\tau)}{{1+|V(\tau)|^2}}
\notag\\&+\CV(\mu^{\frac{1}{2}}(\al G_1+\al G_R))
+\CV(f'_1+\mu^{\frac{1}{2}}f'_2)\geq\nu_0/2,\label{M11}
\\M_{12}=&\frac{\beta|V(\tau)|^2}{2}\mu^{\frac{1}{2}}+\frac{\al}{2}\mu^{\frac{1}{2}}V_2(\tau)V_1(\tau)
+\mu^{\frac{1}{2}}\CV(\mu^{\frac{1}{2}}(\al G_1+\al G_R))+
\mu^{\frac{1}{2}}\CV(f'_1+\mu^{\frac{1}{2}}f'_2),
\notag\\M_{22}=&\nu_0-3\beta+2l \bet \frac{|V(\tau)|^2}{1+|V(\tau)|^2}+2l \al \frac{V_2(\tau)V_1(\tau)}{{1+|V(\tau)|^2}}\geq\nu_0/2,\label{M22}
\end{align}
and moreover, $X(s)$ and $V(s)$ is defined as \eqref{chlp2}.

Given $[f'_1,f'_2]\in\FW^{\vps_0}_{\al,T}$, we now show that so $[f_1,f_2]$ is.
Since $[f'_1,f'_2]\in\FW^{\vps_0}_{\al,T}$, one sees that
$|M_{12}|\leq C(\al+\vps_0)$. With this, \eqref{M11} and \eqref{M22}, we have by utilizing Lemma \ref{op.es.lem} that
\begin{align}
|w_l[f_1,\al f_2]|\leq& e^{CT(\al+\vps_0)}\vps_0\al
+Te^{CT(\al+\vps_0)}\|w_l[\CF_4,\al \CF_5]\|_{L^\infty}
\notag\\ \leq& e^{CT(\al+\vps_0)}\vps_0\al+ Te^{CT(\al+\vps_0)}\left\{\|w_lf_1'\|_{L^\infty}+\|w_lf_1'\|^2_{L^\infty}+\al\|w_lf_2'\|_{L^\infty}+\|w_lf_2'\|^2_{L^\infty}\right\},\notag
\end{align}
which is further bounded by $4\vps\al$,
provided that $0<T\leq T_{\ast\ast}$ with $T_{\ast\ast}$ sufficiently small. Thus $\CW([f_1',f_2'])\in\FW^{\vps_0}_{\al,T_{\ast\ast}}.$ Note that
$G+f_1+\mu^{\frac{1}{2}}f_2\geq0$ follows from \eqref{ap.bef} and $f'\geq0.$ It remains now to verify that $\CW$ is a contraction.
In fact, given $[f_1',f_2']$, $[f_1'',f_2'']\in\FW^{\vps_0}_{\al,T_{\ast\ast}},$
from \eqref{f12.sol}, it follows
\begin{align*}
|w_l\CW_1([f_1',f_2'])&-w_l\CW_1([f_1'',f_2''])|+\al|w_l\CW_2([f_1',f_2'])-w_l\CW_2([f_1'',f_2''])|
\notag\\&\leq
\int_0^{t}e^{-\int_{s}^t\CM(\tau,V(\tau))d\tau}
w_l\left|[\CF_4,\al \CF_5](f_1',f_2')-[\CF_4,\al \CF_5](f_1''f_2'')\right|ds
\notag\\&\leq T_{\ast\ast}e^{C(\al+\vps_0)T_{\ast\ast}}\left\{\|w_l(f'_1-f''_1)\|_{L^\infty}
+\al\|w_l(f'_2-f''_2)\|_{L^\infty}+\|w_l(f'_2-f''_2)\|^2_{L^\infty}\right\}
\notag\\&\leq \frac{1}{2}\left\{\|w_l(f'_1-f''_1)\|_{L^\infty}
+\al\|w_l(f'_2-f''_2)\|_{L^\infty}\right\},
\end{align*}
provided that $T_{\ast\ast}>0$ is sufficiently small. Therefore, there exists a unique function $[f_1,f_2]\in\FW^{\vps_0}_{\al,T_{\ast\ast}}$
such that $[f_1,f_2]=\CW([f_1,f_2])$, namely \eqref{bef} admits a unique non-negative solution $f=G+f_1+\mu^{\frac{1}{2}}f_2$ with $[f_1,f_2]\in\FW^{\vps_0}_{\al,T_{\ast\ast}}$
for $T_{\ast\ast}>0$ small enough. Finally, since we have obtained the uniform bound \eqref{apes} in the previous steps of this section, one can extend the existing time interval of the above non-negative solution to an arbitrary time $t>0$. Thus, the proof of Theorem \ref{ge.th}
is completed.
\end{proof}

\noindent {\bf Acknowledgements:}
RJD was partially supported by the General Research Fund (Project No.~14302817) from RGC of Hong Kong and a Direct Grant from CUHK. SQL was supported by grants from the National Natural Science Foundation of China (contracts: 11731008 and 11971201). SQL would like to thank Department of Mathematics, CUHK for their hospitality during his visit in January-March in 2020. RJD would thank Florian Theil for introdcuing to him the problem in 2015, and also thank Alexander Bobylev for stimulating discussions on \cite{BNV-2019} during the conference ``Advances in Kinetic Theory" hosted by Chongqing University in October 2019.

%\newpage

\end{document}